\documentclass[reqno,a4paper]{amsart}
\usepackage[english]{babel}

\usepackage{graphicx,mathrsfs,tikz,tikz-cd,latexsym,ifthen,amsmath,amsfonts,amssymb,amsthm,stmaryrd,fancyhdr,amscd,amsbsy,amstext,tensor}
\usepackage{enumerate,enumitem,empheq,mathtools,mathpazo,verbatim,a4wide} 
\mathtoolsset{showonlyrefs,showmanualtags}

\usepackage{hyperref}
\hypersetup{
	anchorcolor=black, 
	linktocpage=true,
	colorlinks=true,
	citecolor=blue,
	linkcolor=blue,
	urlcolor=black,
	pdfauthor={Andrea Appel, Francesco Sala},
	pdftitle={}, 
	breaklinks=true,
	plainpages=true
}

\usepackage{tikz} 
\usetikzlibrary{shapes, backgrounds} 
\usepackage[all,cmtip,knot]{xy}
\xyoption{arc}

\usepackage{epigraph}

\usepackage[utf8x]{inputenc}

\usepackage[toc,page]{appendix}
\usepackage[mathscr]{euscript} 
\allowdisplaybreaks

\newif\ifpersonal
\personaltrue 
\ifpersonal
\newcommand{\personal}[1]{\textcolor[rgb]{0,0,1}{(Personal: #1)}}
\newcommand{\todo}[1]{\textcolor{red}{(Todo: #1)}}
\else
\newcommand*{\personal}[1]{\ignorespaces}
\newcommand*{\todo}[1]{\ignorespaces}
\fi

\usetikzlibrary{positioning,shapes,shadows,arrows}

\usepackage{color}
\definecolor{MyDarkBlue}{rgb}{0.15,0.25,0.45}
\newcommand{\Omit}[1]{}

\newcommand{\calB}{{\mathcal B}}

\newcommand{\calD}{{\mathcal D}}

\newcommand{\calH}{{\mathcal H}}
\newcommand{\calI}{{\mathcal I}}
\newcommand{\calJ}{{\mathcal J}}

\newcommand{\calQ}{{\mathcal Q}}

\newcommand{\K}{{\mathbb{K}}}
\newcommand{\R}{{\mathbb{R}}}
\newcommand{\N}{{\mathbb{N}}}
\newcommand{\C}{{\mathbb{C}}}
\newcommand{\Z}{{\mathbb{Z}}}
\newcommand{\Q}{{\mathbb{Q}}}

\newcommand{\scrR}{{\mathscr R}}

\newcommand{\sfA}{{\mathsf A}}

\newcommand{\sfQ}{{\mathsf Q}}
\newcommand{\sfR}{{\mathsf R}}

\newcommand{\sfF}{{\mathsf{F}}}

\newcommand {\bfI}{\mathbf I}
\newcommand{\bfU}{\mathbf U}
\newcommand{\bfV}{\mathbf{V}}

\renewcommand {\a}{\mathfrak a}
\renewcommand {\b}{\mathfrak b}
\renewcommand {\c}{\mathfrak c}
\renewcommand {\d}{\mathfrak d}

\newcommand {\g}{\mathfrak{g}}
\newcommand {\h}{\mathfrak h}

\newcommand {\n}{\mathfrak n}

\renewcommand {\r}{\mathfrak r}

\renewcommand{\sl}{\mathfrak{sl}}
\newcommand{\heis}{\mathfrak{heis}}

\newcommand{\bsfldh}{\mathbf{K}}
\newcommand{\vectK}{\operatorname{Vect}_{\bsfldh}}

\newcommand{\FL}[1]{\mathcal{L}_{#1}}
\newcommand{\xzpm}[1]{\xi_{#1}^{\pm}}
\newcommand{\xzmp}[1]{\xi_{#1}^{\mp}}
\newcommand{\xzp}[1]{\xi_{#1}^{+}}
\newcommand{\xzm}[1]{\xi_{#1}^{-}}
\newcommand{\hcorpm}[1]{\hcor{#1}^{\pm}}

\newcommand{\id}{{\mathsf{id}}}

\newcommand{\colim}{\mathsf{colim}}
\newcommand{\rank}{\mathsf{rk}}

\newcommand{\End}{\mathsf{End}}

\newcommand {\ol}{\overline}
\newcommand {\wt}{\widetilde}
\newcommand {\wh}{\widehat}

\newcommand{\scsop}{\scriptscriptstyle\operatorname} 

\newcommand {\ie}{{\emph{i.e.}}, }

\newcommand{\serre}[1]{\mathsf{Serre}(#1)}

\newcommand{\drc}[1]{\delta_{#1}}

\newcommand{\sfad}{\mathsf{ad}}
\newcommand{\ten}{\otimes}

\newcommand{\bsfld}{{\mathbf k}} 
\newcommand{\gcm}{\sfA} 
\newcommand{\sgpp}{\oplus} 
\newcommand{\sgpm}{\ominus} 

\newcommand{\cf}[1]{{\mathbb 1}_{#1}} 
\newcommand{\fun}[1]{\mathfrak{f}_{#1}} 
\newcommand{\intsf}{\mathsf{Int}} 
\newcommand{\abf}[2]{\langle #1, #2\rangle} 
\newcommand{\rbf}[2]{\left( #1 | #2 \right)} 
\newcommand{\abfcf}[2]{\langle \cf{#1}, \cf{#2}\rangle} 
\newcommand{\rbfcf}[2]{\left( \cf{#1} | \cf{#2}\right)} 
\newcommand{\intnext}{\to} 
\newcommand{\rsub}{\vdash} 
\newcommand{\lsub}{\dashv} 
\newcommand{\intcap}{\pitchfork} 
\newcommand{\rtl}{\sfQ} 
\newcommand{\rts}{\sfR} 
\newcommand{\hcor}[1]{h_{#1}} 
\newcommand{\cow}[1]{\lambda^{\vee}_{#1}} 
\newcommand{\rls}[1]{\scrR} 
\newcommand{\mrls}[1]{\scrR_{\scsop{min}}} 
\newcommand{\crls}[1]{\ol{\scrR}} 

\newcommand{\sfD}{\mathsf{D}}

\newcommand{\bfJ}{\mathbf{J} }
\newcommand{\ul}{\underline}

\newcommand{\xg}[2]{x^{#1}_{#2}}

\newcommand{\xz}[1]{\xi_{#1}}
\newcommand{\xp}[1]{\xg{+}{#1}}
\newcommand{\xm}[1]{\xg{-}{#1}}
\newcommand{\xpm}[1]{\xg{\pm}{#1}}
\newcommand{\xmp}[1]{\xg{\mp}{#1}}

\newcommand{\iip}[2]{\rbf{#1}{#2}}

\newcommand{\gb}{\g_{\b}}
\newcommand{\grb}{\g^{\scsop{res}}_{\b}}

\newcommand{\gtwo}{\g^{(2)}}
\newcommand{\btwo}{\b^{(2)}}
\newcommand{\zh}{\h^{\scsop{z}}}

\newcommand{\cond}{\vartriangleright}
\newcommand{\donc}{\vartriangleleft}
\newcommand{\dcs}{\triangleright\negthinspace\negthinspace\triangleleft}

\newcommand{\dtwo}{\d^{(2)}}

\renewcommand{\DJ}[1]{\bfU_q#1}
\newcommand{\qxg}[2]{X^{#1}_{#2}}

\newcommand{\qxz}[2]{K_{#1}^{#2}}
\newcommand{\qxp}[1]{\qxg{+}{#1}}
\newcommand{\qxm}[1]{\qxg{-}{#1}}
\newcommand{\qxpm}[1]{\qxg{\pm}{#1}}

\newcommand{\D}{\calD}
\newcommand{\iM}[2]{{#1}\triangledown{#2}}
\newcommand{\im}[2]{{#1}\negthinspace\vartriangle\negthinspace{#2}}

\newcommand{\hext}[1]{{#1}[\negthinspace[\hbar]\negthinspace]}

\newcommand{\ia}{\alpha}
\newcommand{\ib}{\beta}
\newcommand{\ic}{\gamma}

\newcommand{\ca}[2]{\mathsf{a}_{#1#2}}
\newcommand{\cb}[2]{\mathsf{b}_{#1#2}}

\newcommand{\qcb}[3]{\mathsf{b}^{#3}_{#1#2}}
\newcommand{\qcc}[3]{\mathsf{c}^{#3}_{#1#2}}

\newcommand{\qcr}[3]{\mathsf{r}^{#3}_{#1#2}}
\newcommand{\qcs}[3]{\mathsf{s}^{#3}_{#1#2}}

\newcommand{\half}{\frac{1}{2}}

\newcommand{\triend}{\parbox{2mm}{\hfill} \hfill\mbox{\hspace{0.2mm}}\hfill$\triangle$}
\newcommand{\ocend}{\parbox{2mm}{\hfill} \hfill\mbox{\hspace{0.2mm}}\hfill$\oslash$}

\newtheorem{theorem}{Theorem}
\newtheorem{proposition}[theorem]{Proposition}

\newtheorem{corollary}[theorem]{Corollary}

\newtheorem{def-thm}{Definition--Theorem}
\newtheorem{corollary*}{Corollary}
\newtheorem*{theorem*}{Theorem}
\newtheorem*{definition*}{Definition}
\newtheorem*{def-thm*}{Definition--Theorem}
\newtheorem*{proposition*}{Proposition}
\newtheorem*{conjecture*}{Conjecture}

\numberwithin{equation}{section}
\numberwithin{theorem}{section}

\theoremstyle{remark}
\newtheorem{ex}[theorem]{Example}
\newenvironment{example}{\begin{ex}}{\triend\end{ex}}

\theoremstyle{remark}
\newtheorem{rem}[theorem]{Remark}
\newenvironment{remark}{\begin{rem}}{\triend\end{rem}}

\theoremstyle{definition}
\newtheorem{defin}[theorem]{Definition}
\newtheorem*{defin*}{Definition}
\newenvironment{definition}{\begin{defin}}{\ocend\end{defin}}
\newenvironment{nndefinition}{\begin{defin*}}{\ocend\end{defin*}}

\title[Quantization of continuum Kac--Moody algebras]{Quantization of continuum Kac--Moody algebras}

\author[A.~Appel]{Andrea Appel}
\address[Andrea Appel]{Dipartimento di Scienze Matematiche, Fisiche e Informatiche,
Universit\`a di Parma, Italy}
\curraddr{}
\email{\href{mailto:andrea.appel@unipr.it}{andrea.appel@unipr.it}}

\author[F.~Sala]{Francesco Sala}
\address[Francesco Sala]{Dipartimento di Matematica, Universit\`a di Pisa, Italy}
\curraddr{}
\email{\href{mailto:francesco.sala@unipi.it}{francesco.sala@unipi.it}}

\thanks{The work of the first--named author is supported by the ERC Grant 637618.
	The work of the second-named author is partially supported by World Premier International 
	Research Center Initiative (WPI), MEXT, Japan, by JSPS KAKENHI Grant number JP17H06598 and 
	by JSPS KAKENHI Grant number JP18K13402.}


\begin{document}

\begin{abstract}
	Continuum Kac--Moody algebras have been recently introduced by 
	the authors and O. Schiffmann in \cite{appel-sala-schiffmann-18}. 
	These are Lie algebras governed by a {\em continuum} root system, which
	can be realized as uncountable colimits of Borcherds--Kac--Moody 
	algebras. In this paper, we prove that any continuum Kac--Moody algebra 
	$\g$ is canonically endowed with a non--degenerate invariant bilinear form. 
	The positive and negative Borel subalgebras form a Manin triple with respect to 
	this pairing, which allows to define on $\g$ a topological quasi--triangular Lie 
	bialgebra structure. We then construct an explicit quantization of $\g$, which we 
	refer to as a \emph{continuum quantum group}, and we show that the latter is 
	similarly realized as an uncountable colimit of Drinfeld--Jimbo quantum groups.
\end{abstract}

\maketitle\thispagestyle{empty}

\begin{center}
	\emph{Dedicated to Prof. Kyoji Saito on the occasion of his 75th birthday.}
\end{center}

\vskip 1cm

\tableofcontents
\addtocontents{toc}{\protect\setcounter{tocdepth}{1}} 

\section{Introduction}
Continuum Kac--Moody algebras have been recently introduced by the authors and O. Schiffmann
in \cite{appel-sala-schiffmann-18}. Their definition is similar to that of a Kac--Moody algebra.
However, they are governed by a \emph{continuum} root system, arising from the combinatorics of 
connected intervals in a one--dimensional topological space. They are not Kac--Moody 
algebras themselves, but they can be realized as uncountable colimits of symmetric 
Borcherds--Kac--Moody algebras\footnote{Specifically, we allow the diagonal entries of the Cartan
	matrix to be zero.}.

In this paper, we provide a gentle introduction to this new theory,
avoiding the technicalities of \cite{appel-sala-schiffmann-18}, and we push 
further the study of these Lie algebras, providing two main contributions. 
First, we prove that continuum Kac--Moody algebras have a canonical structure
of (topological) Lie bialgebras, which arises, as in the classical Kac--Moody case, 
from the construction of a non--degenerate invariant symmetric bilinear form.
Then, we construct an explicit algebraic quantization of these topological structures,
which we call \emph{continuum quantum group}: they can be similarly realized as
uncountable colimits of Drinfeld--Jimbo quantum groups. Moreover, we prove that, 
in the simplest cases of the line and the circle, they coincide with the 
quantum groups constructed \emph{with geometric methods} in \cite{sala-schiffmann-17} by the second--named 
author and O. Schiffmann in terms of Hall algebras. 
In the forthcoming work \cite{appel-kuwagaki-sala-schiffmann-18}, we shall adopt 
a similar approach to show that continuum quantum groups admit analogous geometric
realizations arising from Hall algebras. 

In the remaining part of this introduction, we shall explain our work in greater detail.

\subsection*{The continuum Kac--Moody algebra}

The defining datum of a continuum Kac--Moody algebra is a continuum analogue 
of a quiver, defined as follows. Recall that the latter is just an oriented graph 
$\calQ=(\calQ_0, \calQ_1)$ with set of vertices $\calQ_0$ and a set of edges $\calQ_1$.
In a \emph{continuum} quiver, the discrete set $\calQ_0$ is replaced by a \emph{vertex space} 
$X$, which is, roughly, a Hausdorff topological space locally modeled over $\R$ (cf.\ Definition~\ref{def:topological-quiver}).
Examples of vertex spaces are the line $\R$, the circle $S^1=\R/\Z$, smoothings of possibly infinite trees, or 
combinations of these. Thus, it is possible to lift the notion of \emph{connected interval} from $\R$ to $X$, in 
such a way that the set of all possible \emph{intervals in $X$}, denoted $\intsf(X)$, is naturally 
endowed with two \emph{partially defined} operations, that is, a sum $\sgpp$, given by concatenation 
of intervals, and a difference $\sgpm$, given by set difference whenever the outcome is again in $\intsf(X)$. 

The set $\intsf(X)$ comes naturally equipped with a set--theoretic non--degenerate pairing 
$\rbf{\cdot}{\cdot}\colon$ $\intsf(X)\times\intsf(X)\to\Z$, defined as follows. For a locally 
constant, compactly supported, left--continuous function $h:\R\to\R$, we set 
$h_{\pm}(x)\coloneqq\lim_{\epsilon\to0^+}h(x\pm\epsilon)$ and define a 
non--symmetric bilinear form given by
\begin{align*}
	\abf{f}{g}\coloneqq \sum_{x} f_-(x)(g_-(x)-g_+(x)) 
\end{align*}
Identifying an interval $\ia$ with its characteristic function $\cf{\ia}$, we obtain a bilinear
form on $\intsf(\R)$. Then, we lift it from $\R$ to $X$ by decomposing every 
interval in $X$ into an iterated concatenation of \emph{elementary} intervals in $\R$. 
Finally, the \emph{Euler form} on $\intsf(X)$ is given by
$\rbf{\ia}{\ib}\coloneqq\abf{\ia}{\ib}+\abf{\ib}{\ia}$. We refer to the
datum $\calQ_X\coloneqq(\intsf(X), \sgpp,\sgpm,\abf{\cdot}{\cdot}, \rbf{\cdot}{\cdot})$ 
as the \emph{continuum quiver} of the vertex space $X$. Henceforth,
we shall denote by $\fun{X}$ the span of the characteristic functions $\cf{\ia}$, $\ia\in\intsf(X)$.

Given a continuum quiver $\calQ_X$, together with O. Schiffmann, we construct in \cite{appel-sala-schiffmann-18} 
a Lie algebra $\g_X $, which we refer to as the \emph{continuum Kac--Moody algebra of $\calQ_X$}, whose Cartan subalgebra is 
generated by the characteristic functions of the intervals of $X$. The definition of $\g_X $ mimics the usual construction 
of Kac--Moody algebras, with some fundamental differences controlled by the partial operations of $\calQ_X$. Namely, we 
first consider the Lie algebra $\wt{\g}_X$ over $\C$, freely generated by $\fun{X}$ and the elements $\xpm{\ia}$, $\ia\in\intsf(X)$, 
subject to the relations:
\begin{align*}
	\begin{array}{ll}
		[\xz{\ia},\xz{\ib} ]&=0\ ,\\
		{}[\xz{\ia} ,\xpm{\ib} ]&=\pm\rbf{\ia}{\ib}\cdot \xpm{\ib}\ ,\\
		{}[\xp{\ia} ,\xm{\ib} ]&=\drc{\ia\ib}\xz{\ia} +
		\ca{\ia}{\ib}\cdot (\xp{\ia\sgpm \ib}-\xm{\ib \sgpm \ia})\ ,
	\end{array}
\end{align*}
where $\xz{\ia} \coloneqq\cf{\ia}$ and $\ca{\ia}{\ib}\coloneqq(-1)^{\abf{{\ia} }{{\ib} }}\cdot\rbf{\ia}{\ib}$.
Then, we set $\g_X \coloneqq\wt{\g}_X/\r_X$, where $\r_X\subset\wt{\g}_X$ is the sum of all two--sided graded
\footnote{The gradation is with respect to $\fun{X}$: we set $\deg(\xpm{\ia})=\pm\cf{\ia}$ and $\deg(\xz{\ia})=0$.} 
ideals having trivial intersection with $\fun{X}$.

In \cite{appel-sala-schiffmann-18}, we show that the ideal $\r_X$ is generated by certain quadratic Serre relations 
governed by the concatenation of intervals, thus generalizing Gabber--Kac theorem for continuum Kac--Moody algebras
(cf.\ \cite{gabber-kac-81}) and obtaining an explicit description of $\g_X $ (cf.\ \cite[Thm.~5.17]{appel-sala-schiffmann-18} 
or Theorem~\ref{thm:ass-18} below). 
That is, $\g_X $ is generated by the abelian Lie algebra 
$\fun{X}$ and the elements $\xpm{\ia} $, $\ia\in \intsf(X)$, subject to the following 
defining relations:
\begin{enumerate}\itemsep0.2cm
	\item[] {\bf Diagonal action:} for $\ia, \ib\in\intsf(X)$, 
	\begin{align*}
		[\xz{\ia} , \xpm{\ib} ] =\pm \rbf{\ia} {\ib} \cdot\, \xpm{\ib} \ ;
	\end{align*}
	\item[] {\bf Double relations:} for $\ia, \ib\in\intsf(X)$, 
	\begin{align*}
		[\xp{\ia} ,\xm{\ib} ]=\drc{\ia\ib}\, \xz{\ia} +\ca{\ia}{\ib}\cdot\left(\xp{\ia\sgpm\ib}-\xm{\ib\sgpm\ia}\right) \ ;
	\end{align*} 
	\item[] {\bf Serre relations:} for $(\ia,\ib)\in \serre{X}$,
	\begin{align*}
		[\xpm{\ia} , \xpm{\ib} ]=  \pm\ca{\ia}{,\, \ia\sgpp\ib}\cdot \xpm{\ia\sgpp\ib} \ .
	\end{align*}
\end{enumerate}
Here, $\serre{X}$ is the set of all pairs $(\ia,\ib)\in\intsf(X)\times\intsf(X)$ such that one of the
following occurs:
\begin{itemize}\itemsep0.3cm
	\item
	$\ia$ is contractible, does not contain any \emph{critical} point of $\ib$ 
	(cf. Definition~\ref{def:topological-quiver}) 
	and, for subintervals $\ia'\subseteq \ia$ and $\ib'\subseteq\ib$ 
	with $\rbf{\ib}{\ib'}\neq0$ whenever $\ib'\neq\ib$, $\ia'\sgpp\ib'$ is either undefined or 
	non--homeomorphic to $S^1$;
	\item $\ia\perp\ib$, \ie $\ia\sgpp\ib$ does not exist and $\ia\cap\ib=\emptyset$.
\end{itemize}

As mentioned earlier, $\g_X $ can be equivalently realized as certain continuous 
colimits of Borcherds--Kac--Moody algebras, further motivating our choice of the 
terminology. This is based on the following observation. Let $\calJ=\{\ia_k\}_{k}$ 
be an \emph{irreducible} finite set of intervals $\ia_k\in\intsf(X)$, \ie
\begin{enumerate}\itemsep0.2cm
	\item every interval is either contractible or homeomorphic to $S^1$; 
	\item given two intervals $\ia,\ib\in\calJ$, $\ia\neq\ib$, one of the following mutually exclusive cases occurs:
	\begin{itemize}\itemsep0.2cm
		\item[(a)] $\ia\sgpp\ib$ exists;
		\item[(b)] $\ia\sgpp\ib$ does not exist and $\ia\cap\ib=\emptyset$;
		\item[(c)] $\ia\simeq S^1$ and $\ib\subset\ia$.
	\end{itemize}
\end{enumerate}
Let $\sfA_{\calJ}$ be the matrix given by the values of $\rbf{\cdot}{\cdot}$ on $\calJ$, 
\ie $\big(\sfA_{\calJ}\big)_{\ia\ib}=\rbf{\ia}{\ib}$ for $\ia,\ib\in \calJ$. Note that the diagonal
entries of $\sfA_{\calJ}$ are either $2$ or $0$, while the only possible
off--diagonal entries are $0,-1,-2$. Let $\calQ_{\calJ}$ be the corresponding quiver with Cartan matrix $\sfA_{\calJ}$. 
For example, we obtain the following quivers.

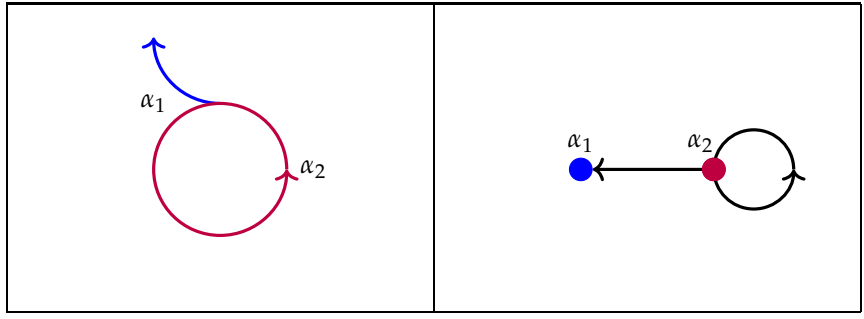
\begin{align*}
	\begin{array}{|c|c|}
		\hline
		\text{Configuration of intervals} & \text{Borcherds--Cartan diagram}\\
		\hline &\\
		\begin{tikzpicture}[scale=.35]
			\begin{scope}[on background layer]
				\draw [white] (0,0) rectangle (15,10);
			\end{scope}
			\draw [->, very thick, blue] (3,5) -- (6,5); 
			\draw [->, very thick, purple] (6,5) -- (9,5); 
			\draw [->, very thick, yellow] (9,5) -- (12,5);
			\node at (4.5, 6) {$\alpha_1$};
			\node at (7.5, 6) {$\alpha_2$};
			\node at (10.5, 6) {$\alpha_3$};
		\end{tikzpicture}
		&
		\begin{tikzpicture}[scale=.35]
			\begin{scope}[on background layer]
				\draw [white] (0,0) rectangle (15,10);
			\end{scope}
			\node (V1) at (3.5,5)      [circle,draw=blue,fill=blue, inner sep=3pt]          {};
			\node (V2) at (7.5,5)      [circle,draw=purple,fill=purple, inner sep=3pt]    {};
			\node (V3) at (11.5,5) [circle,draw=yellow,fill=yellow, inner sep=3pt]   {};
			\draw [->, very thick] (V1) -- (V2);  
			\draw [<-, very thick] (V3) -- (V2);
			\node at (3.5, 6)    {$\alpha_1$};
			\node at (7.5, 6)    {$\alpha_2$};
			\node at (11.5, 6)  {$\alpha_3$};
		\end{tikzpicture}
		\\
		\hline &\\
		\begin{tikzpicture}[scale=.35]
			\begin{scope}[on background layer]
				\draw [white] (0,0) rectangle (15,10);
			\end{scope}
			\draw [->, very thick, purple] (7.5,2) -- (7.5,5); 
			\draw [->, very thick, blue] (7.5,5) arc (0:90:2.5); 
			\draw [->, very thick, yellow] (7.5,5) arc (180:90:2.5);
			\node at (6, 3.5) {$\alpha_2$};
			\node at (6, 8.5) {$\alpha_1$};
			\node at (9, 8.5) {$\alpha_3$};
		\end{tikzpicture}
		&
		\begin{tikzpicture}[scale=.35]
			\begin{scope}[on background layer]
				\draw [white] (0,0) rectangle (15,10);
			\end{scope}
			\node (V1) at (3.5,5)      [circle,draw=blue,fill=blue, inner sep=3pt]          {};
			\node (V2) at (7.5,5)      [circle,draw=purple,fill=purple, inner sep=3pt]    {};
			\node (V3) at (11.5,5) [circle,draw=yellow,fill=yellow, inner sep=3pt]   {};
			\draw [->, very thick] (V2) -- (V1);  
			\draw [<-, very thick] (V3) -- (V2);
			\node at (3.5, 6)    {$\alpha_1$};
			\node at (7.5, 6)    {$\alpha_2$};
			\node at (11.5, 6)  {$\alpha_3$};
		\end{tikzpicture}
		\\
		\hline
	\end{array}
\end{align*}
\begin{align*}
	\begin{array}{|c|c|}
		\hline &\\
		\begin{tikzpicture}[scale=.35]
			\begin{scope}[on background layer]
				\draw [white] (0,0) rectangle (15,10);
			\end{scope}
			\draw [<-, very thick, blue] (5,10) arc (180:270:2.5);
			\draw [->, very thick, purple] (10,5) arc (0:360:2.5);
			\node at (5, 7.5)    {$\alpha_1$};
			\node at (11, 5)    {$\alpha_2$};
		\end{tikzpicture}
		&
		\begin{tikzpicture}[scale=.35]
			\begin{scope}[on background layer]
				\draw [white] (0,0) rectangle (15,10);
			\end{scope}
			\node (V1) at (5, 5)  [circle,draw=blue,fill=blue, inner sep=3pt]    {};
			\node (V3) at (10, 5)  [circle,draw=purple,fill=purple, inner sep=3pt]    {};
			\draw [->, very thick] (V3) -- (V1);
			\draw [->, very thick] (13,5) arc (0:360:1.5);  
			\node at (10, 5)  [circle,draw=purple,fill=purple, inner sep=3pt]    {};
			\node at (5, 6)    {$\alpha_1$};
			\node at (9.5, 6)    {$\alpha_2$};
		\end{tikzpicture}
		\\
		\hline
	\end{array}\\
\end{align*}
Note, in particular, that any contractible elementary interval corresponds to a vertex of 
$\calQ_{\calJ}$ without loops, while any interval homeomorphic to $S^1$, corresponds 
to a vertex having exactly one loop. 

There are two Lie algebras naturally associated to $\calJ$:
\begin{enumerate}\itemsep0.2cm
	\item the Lie subalgebra $\g_{\calJ}\subset\g_X $ generated by the elements 
	$\{\xpm{\ia},\, \xz{\ia}\;\vert\; \ia\in \calJ\}$;
	\item the derived Borcherds--Kac--Moody algebra $\g_{\calJ}^{\scsop{BKM}}\coloneqq
	\g(\sfA_{\calJ})'$.
\end{enumerate}
In \cite[Section~5.5]{appel-sala-schiffmann-18}, we show that $\g_{\calJ}$ and 
$\g_{\calJ}^{\scsop{BKM}}$ are canonically isomorphic. In particular, $\g_X$ can be 
\emph{covered} by Borcherds--Kac--Moody algebras. Moreover, we show that,
given two \emph{compatible} irreducible sets $\calJ,\calJ'$,
there is an obvious embedding $\phi_{\calJ',\calJ}\colon \g_{\calJ}\to\g_{\calJ'}$,
and the collection of all such $\phi$'s is a direct system, so that we get a 
canonical isomorphism of Lie algebras  $\g_X\simeq\colim_{\calJ}\, \g_{\calJ}^{\scsop{BKM}}$
(cf.\ \cite[Cor. 5.18]{appel-sala-schiffmann-18} or Corollary~\ref{cor:bkm-sgp} below).

\subsection*{Continuum Lie bialgebras}

It is well--known that any symmetrisable Borcherds--Kac--Moody algebra $\g$ is endowed with 
a symmetric non--degenerate bilinear form, inducing an isomorphism of graded vector
spaces $\b_+\simeq\b_-^\star$ between the positive and negative Borel subalgebras,
and consequently defining a Lie bialgebra structure on $\g$. Moreover, the latter is 
quasi--triangular with respect to the canonical element $r\in\b_+\wh{\ten}\b_-$ 
corresponding to the perfect pairing $\b_+\ten\b_-\to\C$ (cf.\ Section~\ref{s:km-dj}).

The first contribution of this paper is the extension of these results for 
continuum Kac--Moody algebras. 

\begin{theorem*}[{cf. Theorem~\ref{thm:cont-km-lba}}]
	Let $\calQ_X$ be a continuum quiver and $\g_X$ the corresponding 
	continuum Kac--Moody algebras. 
	\begin{enumerate}[leftmargin=2em]\itemsep0.2cm
		\item The Euler form on $\fun{X}$ uniquely extends to 
		an invariant symmetric bilinear form $\rbf{\cdot}{\cdot}\colon\wt{\g}_X \ten\wt{\g}_X \to\C$
		defined on the generators as follows:
		\begin{align*}
			\rbf{\xz{\ia}}{\xz{\ib}}\coloneqq \rbf{\ia}{\ib},\quad 
			\rbf{\xpm{\ia}}{\xz{\ib}}\coloneqq 0,\quad 
			\rbf{\xpm{\ia}}{\xpm{\ib}}\coloneqq 0,\quad
			\rbf{\xp{\ia}}{\xm{\ib}}\coloneqq \drc{\ia\ib}.
		\end{align*}
		Moreover, $\ker\rbf{\cdot}{\cdot}=\r_X$ and therefore the Euler form descends
		to a non--degenerate invariant symmetric bilinear form on $\g_X$.
		\item There is a unique topological cobracket $\delta\colon \g_X \to\g_X \wh{\ten}\g_X $ 
		defined on the generators by
		\begin{align*}
			\delta(\xz{\ia})\coloneqq 0\quad\mbox{and}\quad
			\delta(\xpm{\ia} )\coloneqq \xzpm{\ia} \wedge\xpm{\ia}+\sum_{\ib\sgpp \ic=\ia} \ca{\ib}{,\, \ib\sgpp\ic}\cdot
			\xpm{\ib} \wedge\xpm{\ic} \ ,
		\end{align*}
		and inducing on $\g_X$ a topological Lie bialgebra structure, with respect to which
		the positive and negative Borel subalgebras $\b_X^{\pm}$ are Lie sub-bialgebras.
		\item The Euler form restricts to a non--degenerate pairing of Lie bialgebras 
		$\rbf{\cdot}{\cdot}\colon \b_X^+\ten(\b_X^-)^{\scsop{cop}}\to\C$. Then, the canonical 
		element $r_X\in\b_X^+\wh{\ten}\b_X^-$ corresponding to $\rbf{\cdot}{\cdot}$
		defines a quasi--triangular structure on $\g_X$.  
	\end{enumerate}
\end{theorem*}

Note however that in order to prove this result one cannot rely on the colimit realization 
of $\g_X$ given above, since the embeddings $\phi_{\calJ',\calJ}\colon \g_{\calJ}^{\scsop{BKM}}
\to\g_{\calJ'}^{\scsop{BKM}}$, do not respect the cobracket, as clear from their definition (cf.\ Corollary~\ref{cor:bkm-sgp}). 
Instead, our proof is based on an alternative realization of $\g_X$ \emph{by duality}, inspired by
the work of G. Halbout \cite{halbout-99} which relies on a semi--classical version of techniques
coming from the foundational theory of quantum groups \cite{drinfeld-quantum-groups-87,lusztig-book-94}.

By the result above, we can now associate to any continuum quiver $\calQ_X$ a topological
quasi--triangular Lie bialgebra $(\g_X,[\cdot,\cdot],\delta)$. The second and main contribution of
this paper is the algebraic explicit construction of a quantization $\DJ{\g_X}$, \ie a topological
quasi--triangular Hopf algebra over $\hext{\C}$ such that
\begin{enumerate}\itemsep0.2cm
	\item there exists an isomorphism of Hopf algebras $\DJ{\g_X}/\hbar\DJ{\g_X}\simeq \bfU\g_X$;
	\item for any $x\in\g_X$,
	\begin{align*}
		\delta(x)=\frac{\Delta(\wt{x})-\Delta^{21}(\wt{x})}{\hbar}\mod\hbar\ ,
	\end{align*}
	where $\wt{x}\in\DJ{\g_X}$ is any lift of $x\in\g_X$.
\end{enumerate} 
We refer to $\DJ{\g_X}$ as \emph{the continuum quantum group of $\calQ_X$}.

\subsection*{The continuum quantum group}
The definition of $\DJ{\g_X}$ is very similar in spirit to that of $\g_X$, but it depends on two 
additional partial operations on $\intsf(X)$:
\begin{enumerate}\itemsep0.2cm
	\item the \emph{strict union of two non--orthogonal intervals $\ia$ and $\ib$}, whenever defined, 
	is the smallest interval $\iM{\ia}{\ib}\in\intsf(X)$ for which $(\iM{\ia}{\ib})\sgpm\ia$ 
	and $(\iM{\ia}{\ib})\sgpm\ib$ are both defined;
	\item the \emph{strict intersection of two non--orthogonal intervals $\ia$ and $\ib$}, whenever defined, 
	is the biggest interval $\im{\ia}{\ib}\in\intsf(X)$ for which $\ia\sgpm(\im{\ia}{\ib})$ and 
	$\ib\sgpm(\im{\ia}{\ib})$ are both defined.
\end{enumerate}
Note that $\iM{\ia}{\ib}$ (resp. $\im{\ia}{\ib}$) is defined and coincides with $\ia\cup\ib$ 
(resp. $\ia\cap\ib$) whenever it contains \emph{strictly} $\ia$ and $\ib$ (resp.\ it is contained 
\emph{strictly} in $\ia$ and $\ib$).

\begin{nndefinition}[cf. Definition \ref{def:cont-qg}]
	Let $\calQ_X$ be a continuum quiver. The \emph{continuum quantum group of $X$} 
	is the associative algebra $\bfU_q\g_X $ generated by $\fun{X}$ 
	and the elements $\qxpm{\ia}$, $\ia\in\intsf(X)$, satisfying the following defining relations:
	\begin{enumerate}\itemsep0.2cm
		\item {\bf Diagonal action:} for any $\ia,\ib\in\intsf(X)$,
		\begin{align*}
			[\xz{\ia},\xz{\ib}]=0\qquad\mbox{and}\qquad [\xz{\ia},\qxpm{\ib}]=\pm\rbf{\ia}{\ib}\qxpm{\ib}\ .
		\end{align*}
		In particular, for $\qxz{\ia}{}\coloneqq\exp(\hbar/2\cdot\xz{\ia})$,
		it holds $\qxz{\ia}{}\qxpm{\ib}=q^{\pm\rbf{\ia}{\ib}}\cdot\qxpm{\ib}\qxz{\ia}{}$.
		
		\item {\bf Quantum double relations:} for any $\ia,\ib\in\intsf(X)$,
		\begin{align*} 
			[\qxp{\ia} ,\qxm{\ib} ] = \drc{\ia \ib}&\frac{\qxz{\ia}{}-\qxz{\ia}{-1}}{q-q^{-1}}\\ 
			&+\ca{\ia}{\ib}\cdot\left(q^{\qcc{\ia}{\ib}{+}}\, \qxp{\ia\sgpm\ib}
			\, \qxz{\ib}{\ca{\ia}{\ib}}-
			q^{\qcc{\ia}{\ib}{-}}\, \qxz{\ia}{\ca{\ia}{\ib}}\, \qxm{\ib\sgpm\ia}\right)\\
			&\quad+\qcb{\ib}{\ia}{}\, q^{\qcb{\ib}{\ia}{}}\,(q-q^{-1})\,\qxp{(\iM{\ia}{\ib})\sgpm\ib}\,
			\qxz{\im{\ia\, }{\, \ib}}{\qcb{\ia}{\ib}{}}\,\qxm{(\iM{\ia}{\ib})\sgpm{\ia}}\ .
		\end{align*}
		\item {\bf Quantum Serre relations:} for any $(\ia,\ib)\in\serre{X}$,
		\begin{align*}
			\qxpm{\ia}\qxpm{\ib}-q^{\qcr{\ia}{\ib}{}}&\cdot\qxpm{\ib}\qxpm{\ia}=
			\pm\qcb{\ia}{\ib}{}\cdot q^{\qcs{\ia}{\ib}{\pm}}\cdot\qxpm{\ia\sgpp\ib}
			+\qcb{\ia}{\ib}{}\cdot(q-q^{-1})\cdot\qxpm{\iM{\ia}{\ib}}\qxpm{\im{\ia\, }{\,\ib}}\ .
		\end{align*}
	\end{enumerate}
\end{nndefinition}

In the definition above, we assume that $\qxpm{\ia \odot \ib }=0$ whenever $\ia \odot \ib $ 
is not defined, for $\odot=\sgpp,\sgpm, \iM{}{}, \im{}{}$. Moreover, the coefficients are defined
as follows:
\begin{itemize}\itemsep0.2cm
	\item $\displaystyle \ca{\ia}{\ib}\coloneqq(-1)^{\abf{\ia}{\ib}}\rbf{\ia}{\ib}$;
	\item $\displaystyle \qcb{\ia}{\ib}{}\coloneqq\ca{\ia}{,\, \iM{\ia}{\ib}}$;
	\item $\displaystyle \qcc{\ia}{\ib}{+}\coloneqq\half\left(\ca{\ib}{,\,\ia\sgpm\ib}-1\right)$
	and $\displaystyle \qcc{\ia}{\ib}{-}\coloneqq\half\left(\ca{\ib\sgpm\ia}{,\,\ia}+1\right)$;
	\item $\displaystyle \qcr{\ia}{\ib}{}\coloneqq(1-\drc{\ia\ib})(-1)^{\abf{\ia}{\ib}}\rbf{\ia}{\ib}^2$;
	\item $\displaystyle \qcs{\ia}{\ib}{\pm}\coloneqq\half\left(\ca{\ib}{,\, \ia\sgpp\ib}\pm1\right)$.
\end{itemize}

In order to prove that $\DJ{\g_X}$ is naturally endowed with a topological quasi--triangular Hopf
algebra structure, we proceed as in the classical case, by showing that $\DJ{\g_X}$ can be also
realized \emph{by duality}. This leads to the following.

\begin{theorem*}[{cf. Theorem~\ref{thm:cont-qg-bia}}]
	Let $\calQ_X$ be a continuum quiver and $\bfU_q\g_X $ the corresponding 
	continuum quantum group. 
	\begin{enumerate}\itemsep0.2cm
		\item
		The algebra $\DJ{\g_X}$ is a topological Hopf algebra with respect to the maps
		\begin{align*}
			\Delta\colon \DJ{\g_X}\to\DJ{\g_X}\wh{\ten}\DJ{\g_X}\quad \text{and} \quad \varepsilon\colon \DJ{\g_X}\to\hext{\C} \ ,
		\end{align*}
		defined on the generators
		by $\varepsilon(\xz{\ia})\coloneqq 0\eqqcolon \varepsilon(\qxpm{\ia})$, $\Delta(\xz{\ia})\coloneqq \xz{\ia}\ten1+1\ten\xz{\ia}$, and
		\begin{align*}
			\Delta(\qxp{\ia})&\coloneqq \qxp{\ia}\otimes 1+ K_\ia \otimes \qxp{\ia}
			+\sum_{\ia=\ib\sgpp\ic}\, \ca{\ic}{, \,\ib\sgpp\ic} \, \qcs{\ib}{\ic}{-}\cdot q^{-1}(q-q^{-1})\, \qxp{\ib} K_\ic\otimes \qxp{\ic}\ ,\\[3pt]
			\Delta(\qxm{\ia})&\coloneqq 1\otimes \qxm{\ia}+ \qxm{\ia}\otimes K_\ia^{-1}
			-\sum_{\ia=\ib\sgpp\ic}\, \ca{\ic}{, \,\ib\sgpp\ic} \, \qcs{\ib}{\ic}{-}\cdot (q-q^{-1})\, \qxm{\ic} \otimes \qxm{\ib} K_\ic^{-1}\ .
		\end{align*}
		In particular, $\varepsilon(\qxz{\ia}{})=1$ and $\Delta(\qxz{\ia}{})=\qxz{\ia}{}\ten\qxz{\ia}{}$. As usual, the antipode is given by the formula 
		\begin{align*}
			S\coloneqq\sum_n m^{(n)}\circ(\id-\iota\circ\varepsilon)^{\ten n}\circ\Delta^{(n)}\ ,
		\end{align*}
		where $m^{(n)}$ and $\Delta^{(n)}$ denote the $n$th iterated product and coproduct, respectively.
		\item 
		Denote by $\DJ{{\b_X^{\pm}}}$ the Hopf subalgebras generated by 
		$\fun{X}$ and $\qxpm{\ia}$, $\ia\in\intsf(X)$. Then, there exists a unique non--degenerate 
		Hopf pairing 
		$\rbf{\cdot}{\cdot}$
		$\colon \DJ{{\b}_X^+}\ten(\DJ{{\b}_X^-})^{\scsop{cop}}\to\C(\negthinspace(\hbar)\negthinspace)$,
		defined on the generators by
		\begin{align*}
			\rbf{1}{1}\coloneqq 1\ , \quad 
			\rbf{\xz{\ia}}{\xz{\ib}} \coloneqq \frac{1}{\hbar}\rbf{\ia}{\ib}\ , \quad
			\rbf{\qxp{\ia}}{\qxm{\ib}}\coloneqq \frac{\delta_{\ia\ib}}{q-q^{-1}}\ ,
		\end{align*}
		and zero otherwise. In particular, $\rbf{\qxz{\ia}{}}{\qxz{\ib}{}}=q^{\rbf{\ia}{\ib}}$.
		\item Through the Hopf pairing $\rbf{\cdot}{\cdot}$, the Hopf algebras 
		$(\DJ{\b_X^{+}}, \DJ{\b_X^{-}})$ give rise to a match pair of Hopf algebras. 
		Then, $\DJ{\g_X}$ is realized as a quotient of the double cross product Hopf algebra
		$\DJ{{\b}_X^+}\dcs\DJ{\b_X^-}$ obtained by identifying the two copies of the commutative 
		subalgebra $\fun{X}$. In particular, $\DJ{\g_X}$ is a topological quasi--triangular Hopf algebra. 
		\item
		The topological quasi--triangular Hopf algebra $\DJ{\g_X}$ is a quantization of the 
		topological quasi--triangular Lie bialgebra $\g_X$. 
	\end{enumerate}
\end{theorem*}

Moreover, we prove that, as in the classical case, the continuum quantum group can be
realized as an uncountable colimits of Drinfeld--Jimbo quantum groups.

\begin{theorem*}[{cf.\ Corollary~\ref{cor:q-bkm-sgp}}]
	Let $\calJ,\calJ'$ be two irreducible (finite) sets of intervals in $X$.
	\begin{enumerate}\itemsep0.2cm
		\item Let $\bfU_q{\g_{\calJ}}$ be the Hopf subalgebra in $\bfU_q\g_X$ generated by the elements
		$\xz{\ia}$ and $\qxpm{\ia}$, with $\ia\in\calJ$. Then, there is a canonical isomorphism of algebras
		$\bfU_q{\g_{\calJ}^{\scsop{BKM}}}\to\bfU_q{\g_{\calJ}}$.
		\item If $\calJ'\subseteq\calJ$, there is a canonical embedding
		$\phi'_{\calJ,\calJ'}\colon \bfU_q\g_{\calJ'}\to\bfU_q\g_{\calJ}$ sending generator to generator.
		\item If $\calJ$ is obtained from $\calJ'$ by replacing an element $\ic\in\calJ'$
		with two intervals $\ia ,\ib $ such that $\ic=\ia \sgpp \ib $,
		there is a canonical embedding $\phi''_{\calJ,\calJ'}\colon \bfU_q\g_{\calJ'}\to\bfU_q\g_{\calJ}$, 
		which is the identity on the diagrammatic subalgebra  $\bfU_q\g_{\calJ'\setminus\{\ic\}}=\bfU_q\g_{\calJ\setminus\{\ia,\ib\}}$
		and sends
		\begin{align*}
			\xz{\ic} \mapsto \xz{\ia } +\xz{\ib } \ , \quad 
			\qxpm{\ic}  \mapsto 
			\mp\qcb{\ia}{\ib}{-1}\cdot q^{-\qcs{\ia}{\ib}{\pm}}\cdot
			\left(\qxpm{\ia}\qxpm{\ib}-q^{\qcr{\ia}{\ib}{}}\cdot\qxpm{\ib}\qxpm{\ia}\right)\ .
		\end{align*}
		\item The collection of embeddings $\phi'_{\calJ,\calJ'}, \phi''_{\calJ,\calJ'}$, 
		indexed by all possible irreducible sets of intervals in $X$, form a direct system. 
		Moreover, there is a canonical isomorphism of algebras  
		\begin{align*}
			\bfU_q\g_X\simeq\colim_{\calJ}\, \bfU_q\g_{\calJ}^{\scsop{BKM}} \ .
		\end{align*}
	\end{enumerate}
\end{theorem*}

\subsection*{The quantum groups of the line and the circle}

In \cite{sala-schiffmann-17}, the second--named author and O. Schiffmann introduced the 
line quantum group $\bfU_q \sl(\R)$ and the circle quantum group $\bfU_q \sl(S^1)$, the latter arising from the Hall algebra of parabolic (torsion) coherent sheaves on a curve. These are the simplest examples of continuum quantum groups. Namely, we get the
following.
\begin{theorem*}[{cf.\ Propositions~\ref{prop:slR-presentation} and \ref{prop:q-slR-presentation}}]
	There exists a canonical isomorphism of topological Hopf algebras $\bfU_q\sl(\R)\to\bfU_q\g_\R$. 
	At $q=1$, it gives rise to an isomorphism of topological Lie bialgebras $\sl(\R)\to\g_\R$.
\end{theorem*}
The case of the circle is slightly more delicate. Namely, the continuum Kac--Moody algebra $\g_{S^1}$ 
contains strictly the Lie algebra $\sl(S^1)$. Their difference is reduced to the elements $\xpm{S^1}$ 
corresponding to the full circle. More precisely, let $\ol{\g}_{S^1}$ be the subalgebra in $\g_{S^1}$ 
generated by the elements $\xpm{\ia}$, $\xz{\ia}$, $\ia\neq S^1$. Note that the elements $\xpm{S^1}$, $\xz{S^1}$,
generate a Heisenberg Lie algebra of order one in $\g_{S^1}$, which we denote $\heis_{S^1}$. Then, 
$\g_{S^1}=\ol{\g}_{S^1}\oplus\heis_{S^1}$ and there is a canonical embedding $\sl(S^1)\to\g_{S^1}$, 
whose image is $\ol{\g}_{S^1}\oplus\bsfld\cdot \xz{S^1}$. A similar relation holds for 
the Hopf algebras $\bfU_q \sl(S^1)$ and $\bfU_q \g_{S^1}$, where the role of $\heis_{S^1}$ is 
played by the subalgebra generated in $\bfU_q \g_{S^1}$ by $\xz{S^1}$ and $\qxpm{S^1}$.

\subsection*{Future directions}

In this last section, we shall outline some further directions of 
research, currently under investigations.

\subsubsection*{Geometric quantization}

As mentioned earlier, the continuum quantum groups  $\bfU_q\sl(S^1)$ and 
$\bfU_q\sl(\R)$ originate from a Hall algebra type construction. More precisely, the 
\emph{rational} circle quantum group $\bfU_q\sl(\Q/\Z)$ was realized in \cite{sala-schiffmann-17}  
in two different ways. That is, by the second--named author and O. Schiffmann, 
as the (reduced) quantum double of the spherical Hall algebra of torsion 
parabolic sheaves on a smooth projective curve over a finite field, 
and, by T. Kuwagaki, from the spherical Hall algebra of locally constant 
sheaves on $\Q/\Z$ with fixed singular support. The latter approach generalizes 
easily to $\R$ and $S^1$, and to the type $D$ case (a smooth tree with one root,
one node, and two leaves). 

In \cite{appel-kuwagaki-sala-schiffmann-18}, together with T. Kuwagaki and O. Schiffmann, 
we will provide two geometric realization of $\bfU_q\g_X $ arising from Hall algebras associated 
with the following abelian categories defined over a finite field. We first consider the category of 
\emph{coherent persistence modules} (extending the definition given in \cite{sala-schiffmann-19} 
for the line $\R$ and the circle $S^1$ to an arbitrary continuum quiver). Such objects can be thought 
of as a generalization of the usual notion of parabolic torsion sheaves on a curve, mimicking the first
realization of the circle quantum group. The analogue of the second \emph{symplectic} realization
is instead obtained from the category of locally constant sheaves over the underlying vertex space.

\subsubsection*{Highest weight theory}

In general, the usual combinatorics governing the highest weight theory of 
Borcherds--Kac--Moody algebras does not extend in a straightforward way to 
continuum Kac--Moody algebras, mainly due to the lack of simple roots. The 
appropriate tools to describe the highest weight theory of $\g_X $, the corresponding 
continuum Weyl group, and the character formulas, are currently under study. 
The same difficulties arise also at the quantum level.

Nonetheless, the geometric realization of continuum quantum groups would 
likely help towards a better understanding of its representation theory.
An inspiring example is given in \cite{sala-schiffmann-19}, where the second--named 
author and O. Schiffmann define the Fock space for $\bfU_q \sl(\R)$, considering 
a continuum analogue of the usual combinatorial construction in the case of $\bfU_q \sl(\infty)$. 
In addition, the quantum group $\bfU_q \sl(S^1)$ act on such a Fock space, in a way similar 
to the \emph{folding procedure} of Hayashi--Misra--Miwa. This construction should
extend to the case of an arbitrary continuum quiver $X$, producing a wide class of 
interesting representations for the continuum quantum group $\bfU_q \g_X $, and 
therefore for the continuum Kac--Moody algebra $\g_X$.

\subsection*{Outline}

In Section~\ref{s:km-dj}, we recall the basic definition of Kac--Moody algebras 
and Drinfeld--Jimbo quantum groups in the more general framework 
of quantization of Lie bialgebras.
In Section~\ref{s:cont-km}, we provide a concise exposition of the construction
of continuum Kac--Moody algebras, as introduced in \cite{appel-sala-schiffmann-18}, 
and their realization as uncountable colimits of Borcherds--Kac--Moody algebras.
In Section~\ref{s:cont-km-lba}, we prove the first main result of the paper, 
showing that continuum Kac--Moody algebras are
naturally endowed with a standard {topological} quasi--triangular 
Lie bialgebra structure.
In Section~\ref{s:cont-qg}, we define the continuum quantum group associated
to a continuum quiver and show that, in the cases of $\R$ and $S^1$, it coincides
with the quantum groups of the line and the circle introduced in \cite{sala-schiffmann-17}. 
Finally, in Section~\ref{s:cont-qg-bia}, we prove the second main result of the paper, 
showing that continuum quantum groups are topological quasi--triangular 
Hopf algebra, quantizing the standard Lie bialgebra structure of continuum 
Kac--Moody algebras.

\subsection*{Acknowledgements}

The second--named author would like to thank Professor Kyoji Saito for his contribution to create the Mathematics Group at Kavli IPMU, where the author worked as a postdoc, and to establish the GTM seminar, where \cite{sala-schiffmann-17} was presented, prompting the collaboration with
Tatsuki Kuwagaki. Moreover, this paper was begun while the first--named author was visiting Kavli IPMU. He is grateful to Kavli IPMU for its hospitality and wonderful working conditions. Finally, we would like to thank Tatsuki Kuwagaki and Olivier Schiffmann for many enlightening conversations, and Fabio Gavarini for his comments and interest in this work.


\bigskip\section{Kac--Moody algebras and quantum groups}\label{s:km-dj}

In this section, we recall the basic definition of Kac--Moody algebras 
and Drinfeld--Jimbo quantum groups in the more general framework 
of quantization of Lie bialgebras. The results of this section are well--known
and due to \cite{kac-90, drinfeld-quantum-groups-87}. We follow the exposition
of \cite{appel-toledano-18}.   

Henceforth, we fix a base field $\bsfld$ of characteristic zero and set $\bsfldh\coloneqq\hext{\bsfld}$.

\subsection{Quantization of Lie bialgebras}\label{ss:quant-lba}

A \emph{Lie bialgebra} is a triple $(\b,[\cdot,\cdot]_{\b}, \delta_{\b})$ where
\begin{enumerate}\itemsep0.2cm
	\item $\b$ is a discrete $\bsfld$--vector space;
	\item  $(\b,[\cdot,\cdot]_{\b})$ is a Lie algebra, \ie $[\cdot,\cdot]_{\b}\colon \b\ten\b\to\b$
	is anti-symmetric and satisfies the Jacobi identity
	\begin{align}\label{eq:Jacobi}
		[\cdot,\cdot]_{\b}\circ\id_{\b}\ten[\cdot,\cdot]_{\b}\circ(\id_{\b^{\ten 3}}+(1\,2\,3)+(1\,3\,2))=0\ ;
	\end{align}
	\item $(\b,\delta_{\b})$ is a Lie coalgebra, \ie $\delta_{\b}\colon\b\to\b\ten\b$
	is anti-symmetric and satisfies the co--Jacobi identity
	\begin{align}\label{eq:coJacobi}
		(\id_{\b^{\ten 3}}+(1\,2\,3)+(1\,3\,2))\circ\id_{\b}\ten\delta_{\b}\circ\delta_{\b}=0\ ;
	\end{align}
	\item the cobracket $\delta_{\b}$ satisfies the \emph{cocycle condition}
	\begin{align}\label{eq:cocycle}
		\delta_{\b}\circ[\cdot,\cdot]_{\b}=\sfad_{\b}\circ\id_{\b}\ten\delta_{\b}\circ(\id_{\b^{\ten 2}}-(1\,2))\ ,
	\end{align}
	as maps $\b\ten\b\to\b\ten\b$, 
	where $\sfad_{\b}\colon \b\ten\b\ten\b\to\b\ten\b$ denotes the left adjoint action of $\b$
	on $\b\ten\b$.
\end{enumerate}

A \emph{quantized enveloping algebra} (QUE) is a Hopf algebra $B$ in $\vectK$ such that 
\begin{enumerate}\itemsep0.2cm
	\item $B$ is endowed with the $\hbar$--adic topology, that is $\{\hbar
	^n B\}_{n\geq 0}$ is a basis of neighborhoods of $0$. Equivalently,
	$B$ is isomorphic, as topological $\bsfldh$--module, to $\hext{B_0}$, for
	some discrete topological vector space $B_0$.
	\item $B/\hbar B$ is a connected, cocommutative Hopf algebra over
	$\bsfld$. Equivalently, $B/\hbar B$ is isomorphic to an enveloping algebra $\bfU\b$ 
	for some Lie bialgebra $(\b, [\cdot,\cdot]_{\b}, \delta_{\b})$ and, under this identification,
	\begin{align*}
		\delta_{\b}(b)=\frac{\Delta(\wt{b})-\Delta^{21}(\wt{b})}{\hbar} \mod\hbar\ ,
	\end{align*}
	where $\wt{b}\in B$ is any lift of $b\in\b$. 
\end{enumerate}
We say that $B$ is a \emph{quantization} of $(\b, [\cdot,\cdot]_{\b}, \delta_{\b})$.

In Sections~\ref{ss:km}--\ref{ss:DJ} we will describe the standard Lie bialgebra structure on
symmetrisable Kac--Moody algebras and their quantization provided by Drinfeld--Jimbo
quantum groups.

\bigskip\subsection{Kac--Moody algebras}\label{ss:km}

We recall the definition from \cite[Chapter~1]{kac-90}. Fix a finite set $\bfI$ 
and a matrix $\gcm =(a_{ij})_{i,j\in\bfI}$ with entries in $\bsfld$. Recall that a realization $(\h,\Pi,\Pi^{\vee})$ of $\gcm$ is the datum of 
a finite dimensional $\bsfld$--vector space $\h$, and linearly independent vectors 
$\Pi\coloneqq \{\alpha_i\}_{i\in\bfI}\subset\h^\ast$, $\Pi^{\vee}\coloneqq \{h_i\}_{i\in\bfI} 
\subset\h$ such that $\alpha_i(h_j)= a_{ji}$. One checks easily that, in any realization 
$(\h,\Pi,\Pi^{\vee})$, $\dim\h\geqslant 2\vert\bfI\vert-\rank(\rls{\gcm})$. Moreover,
up to a (non--unique) isomorphism, there is a unique realization of minimal dimension 
$2\vert\bfI\vert-\rank(\rls{\gcm})$. 

For any realization $\rls{\gcm}=(\h,\Pi,\Pi^{\vee})$, let $\wt{\g}(\rls{\gcm})$ be the Lie 
algebra generated by $\h$, $\{e_i, f_i\}_{i\in\bfI}$ with relations $[h,h']=0$, for any $h,h'\in\h$, 
and
\begin{align}\label{eq:km-rel}
	[h,e_i]=\alpha_i(h)\, e_i\ ,
	\quad
	[h,f_i]=-\alpha_i(h)\, f_i\ ,
	\quad
	[e_i,f_j]=\drc{ij}\, h_i\ .
\end{align}
Set
\begin{align*}
	\rtl_+\coloneqq\bigoplus_{i\in\bfI}\Z_{\geqslant0}\, {\alpha}_i\subseteq{\h}^\ast\ ,
\end{align*}
$\rtl\coloneqq\rtl_+\oplus(-\rtl_+)$, and denote by $\wt{\n}_{+}$ (resp. $\wt{\n}_-$) the 
subalgebra generated by $\{e_i\}_{i\in\bfI}$ (resp. $\{f_i\}_{i\in\bfI}$). Then, as vector spaces, 
$\wt{\g}(\rls{\gcm})=\wt{\n}_+\oplus\h\oplus\wt{\n}_-$ and, with respect to $\h$, one has 
the root space decomposition
\begin{align*}
	\wt{\n}_{\pm}=\bigoplus_{\genfrac{}{}{0pt}{}{\alpha\in\rtl_+}{\alpha\neq0}}\, \wt{\g}_{\pm\alpha}
\end{align*}
where $\wt{\g}_{\pm\alpha}=\{X\in\wt{\g}(\rls{\gcm})\;\vert\;\forall h\in\h, [h,X]=\pm\alpha(h)X\}$. 
Note also that $\wt{\g}_0=\h$ and $\dim\wt{\g}_{\pm\alpha}<\infty$.

The \emph{Kac--Moody algebra} corresponding to the realization $\rls{\gcm}$ is the Lie 
algebra $\g(\rls{\gcm})\coloneqq\wt{\g}(\rls{\gcm})/\r$, where $\r$ is the sum of all 
two--sided graded ideals in $\wt{\g}(\rls{\gcm})$ having trivial intersection with $\h$. 
In particular, as ideals, $\r=\r_+\oplus\r_-$, where $\r_{\pm}\coloneqq\r\cap\wt{\n}_{\pm}$. 
\footnote{The terminology differs slightly from the one given in \cite{kac-90} where 
	$\g(\rls{\gcm})$ is called a Kac--Moody algebra if $\gcm$ is a generalised Cartan 
	matrix (cf.\ Remark~\ref{rem:r=s}) and $\rls{\gcm}$ is the minimal realization. Note also 
	that in \cite[Theorem~1.2]{kac-90} $\r$ is set to be the sum of all two--sided ideals, not 
	necessarily graded. However, since the functionals $\alpha_i$ are linearly independent in 
	$\h^\ast$ by construction, $\r$ is automatically graded and satisfies $\r=\r_+\oplus\r_-$ 
	(cf.\ \cite[Proposition~1.5]{kac-90}).} 

Since $\r=\r_+\oplus\r_-$, the Lie algebra $\g(\rls{\gcm})$ has an induced triangular decomposition 
$\g(\rls{\gcm})=\n_-\oplus\h\oplus\n+$ (as vector spaces), where
\begin{align*}
	\n_\pm\coloneqq \bigoplus_{\genfrac{}{}{0pt}{}{\alpha\in\rtl_+}{\alpha\neq0}}\g_{\pm\alpha}\ , \quad 
	\g_{\alpha}\coloneqq \{X\in{\g(\rls{\gcm})}\;\vert \; \forall\, h\in{\h},\, [h,X]=\alpha(h)\, X\}\ .  
\end{align*}
Note that $\dim\g_{\alpha}<\infty$. The set of positive roots is denoted $\rts_+\coloneqq 
\{\alpha\in\rtl_+\setminus\{0\}\;\vert \; {\g}_{\alpha}\neq0\}$.


\begin{remark}\label{rem:der-km}
	
	The derived subalgebra $\g(\rls{\gcm})'\coloneqq[\g(\rls{\gcm}),
	\g(\rls{\gcm})]$ has a similar and somewhat simpler description. One can show
	easily that $\g(\rls{\gcm})'$ is generated by the Chevalley generators $\{e_,f_i,h_i\}_{i\in\bfI}$ 
	and admits a presentation similar to that of $\g(\rls{\gcm})$. Namely, let $\wt{\g}'$ be the 
	Lie algebra generated by $\{h_i,e_i,f_i\}_{i\in\bfI}$ with relations
	\begin{align*}
		[h_i,h_j]=0,
		\quad
		[h_j,e_i]=\alpha_i(h_j)\, e_i,
		\quad
		[h_j,f_i]=-\alpha_i(h_j)\, f_i,
		\quad
		[e_i,f_j]=\drc{ij}\, h_i.
	\end{align*}
	Then, $\wt{\g}'$ has a $\rtl$--gradation defined by $\deg(e_i)=\alpha_i$, $\deg(f_i)=-\alpha_i$,
	and $\deg(h_i)=0$. Clearly, $\wt{\g}'_0=\h'$, where the latter is the $|\bfI|$--dimensional span of 
	$\{h_i\}_{i\in\bfI}$. The quotient of $\wt{\g}'$ by the sum of all two--sided graded ideals with 
	trivial intersection with $\h'$ is easily seen to be canonically isomorphic to $\g(\rls{\gcm})'$.  
\end{remark}


\begin{remark}\label{rem:ext-KM}
	
	It is sometimes convenient to consider the Kac--Moody algebras associated to a (non--minimal)
	\emph{canonical realization}, which allows to obtain a presentation similar to that of the
	derived subalgebra  (cf.\ \cite{feigin-zelevinsky-85, maulik-okounkov-12, appel-toledano-16}).
	Namely, let $\crls{\gcm}=(\ol{\h},\ol{\Pi},\ol{\Pi}^\vee)$ be the realization given by 
	$\ol{h}\cong\bsfld^{2|\bfI |}$ with basis $\{\hcor{i}\}_{i\in\bfI}\cup\{\cow{i}\}_{i\in\bfI}$, 
	$\ol{\Pi}^\vee=\{\hcor{i}\}_{i\in\bfI}$ and $\ol{\Pi}=\{\alpha_i\}_{i\in\bfI}\subset\ol{\h}^\ast$, 
	where $\alpha_i$ is defined by
	\begin{align*}
		\alpha_i(\hcor{j} )=a_{ji}
		\qquad\mbox{and}\qquad
		\alpha_i(\cow{j} )=\delta_{ij} \ .
	\end{align*}
	Then, $\wt{\g}(\crls{\gcm})$ is the	Lie algebra generated by $\{\hcor{i}, \cow{i}, e_i,f_i\}_{i\in\bfI}$ 
	with relations
	\begin{align*}
		[\hcor{i}, \hcor{j}]=0\ , \quad [\hcor{i}, \cow{j}]=0\ , \quad [\cow{i}, \cow{j}]=0 \ ,\\
		[\hcor{i},e_j]=a_{ij}\, e_j\ ,
		\quad
		[\hcor{i},f_j]=-a_{ij}\, f_j\ ,\\
		[\cow{i},e_j]=\drc{ij}\,e_j\ ,
		\quad
		[\cow{i},f_j]=-\drc{ij}\, f_j\ ,
	\end{align*}
	and $[e_i,f_j]=\drc{ij}\, h_i$.
	It is easy to check that the Kac--Moody algebra $\g(\crls{\gcm})$ is just a central extension 
	of $\g(\mrls{\gcm})$, \ie $\g(\crls{\gcm})\simeq\g(\mrls{\gcm})\oplus\c$, with 
	$\dim\c=\rank(\gcm)$.
\end{remark}


\begin{remark}\label{rem:r=s}
	
	It is well--known that  in certain cases the ideal $\r$ can be described explicitly.
	If $\gcm $ is a \emph{generalised Cartan matrix} (\ie $a_{ii}=2$, $a_{ij}\in\Z_{\leqslant 0}$, 
	$i\neq j$, and $a_{ij}=0$ implies $a_{ji}=0$), then $\r$ contains the ideal generated 
	by the Serre relations
	\begin{align}\label{eq:SerrerelgCm}
		\sfad(e_i)^{1-a_{ij}}(e_j)=0=\sfad(f_i)^{1-a_{ij}}(f_j)\qquad i\neq j
	\end{align}
	and coincides with it if $\gcm$ is also symmetrizable \cite{gabber-kac-81}. 
	
	A similar description of $\r$ holds for any \emph{Borcherds--Cartan matrix} $\gcm$ (\ie 
	such that $a_{ij}\in\Z_{\leqslant0}$, $i\neq j$, and $2a_{ij}/a_{ii}\in\Z$ whenever $a_{ii}>0$). 
	In this case, $\g$ is called a Borcherds--Kac--Moody algebra and the corresponding maximal 
	ideal contains the Serre relations
	\begin{align}\label{eq:Serrerel-BKM-1}
		\sfad(e_i)^{1-\frac{2}{a_{ii}}a_{ij}}(e_j)=0=\sfad(f_i)^{1-\frac{2}{a_{ii}}a_{ij}}(f_j)
	\end{align}
	if $a_{ii}>0$ and $i\neq j$, and
	\begin{align}\label{eq:Serrerel-BKM-2}
		[e_i,e_j]=0=[f_i,f_j]
	\end{align}
	if $a_{ii}\leqslant0$ and $i\neq j$.
	Even in this case, if $\gcm$ is symmetrizable, $\r$ is generated by the Serre relations (cf.\ \cite[Corollary~2.6]{borcherds-88}).
\end{remark}


If the matrix $\gcm$ is symmetrizable, the corresponding Kac--Moody 
algebra can be further endowed with a standard Lie bialgebra structure.
Assume henceforth that $\gcm$ is symmetrizable, and fix a realization 
$\rls{\gcm}=(\h,\Pi,\Pi^\vee)$ and an invertible
diagonal matrix $\sfD=\mathsf{diag}(d_i)_{i\in\bfI}$ such that 
$\sfD\sfA$ is symmetric. Let $\h'\subset\h$ be the span of $\{\hcor{i}\}_{i\in\bfI}$, 
and $\h''\subset\h$ a complementary subspace. Then, there is a symmetric, non--degenerate 
bilinear form $\iip{\cdot}{\cdot}$ on $\h$, which is uniquely determined by
$\iip{\hcor{i}}{\cdot}\coloneqq d_i^{-1}\alpha_i(\cdot)$ and $\iip{\h''}{\h''}\coloneqq 0$. 
The form $\iip{\cdot}{\cdot}$ uniquely extends to an invariant symmetric
bilinear form on $\wt{\g}$, and $\iip{e_i}{f_j}=\drc{ij}d_i^{-1}$. The kernel 
of this form is precisely $\r$, so that $\iip{\cdot}{\cdot}$ descends to a 
non--degenerate form on $\g$.\footnote{Since $\iip{\cdot}{\cdot}$ 
	is non--degenerate on $\h$, the kernel $\mathfrak{k}\coloneqq\ker\iip{\cdot}{\cdot}$ is a graded 
	ideal trivially intersecting $\h$ and therefore it is contained in $\r$. Conversely, 
	for any graded ideal $\mathfrak{i}=\bigoplus_{\alpha}\mathfrak{i}_{\alpha}$ trivially
	intersecting $\h$, one has $\mathfrak{i}\subseteq\mathfrak{k}$. More precisely, 
	let $X\in\mathfrak{i}_{\alpha}$, $Y\in\wt{\g}_{\beta}$ and $Z\in\h$ such that $\beta(Z)\neq0$.
	Then,
	\begin{align*}
		\beta(Z)\cdot\iip{X}{Y}=\iip{X}{[Z,Y]}=-\iip{[X,Y]}{Z}=0\ .
	\end{align*}
	In particular $\r\subseteq\mathfrak{k}$ and therefore $\r=\mathfrak{k}$.}

Set $\b_{\pm}\coloneqq\h\oplus\bigoplus_{\alpha\in\sfR_+}\g_{\pm\alpha}\subset\g$. 
One can see easily that the bilinear form induces a canonical isomorphism of graded 
vector spaces $\b_{+}\simeq\b_{-}^{\star}$, where 
$\b_{-}^{\star}\coloneqq\h^\ast\oplus\bigoplus_{\alpha\in\sfR_+}\g_{-\alpha}^\ast$, and, more 
specifically, $\g_{\alpha}\simeq\g_{-\alpha}^\ast$. 

Let $\{e_{\alpha,i}\;|\; i=1,\dots,\dim\g_\alpha\}$ 
and $\{f_{\alpha,i}\;|\; i=1,\dots,\dim\g_\alpha\}$ be bases of $\g_{\alpha}$ and $\g_{-\alpha}$, 
respectively, which are dual to each other with respect to $\iip{\cdot}{\cdot}$, and set
\begin{align*}
	r\coloneqq\sum_{\alpha\in\sfR_+}\sum_{i=1}^{\dim\g_{\alpha}}e_{\alpha,i}\ten f_{\alpha,i}+\sum_{i=1}^{\dim\h}x_i\ten x_i\ ,
\end{align*}
where $\{x_i\;|\; i=1,\dots,\dim\h\}$ is an orthonormal basis of $\h$.

We will show in Section~\ref{ss:km-lba} that $\g$ has a natural structure of Lie bialgebra 
with cobracket  $\delta\colon\g\to\g\wedge\g$ given by
\begin{align*}
	\delta|_{\h}\coloneqq 0\ ,
	\qquad
	\delta(e_i)\coloneqq d_i \hcor{i}\wedge e_i\ ,
	\qquad
	\delta(f_i)\coloneqq d_i \hcor{i}\wedge f_i\ .
\end{align*}
Moreover, it satisfies $\delta(x)=[x\ten1+1\ten x, r]$. 

\subsection{Quasi--triangular Lie bialgebras}\label{ss:lba}

A Lie bialgebra is \emph{quasi--triangular} if there exists a tensor
$r\in\b\ten\b$ such that
\begin{enumerate}\itemsep0.2cm
	\item $\Omega\coloneqq r+r_{21}$ is $\b$--invariant, \ie $[x\ten1+1\ten x,\Omega]=0$
	for any $x\in\b$;
	\item $r$ is a solution of the \emph{classical Yang--Baxter equation}, \ie
	\begin{align}\label{eq:cybe}
		[r_{12},r_{13}]+[r_{12},r_{23}]+[r_{13},r_{23}]=0\ ;
	\end{align}
	\item $\delta_{\b}=\partial r$, \ie for any $x\in\b$, $\delta_{\b}(x)=[x\ten1+1\ten x, r]$.
\end{enumerate} 

It is well--known that any Lie bialgebra $(\b,[\cdot,\cdot]_{\b},\delta_{\b})$ can be canonically 
embedded into a quasi--triangular \emph{topological} Lie bialgebra.
We recall below three versions of this construction, in terms of
\emph{Drinfeld doubles}, \emph{Manin triples} and \emph{matched pairs of Lie algebras}.

\subsubsection{Drinfeld double}\label{sss:drinf-dbl}

Let $(\b,[\cdot,\cdot]_{\b},\delta_{\b})$ be a Lie bialgebra.
The Drinfeld double $\gb$ is defined as follows.
\begin{itemize}[leftmargin=1em]\itemsep0.5cm
	\item 	As a vector space, $\gb=\b\oplus\b^\ast$. The canonical pairing $\iip{\cdot}{\cdot}\colon\b\ten\b^\ast
	\to\bsfld$ extends uniquely to a symmetric non--degenerate bilinear form
	on $\gb$, with respect to which $\b$ and $\b^\ast$ are isotropic. The Lie bracket on $\gb$ is
	defined as the unique bracket which coincides with $[\cdot,\cdot]_{\b}$ on $\b$,
	with $\delta_{\b}^t$ on $\b^\ast$, and is compatible with $\iip{\cdot}{\cdot}$, \ie satisfies
	$\iip{[x,y]}{z}=\iip{x}{[y,z]}$                                    
	for all $x,y,z\in\gb$. The mixed bracket of $x\in\b$ and $\phi\in\b^\ast$ is then 
	given by
	\begin{align*}
		[x,\phi]\coloneqq \sfad^\ast(x)(\phi)-\sfad^\ast(\phi)(x)\ ,
	\end{align*}
	where $\sfad^\ast$ denotes the coadjoint actions of $\b$ on $\b^\ast$ and of $\b^\ast$ on $(\b^\ast)^\ast$.
	\item We endow $\b$ and $\b^\ast$ with the discrete and the weak topology, respectively, and 
	$\gb=\b\oplus\b^\ast$ with the product topology. It is clear that the map $[\cdot,\cdot]_{\b}^t\colon \b^\ast\to
	\b^\ast\wh{\ten}\b^\ast$, where $\wh{\ten}$ denotes the completed tensor product, defines on $\b^\ast$
	a topological cobracket. Similarly, $\delta\coloneqq\delta_{\b}-[\cdot,\cdot]_{\b}^t$ defines a topological
	cobracket on $\gb$, which is easily seen to be compatible with $[\cdot,\cdot]$. 
	Therefore, $(\gb,[\cdot,\cdot], \delta)$ is a topological Lie bialgebra.
	\item Finally, the quasi--triangular structre on $\gb$ is given by the \emph{topological} canonical
	tensor $r\in\b\wh{\ten}\b^\ast\subset\gb\wh{\ten}\gb$ corresponding to the identity under the
	identification $\End(\b)\simeq\b\wh{\ten}\b^\ast$.
\end{itemize}

\begin{remark} 
	If  $\b=\bigoplus_{n\in\N}\b_n$ is $\N$--graded with finite--dimensional homogeneous components, the
	restricted dual $\b^{\star}\coloneqq\bigoplus_{n\in\N}\b_n^\ast$ and the restricted double $\grb=\b\oplus\b^{\star}$
	of $\b$ are also Lie bialgebras, with cobracket $\delta_\b-[\cdot,\cdot]_\b^t$. Moreover,
	$\grb$ is quasi--triangular with respect to the canonical tensor 
	$r\in\b\wh{\ten}\b^\star\coloneqq\prod_{n\in\N}\b_n\ten\b_n^\ast$.
\end{remark} 

\subsubsection{Manin triples}\label{sss:manin-triple}

A Manin triple is the datum of a Lie algebra $\g$ with a non--degenerate invariant symmetric 
bilinear form $\iip{\cdot}{\cdot}$ and two isotropic Lie subalgebras $\b_{\pm}\subset\g$
such that 
\begin{enumerate}\itemsep0.2cm
	\item as a vector space $\g=\b_+\oplus\b_-$; 
	\item the inner product defines an isomorphism $\b_+\to\b_-^\ast$;
	\item the commutator of $\g$ is continuous with respect to the topology 
	obtained by putting the discrete and the weak topologies on $\b_-$
	and $\b_+$  respectively.
\end{enumerate}

Under these assumptions, the commutator on $\b_+\simeq\b_-^\ast$
induces a cobracket $\delta\colon \b_-\to\b_-\otimes\b_-$ which satisfies
the cocycle condition. Therefore, $\b_-$ is canonically endowed with a 
Lie bialgebra structure, while $\b_+$ and $\g$ are, in general, only 
topological Lie bialgebras. Moreover, $\g$ is isomorphic, as a topological
Lie bialgebra, to the Drinfeld double of $\b_-$.

\begin{remark}
	If $\b$ is an $\N$--graded Lie bialgebra with finite--dimensional 
	homogeneous components, one can consider \emph{restricted} Manin triples, where the inner 
	product induces a isomorphism $\b_+\to\b_-^{\star}$.
	In this case, $\b_+$ and $\g$ are both Lie bialgebras and the latter is isomorphic
	to the restricted Drinfeld double of $\b_-$.
\end{remark}

\subsubsection{Matched pairs of Lie algebras}\label{sss:matched-lie}

The last construction is due to S. Majid \cite{majid-book-95} and it is, 
from a certain point of view, the most abstract, since it does not rely on a pairing. 
Two Lie algebras $(\c,[\cdot,\cdot]_{\c})$ and $(\d,[\cdot,\cdot]_{\d})$ form a 
\emph{matched pair} if there are maps
\begin{align*}
	\cond\colon \c\ten\d\to\d\qquad\mbox{and}\qquad\donc\colon \c\ten\d\to\c
\end{align*}
such that
\begin{enumerate}
	\item $\cond$ is a left action of $\c$ on $\d$, \ie
	\begin{align*}
		\cond\circ[\cdot,\cdot]_{\c}\ten\id=\cond\circ\id\ten\cond\circ(\id-(1\,2))\ ,
	\end{align*}
	and $\donc$ is a right action of $\d$ on $\c$, \ie
	\begin{align*}
		\donc\circ\id\ten[\cdot,\cdot]_{\d}=\donc\circ\donc\ten\id\circ(\id-(2\,3))\ ;
	\end{align*}
	\item $\donc, \cond$ satisfy the compatibility conditions
	\begin{align*}
		\donc\circ[\cdot,\cdot]_{\c}\ten\id
		&=
		[\cdot,\cdot]_{\c}\circ\donc\ten\id\circ(2\,3)+[\cdot,\cdot]_{\c}\circ\id\ten\donc
		+\donc\circ\id\ten\cond\circ(\id-(1\,2))\ ,\\
		\intertext{and}
		\cond\circ\id\ten[\cdot,\cdot]_{\d}
		&=
		[\cdot,\cdot]_{\d}\circ\cond\ten\id+[\cdot,\cdot]_{\d}\circ\id\ten\cond\circ(1\,2)
		+\cond\circ\donc\ten\id\circ(\id-(2\,3))\ .
	\end{align*}
\end{enumerate}

\begin{remark}
	The conditions above are equivalent to the requirement that the
	vector space $\c\oplus\d$ is endowed with a Lie bracket for which
	$\c,\d$ are Lie subalgebras and, for $X\in\c$ and $Y\in\d$,
	\begin{align*}
		[X,Y]_{\dcs}=X\cond Y+X\donc Y\ .
	\end{align*}
	The Lie algebra $\c\dcs\d=(\c\oplus\d,[\cdot,\cdot]_{\dcs})$
	is called the \emph{bicross sum Lie algebra of $\c,\d$}.
\end{remark}

\begin{example}
	If $(\b,[\cdot,\cdot]_{\b},\delta_{\b})$ is a Lie bialgebra, then $(\b,[\cdot,\cdot]_{\b})$
	and $(\b^\ast,\delta_{\b}^t)$ form a matched pair with respect to the coadjoint
	action of $\b$ on $\b^\ast$ and the opposite coadjoint action of $\b^\ast$ on $\b$.
	The corresponding double cross sum Lie algebra $\b\dcs\b^\ast$ is precisely the 
	Drinfeld double of $\a$.
\end{example}

\subsection{Lie bialgebra structure on Kac--Moody algebras}\label{ss:km-lba}

It is well--known that any symmetrisable Kac--Moody algebra 
has a canonical structure of (quotient of) a Manin triple, which 
induces on it a standard Lie bialgebra structure.

Let $\sfA$ be a symmetrisable Borcherds--Cartan matrix and fix an invertible
diagonal matrix $\sfD=\mathsf{diag}(d_i)_{i\in\bfI}$ such that 
$\sfD\sfA$ is symmetric. The bilinear form $\iip{\cdot}{\cdot}$ induces a 
canonical isomorphisms $\b^{\star}_{\pm}\simeq\b_{\mp}$, where $\b^{\star}_{\pm}$ 
denotes the restricted dual. Consider the product Lie algebra $\gtwo=\g\oplus\zh$, with $\zh=\h$, and
endow it with the inner product $\iip{\cdot}{\cdot}-\left.\iip{\cdot}{\cdot}
\right|_{\zh\times \zh}$. Let $\pi_0\colon \g\to\g_0\coloneqq\h$ be the projection, and
$\btwo_\pm\subset\gtwo$ the subalgebras
\begin{align*}
	\btwo_\pm\coloneqq\{(X,h)\in\b_\pm\oplus\zh|\,\pi(X)=\pm h\}\ .
\end{align*}
Note that the projection $\gtwo\to\g$ onto the first component restricts
to an isomorphism $\btwo_\pm\to\b_\pm$ with inverse $\b_\pm\ni X\mapsto
(X,\pm\pi_0(X))\in\btwo_\pm$. The following is straightforward.
\begin{enumerate}\itemsep0.5cm
	\item $(\gtwo, \btwo_-, \btwo_+)$ is a restricted Manin triple.
	In particular, $\btwo_\mp$ and $\gtwo$ are Lie bialgebras,
	with cobracket $\delta_{\btwo_\mp}\coloneqq[\cdot,\cdot]_{\btwo_\pm}^t$
	and $\delta_{\gtwo}=\delta_{\btwo_-}-\delta_{\btwo_+}$.
	\item The central subalgebra $0\oplus\zh\subset\gtwo$ is a coideal,
	so that the projection $\gtwo\to\g$ induces a Lie bialgebra structure
	on $\g$ and $\b_\mp$. 
	\item The Lie bialgebra structure on $\g$ is given by
	\begin{align*}
		\left.\delta\right|_{\h}=0
		\ ,\quad
		\delta(e_i)=d_i \hcor{i}\wedge e_i
		\ ,\quad
		\delta(f_i)=d_i \hcor{i}\wedge f_i\ .
	\end{align*}
\end{enumerate}

\subsection{Kac--Moody algebras by duality}\label{ss:KM-halb}

We recall  an alternative construction of symmetrisable Kac--Moody algebra, provided
by  G. Halbout in terms of matched pairs of Lie bialgebras \cite{halbout-99}. More precisely,
his construction goes as follows.
\begin{itemize}[leftmargin=1em]\itemsep0.5cm
	\item Let $\sfA$ be a symmetrisable Borcherds--Cartan matrix,  $\sfD=\mathsf{diag}(d_i)_{i\in\bfI}$ 
	an invertible diagonal matrix  such that $\sfD\sfA$ is symmetric, and $\iip{\cdot}{\cdot}$ the corresponding 
	non--degenerate bilinear form on $\h$.
	\item Let $\FL{\pm}$ be the free Lie algebras generated by the set $X_{\pm}\coloneqq\{\xpm{i}, \xzpm{}\;|\; i\in\bfI, \xz{}\in\h\}$.
	The assignment
	\begin{align*}
		\delta_{\pm}(\xzpm{})\coloneqq 0
		\quad\mbox{and}\quad
		\delta_{\pm}(\xpm{i})\coloneqq \mp d_i\hcorpm{i}\wedge\xpm{i}
	\end{align*}
	extends uniquely to a cobracket on $\FL{\pm}$ and induces a Lie bialgebra structure on it.
	\item The assignment
	\begin{align*}
		\abf{\xp{i}}{\xm{j}}\coloneqq d_i^{-1}\drc{ij}\ ,\quad \abf{\xzp{}}{\xzm{}}\coloneqq 2\,\iip{\xz{}}{\xz{}}\ ,
		\quad
		\abf{\xp{i}}{\xzm{}}\coloneqq 0\eqqcolon \abf{\xzp{}}{\xm{i}}\ ,
	\end{align*}
	extends uniquely to Lie bialgebra pairing $\abf{\cdot}{\cdot}\colon \FL{+}\ten\FL{-}\to\bsfld$, \ie
	for $X_{\pm}, Y_{\pm}\in\FL{\pm}$, 
	\begin{align}\label{eq:q:lba-pairing}
		\abf{[X_{\pm}, Y_{\pm}]}{X_{\mp}}=\abf{X_{\pm}\ten Y_{\pm}}{\delta_{\mp}(X_{\mp})}\ .
	\end{align}
	Then, $\FL{+}$ and $\FL{-}$ naturally form a matched pair of Lie bialgebras.
	\footnote{By slight abuse of notation, we impose that $\abf{\cdot}{\cdot}$ is 
		\emph{symmetric}, \ie it can be considered as
		a function on either $\FL{+}\ten\FL{-}$ or $\FL{-}\ten\FL{+}$, regardless of the order.
		Moreover, note that \eqref{eq:q:lba-pairing} can be equivalently restated as
		$\abf{[X_{\pm}, Y_{\pm}]}{X_{\mp}}=\abf{X_{\pm}\wedge Y_{\pm}}{\delta_{\mp}(X_{\mp})}$\ .} 
	
	The pairing $\abf{\cdot}{\cdot}$
	extends to the a possibly degenerate, invariant pairing on the double cross sum Lie bialgebra $\FL{+}\dcs\FL{-}$.
	\item The ideals generated by $[\xzpm{},\xzpm{}]$, $[\xzpm{},\xpm{i}]\mp\alpha_i(\xi)\xpm{i}$, 
	$\sfad(\xpm{i})^{1-\frac{2}{a_{ii}}a_{ij}}(\xpm{j})$ ($i\neq j$ and $a_{ii}>0$), and $[\xpm{i},\xpm{j}]$ ($a_{ii}\leqslant 0$ and $a_{ij}=0$) 
	are orthogonal to $\FL{\mp}$ and are coideals. Let $\mathfrak{s}$ be the sum of these ideals. In particular,  
	$\mathfrak{s}\subseteq\mathfrak{k}=\ker\abf{\cdot}{\cdot}\subseteq\FL{+}\dcs\FL{-}$. 
	\item Finally, one observes that $\FL{+}\dcs\FL{-}/\mathfrak{k}$ has the form $\mathfrak{g}\oplus\zh$, where $\zh$ is a central copy
	of $\h$ and $\g$ is the Borcherds--Kac--Moody algebra associated to $\gcm$.\footnote{Indeed, it is clear that there is a surjective morphism of Lie algebras
		$\pi\colon \g\to\d$, where $\d=\FL{+}\dcs\FL{-}/(\mathfrak{k}\oplus\zh)$, and, since $\ker\pi\subset\g$ is an ideal trivially intersecting
		$\h$, it must be necessarily trivial.} This implies, in particular, that $\mathfrak{k}$ coincides with $\mathsf{s}$ and it is a coideal. 
	Therefore, the Lie bialgebra structure on $\FL{+}\dcs\FL{-}$ naturally descends to $\g$.
\end{itemize} 

\subsection{Drinfeld--Jimbo quantum groups}\label{ss:DJ}

Let $\sfA $ be a symmetrisable Borcherds--Cartan matrix and
fix an invertible diagonal matrix $\mathsf{D}=\mathsf{diag}(d_i)_{i\in\bfI}$ such that 
$\sfD\sfA$ is symmetric. Let $\g=\g(\sfA)$ be the corresponding 
Borcherds--Kac--Moody algebra with its standard Lie bialgebra structure, and set 
$q\coloneqq\exp(\hbar/2)$, $q_i\coloneqq\exp(\hbar/2\cdot d_i)$.
The following is a straightforward generalization to Borcherds--Kac--Moody algebras of
the quantum group defined by Drinfeld \cite{drinfeld-quantum-groups-87} and 
Jimbo \cite{jimbo-85} (cf. also \cite{kang-95}).

The Drinfeld--Jimbo quantum group of $\g$ is the associative
algebra $\DJ{\g}$ topologically generated over $\bsfldh$ by 
$\h$ and the elements $E_i, F_i, i\in\bfI$ satisfying the following defining relations.
\begin{enumerate}\itemsep0.5cm
	\item {\bf Diagonal action:} for $h,h'\in\h$, $i\in\bfI$,
	\begin{align*}
		[h,h']=0\ ,\quad [h,E_i]=\alpha_i(h)E_i\ ,\quad [h,F_i]=-\alpha_i(h)F_i\ .
	\end{align*}
	In particular, for $\qxz{i}{}\coloneqq\exp(\hbar/2\cdot d_ih_i)$,
	it holds $\qxz{i}{}E_j=q_i^{a_{ij}}\cdot E_j\qxz{i}{}$ and $\qxz{i}{}F_j=q_i^{-a_{ij}}\cdot E_j\qxz{i}{}$\ .
	\item {\bf Quantum double relations:} 
	\begin{align*} 
		[E_i, F_i]=\frac{\qxz{i}{}-\qxz{i}{-1}}{q_i-q_i^{-1}}\ ,
	\end{align*}
	\item {\bf Quantum Serre relations:} for $i,j\in\bfI$ with $a_{ij}=0$,
	\begin{align*}
		[E_i, E_j]=0=[F_i,F_j]\ ,
	\end{align*}
	and for $i,j\in\bfI$, $i\neq j$, with $a_{ii}=2$,
	\begin{align*}
		\sum_{m=0}^{1-a_{ij}}\frac{(-1)^m}{[m]_{q_i}! [1-a_{ij}-m]_{q_i}!}X_i^{1-a_{ij}-m}X_jX_i^m=0\qquad (X=E,F)\ .
	\end{align*}
\end{enumerate}

Moreover, $\DJ{\g}$ has a Hopf algebra structure with counit, coproduct and antipode defined,
for $h\in\h$ and $i\in\bfI$, by
\begin{align*}
	\begin{array}{lll}
		\varepsilon(h)=0\ , & \Delta(h)=h\ten1+1\ten h\ ,& S(h)=-h\ ,\\[2pt]
		\varepsilon(E_i)=0\ , & \Delta(E_i)=E_i\ten\qxz{i}{}+1\ten E_i\ ,& S(E_i)=-E_i\qxz{i}{-1}\ ,\\[3pt]
		\varepsilon(F_i)=0\ , & \Delta(F_i)=F_i\ten 1+\qxz{i}{-1}\ten F_i\ ,& S(F_i)=-\qxz{i}{}F_i\ .
	\end{array}
\end{align*}


The following is well--known (cf.\ \cite{drinfeld-quantum-groups-87,lusztig-book-94,chari-pressley}).

\begin{theorem}\hfill
	\begin{enumerate}\itemsep0.3cm
		\item
		The Hopf algebra $\DJ{\g}$ is a quantization of the Lie bialgebra $\g$. 
		\item 
		Denote by $\DJ{{\b_-}}$ (resp. $\DJ{{\b}_+}$) the Hopf subalgebra generated by 
		$\h$ and $\{F_i,i\in\bfI\}$ (resp. $\h$ and $\{E_i,i\in\bfI\}$). 
		Then, $\DJ{{\b_-}}$ (resp. $\DJ{{\b}_+}$) is a quantisation of the Lie bialgebra 
		${\b_-}$ (resp. ${\b}_+$), and there exists a unique non--degenerate Hopf pairing 
		$\iip{\cdot}{\cdot}_{\D}\colon \DJ{{\b}_-}\ten\DJ{{\b}_+}\to\bsfld(\negthinspace(\hbar)\negthinspace)$, \ie a non--degenerate
		bilinear form compatible with the Hopf algebra structure,
		defined on the generators by
		\begin{align*}
			\iip{1}{1}_{\D} \coloneqq 1\ , \quad 
			\iip{h}{h'}_{\D} \coloneqq \frac{1}{\hbar}\iip{h}{h'}\ , \quad
			\iip{F_i}{E_j}_{\D} \coloneqq \frac{\delta_{ij}}{q-q^{-1}}\ ,
		\end{align*}
		and zero otherwise. 
		
		\item The Hopf pairing $\iip{\cdot}{\cdot}_{\D}$ induces an isomorphism of Hopf algebras
		between $\DJ{{\b}_-}$ and $(\DJ{{\b}_+})^{\star}$, which restricts to the identification 
		$\phi\colon \h\to\h^\ast$, $\phi(h)=-2\iip{h}{\cdot}$. Moreover, $\DJ{\g}$ can be realized
		as a quotient of the \emph{restricted quantum double} of $\DJ{{\b}_-}$ with respect
		to this identification, \ie $\D{\DJ{{\b}_-}}/(\h\simeq\h^\ast)\simeq\DJ{\g}$.
		In particular, $\DJ{\g}$ is a quasi--triangular Hopf algebra with $R$--matrix
		\begin{equation}
			\ol{R}=q^{\sum_{i}u_i\ten u_i}\cdot\sum_pX_p\ten X^p,
		\end{equation}
		where $\{u_i\}\subset\h$ is an orthonormal basis with respect to $\iip{\cdot}{\cdot}$,
		$\{X_p\}\subset\DJ{{\n}_-}$, $\{X^p\}\subset\DJ{{\n}_+}$ are dual basis with respect
		to the pairing $\iip{\cdot}{\cdot}_{\D}$.
	\end{enumerate}
\end{theorem}

It is useful to notice here that the proof of the theorem and the construction of the 
Hopf pairing $\rbf{\cdot}{\cdot}_{\D}$ is obtained following a quantum analogue of 
the procedure described in Section \ref{ss:KM-halb} (cf.\ \cite[Part I]{lusztig-book-94}).

\bigskip\section{Continuum Kac--Moody algebras}\label{s:cont-km}

In this section, we recall the notion of continuum Kac--Moody algebras 
introduced in \cite{appel-sala-schiffmann-18}, and their realization as 
continuous colimits of Borcherds--Kac--Moody algebras.

\subsection{Vertex space}\label{ss:vertex-space}

\begin{definition}\label{def:topological-quiver}
	Let $X$ be a Hausdorff topological space. We say that $X$ is a \emph{vertex space} 
	if for any $x\in X$, there exists a \emph{chart $(U, A, \phi)$ around $x$} such that
	\begin{enumerate}\itemsep0.15cm
		\item $U$ is an open neighborhood of $x$,
		\item $A=\{A_i\}$ is a family of closed subsets $A_i\subseteq U$ containing $x$, 
		such that $U=\bigcup_i\, A_i$, 
		\item $\phi=\{\phi_i\}$ is a family of continuous maps $\phi_i\colon A_i\to \R$ which 
		are homeomorphisms onto open intervals of $\R$, such that if the intersection between 
		$A_i$ and $A_j$ strictly contains the point $x$, then $\phi_i\vert_{A_i\cap A_j}=\phi_j
		\vert_{A_i\cap A_j}$ and $\phi_i\vert_{A_i\cap A_j}$ induces a homeomorphism between 
		$A_i\cap A_j$ and a closed interval of $\R$.
	\end{enumerate}
	We say that $x$ is an \emph{regular point} if the exist a chart such that $A=\{U\}$; while, 
	we say that $x$ is a \emph{critical point} if there exists a chart such that the boundary 
	$ \partial(A_i\cap A_j)$ of $A_i\cap A_j$, as a subset of $U$, contains $x$ for any $i,j$.\footnote{Here, somehow we are following the terminology coming from the theory of persistence modules (cf.\ \cite[Section~2.3]{dufresne_sampling}.}
\end{definition}

\begin{remark}
	Let $x$ be a critical point with a chart $(U, A, \phi)$ such that $x\in \partial(A_i\cap A_j)$ 
	for any $i,j$. Then $x\in \partial A_i$ for any $i$.
\end{remark}

\begin{definition}\label{def:topological-quiver-II}
	Let $X$ be a vertex space and let $\ia$ be a subset of $X$. 
	We say that $\ia$ is an \emph{elementary interval} if there exists a chart $(U, A, \phi)$ for which 
	$J\subset A_i$ for some $i$ and $\phi_i(\ia)$ is a open-closed interval of $\R$.
	A sequence of elementary intervals $(\ia_1 ,\dots, \ia_n)$, $n>0$, is \emph{admissible} if
	\begin{itemize}\itemsep0.2cm
		\item[(a)] $(\ia_1 \cup\cdots \cup \ia_{i})\cap \ia_{i+1}=\emptyset$ and $(\ia_1 \cup\cdots\cup \ia_{i})\cup 
		\ia_{i+1}$ is connected for any $i=1,\dots, n-1$;
		\item[(b)] for any $i=1,\dots, n-1$, there exist $x\in X$ and a chart $(U,A, \phi)$ around $x$ for 
		which $U\supseteq(\ia_1 \cup\cdots \cup \ia_{i})\cup \ia_{i+1}$ and 
		$\big(\negthinspace(\ia_1 \cup\cdots \cup \ia_{i})\cup \ia_{i+1}\big)\cap A_k$ is either empty or an elementary interval 
		for any $k$. 
	\end{itemize}
	An \emph{interval} of $X$ is a subset $\ia$ of the form $\ia_1 \cup\cdots\cup \ia_n$, where $(\ia_1 ,\dots, \ia_n)$ 
	is an admissible sequence of elementary intervals.
	We denote by $\intsf(X)$ the set of all intervals in $X$.
\end{definition}

\begin{example}
	Let $\K=\Q, \R$. Then $\K$ is an example of a vertex space. An interval of $\K$ is a subset $\ia\subset \R$ which is an an open--closed 
	interval of the form $\ia=(a,b]\coloneqq \{x\in \R\;\vert\; a<x\leqslant b\}$ for some $\K$-values $a<b$. 
\end{example}

\subsection{Continuum quivers}\label{ss:topological-quiver}

Let $X$ be a vertex space and $\intsf(X)$ the set of all intervals of $X$. Set
\begin{align*}
	\ia\sgpp\ib\coloneqq &
	\begin{cases}
		\ia\cup\ib & \text{if } \ia\cap\ib=\emptyset\text{ and }\ia\cup\ib\in\intsf(X)\ ,\\[4pt]
		\text{n.d.} & \text{otherwise}\ ,
	\end{cases}\\
	\ia\sgpm\ib\coloneqq &
	\begin{cases}
		\ia\setminus\ib & \text{if } \ia\cap\ib=\ib\text{ and }\ia\setminus \ib\in \intsf(X)\ ,\\[4pt]
		\text{n.d.} & \text{otherwise}\ .
	\end{cases}
\end{align*}
We call $\sgpp$ the \emph{sum of intervals}, while we call $\sgpm$ the \emph{difference of intervals}. 

\begin{remark}\label{rem:int-X}
	The elements of $\intsf(X)$ are described as follows
	\cite[Lemma~5.5]{appel-sala-schiffmann-18}.
	\begin{enumerate}\itemsep0.2cm
		\item Every contractible interval is homeomorphic to a finite oriented tree such that any
		vertex is the target of at most one edge.
		\item Every non--contractible interval is homeomorphic to an interval of the form 
		\begin{align*}
			S^1\sgpp \bigoplus_{k=1}^N T_k\coloneqq (\cdots(S^1\sgpp T_1)\sgpp T_2)\cdots \sgpp T_N)
		\end{align*}
		for some pairwise disjoint contractible intervals $T_k$, with $N\geqslant 0$.
	\end{enumerate}
\end{remark}

We denote by $\fun{X}$ the $\Z$-span of the characteristic functions $\cf{\ia} $ for all interval $\ia$ of $X$. 
Note that $\cf{\ia\sgpp\ib}=\cf{\ia} +\cf{\ib} $ for a given $(\ia,\ib)\in \intsf(X)^{(2)}_{\sgpp}$. We call \emph{support} 
of a function $f\in \fun{X}$ the set $\mathsf{supp}(f)\coloneqq \{p\in X\, \vert\, f(p)\neq 0\}$. It is a disjoint 
union of finitely many intervals of $X$.

Define a bilinear form $\abf{\cdot}{\cdot}$ on $\fun{X}$ in the following way. Let $f,g\in\fun{X}$, and 
assume that there exists a point $x$ with a chart $(U, A, \phi)$ for which the supports of $f$ and $g$ 
are contained in $A_i$ for some $i$, then we set
\begin{align}\label{eq:Euler-form}
	\abf{f}{g}\coloneqq \sum_{p\in A_i} f_-(p)(g_-(p)-g_+(p)) \ ,
\end{align}
where $h_{\pm}(x)=\lim_{t\to0^+}h(x\pm t)$.

Since we can always decompose an interval into a sum of elementary subintervals (and we can do similarly 
with supports of functions of $\fun{X}$), we extend $\abf{\cdot}{\cdot}$ with respect to $\sgpp$ by 
imposing the condition that $\abfcf{\ia} {\ib} =0$ for two elementary intervals $\ia, \ib$ for which there does not 
exist a common $A_i$ containing both.

As a consequence of the definition, the bilinear form $\abf{\cdot}{\cdot}$ is compatible with the concatenation 
of intervals, by Remark~\ref{rem:int-X}, it is entirely determined by its values on contractible elements.

\begin{remark}\label{rem:euler-form-identities}
	Thanks to the condition (b) of Definition~\ref{def:topological-quiver-II}, one can easily verify that 
	if $\ib$ is a non--contractible sub--interval of $\ia$, then $\abfcf{\ia} {\ib} =\abfcf{\ia\sgpm\ib}{\ib} $, whenever $\ia\sgpm\ib$ is defined. 
	
	Moreover, whenever $\ia\perp\ib$, \ie $(\ia,\ib)\not\in\intsf(X)^{(2)}_{\sgpp}$ and $\ia\cap\ib=\emptyset$, then $\abfcf{\ia}{\ib} =0$. Note also that 
	\begin{align*}
		\abfcf{\ia} {\ia} =
		\begin{cases}
			1 & \text{if $\ia$ is contractible}\ ,\\
			0 & \text{otherwise}\ .
		\end{cases}
	\end{align*}
\end{remark}

Set 
\begin{align*}
	\rbf{f}{g} \coloneqq \abf{f}{g}+\abf{g}{f}
\end{align*}
for $f,g\in\fun{X}$. Then, if $J,J'\in\intsf(X)$ are contractible, then  
\begin{align*}
	\rbf{\cf{\ia} }{\cf{\ib} }=
	\begin{cases}
		{\phantom{+}}2 & \text{if } \ia=\ib \ ,\\[4pt]
		{\phantom{+}}1 & \text{if } (\ia,\ib)\in\intsf(X)^{(2)}_{\sgpm} \text{ or } (\ib,\ia)\in\intsf(X)^{(2)}_{\sgpm}\ ,\\[4pt]
		{\phantom{+}}0 & \text{if } (\ia,\ib)\not\in\intsf(X)^{(2)}_{\sgpp} \text{ and } \ia\cap\ib=\emptyset, \\[4pt]
		-1 & \text{if } (\ia,\ib)\in\intsf(X)^{(2)}_{\sgpp} \text{ and } \ia\sgpp\ib \text{ is contractible}\ ,\\[4pt]
		-2 & \text{if } (\ia,\ib)\in\intsf(X)^{(2)}_{\sgpp} \text{ and } \ia\sgpp\ib \text{ is non--contractible}\ .
	\end{cases}
\end{align*}
All other cases follow therein. Note in particular that, if $\ia$ is non--contractible, $\rbfcf{\ia}{\ia} =0$.

Henceforth, we set $\abf{\ia} {\ib} \coloneqq\abfcf{\ia} {\ib} $ and $\rbf{\ia}{\ib}\coloneqq \rbf{\cf{\ia} }{\cf{\ib} }$. 
It follows immediately from Remark~\ref{rem:euler-form-identities} that 
\begin{align*}
	\rbf{\alpha}{\alpha}=
	\begin{cases}
		2 & \text{if } \ia \text{ is contractible}\ ,\\
		0 &  \text{if } \ia \text{ is non--contractible}\ .
	\end{cases}
\end{align*}
Therefore, we will use \emph{real} (resp. \emph{imaginary}) as a synonym of contractible (resp. non--contractible) in analogy with the terminology used for the roots of a Kac--Moody algebra.

Finally, we give the following:
\begin{definition}
	Let $X$ be a space of vertices. The \emph{continuum quiver of $X$} is the 
	datum $\calQ_X\coloneqq(\intsf(X), \sgpp,\sgpm,\abf{\cdot}{\cdot}, \rbf{\cdot}{\cdot})$.
\end{definition}

\subsection{Continuum Kac--Moody algebras}\label{ss:cont-km}

It is well--known that the set $\sfR_+$ of positive roots of a Kac--Moody algebra $\g$ 
has a standard structure of \emph{partial semigroup}, induced by its embedding in the
root lattice $\sfQ_+$, and that, as Lie bialgebras, the positive and negative Borel subalgebras $\b_{\pm}$
are graded over $\sfR_+$ (cf.\ \cite[Sec. 8]{appel-toledano-15}). Roughly speaking,  
continuum Kac--Moody algebras are obtained by replacing the semigroup of the positive 
roots with the continuum quiver $\calQ_X$. Namely, to any continuum quiver $\calQ_X$, we 
associate a Lie algebra $\g_X $, whose definition mimics the construction of Kac--Moody algebras. 
Let $\wt{\g}_X$ be the Lie algebra over $\C$, freely generated by $\fun{X}$ and the elements 
$\xpm{\ia} $, $\ia\in\intsf(X)$, 
subject to the relations:
\begin{align*}
	[\xz{\ia},\xz{\ib} ]=0\ ,\qquad
	[\xz{\ia} ,\xpm{\ib} ]=\pm\rbf{\ia}{\ib}\cdot \xpm{\ib} \ ,
\end{align*}
\begin{align*}
	[\xp{\ia} ,\xm{\ib} ]=\drc{\ia\ib}\xz{\ia} +(-1)^{\abf{\ia} {\ib} }\cdot \rbf{\ia}{\ib}\cdot (\xp{\ia\sgpm \ib}-\xm{\ib \sgpm \ia})\ .
\end{align*}
where $\xz{\ia} \coloneqq\cf{\ia}$.

Note that the relation $\xz{\ia\sgpp\ib}=\drc{\ia\sgpp\ib}\left(\xz{\ia} +\xz{\ib}\right)$ holds by definition. Set 
\begin{align*}
	\fun{X}^+\coloneqq\mathsf{span}_{\Z_{\geqslant0}}\{\cf{\ia} \;|\;\ia\in\intsf(X)\}\ .
\end{align*}
There is a natural $\fun{X}$--gradation on $\wt{\g}_X$ given by $\deg(\xpm{\ia})=\pm\cf{\ia}$
and $\deg(\xz{\ia})=0$, inducing a triangular decomposition
\begin{align*}
	\wt{\g}_X=\left(\bigoplus_{\phi\in\fun{X}^+}\wt{\g}_{+\phi}\right)\oplus\fun{X}\oplus\left(\bigoplus_{\phi\in\fun{X}^+}\wt{\g}_{-\phi}\right)\ .
\end{align*}
where $\wt{\g}_{\pm\phi}$ denotes the homogeneous subspace of degree $\pm\phi$.

It is important to observe that the bilinear form $\rbf{\cdot}{\cdot}$ on $\fun{X}$ is non--degenerate unless $X=S^1$,
in which case, $\ker\rbf{\cdot}{\cdot}=\Z\cdot\cf{S^1}$. Therefore, whenever $X\neq S^1$, the homogeneous spaces 
$\wt{\g}_{\pm\phi}$ coincide with weight spaces corresponding to the diagonal action of $\fun{X}$. That is, we have
\begin{align*}
	\wt{\g}_{\pm\phi}=\{x\in\wt{\g}_X\;|\; \forall\psi\in\fun{X}\;\vert\; [\psi, x]=\pm\rbf{\phi}{\psi}\cdot x\}\ ,
\end{align*}
for $\phi\in\fun{X}^+$. 

\begin{definition}
	The \emph{continuum Kac--Moody algebra of $\calQ_X$} is the Lie algebra $\g_X \coloneqq\wt{\g}_X/\r_X$, where $\r_X\subset\wt{\g}_X$ is 
	the sum of all two--sided graded ideals with trivial intersection with $\fun{X}$.
\end{definition}

In particular, $\g_X $ has a triangular decomposition
\begin{align}\label{eq:triangular-dec}
	\g_X =\n_+\oplus\h\oplus\n_-\ ,
\end{align}
where $\h=\fun{X}$ and $\n_{\pm}$ are the Lie subalgebras generated, respectively, by 
the elements $\xpm{\ia}$, $\ia\in\intsf(X)$.

The main result of \cite{appel-sala-schiffmann-18} is a generalization to the case of $\g_X $ of the results of Gabber--Kac \cite{gabber-kac-81} and 
Borcherds \cite{borcherds-88}, showing that the ideal $\r_X$ is generated by the Serre relations. In particular, this gives an explicit description of 
the Lie algebra $\g_X $ as follows. 

\begin{definition}
	Let $\serre{X}$ be the set of all pairs $(\ia,\ib)\in\intsf(X)\times\intsf(X)$ such that one of the
	following occurs:
	\begin{itemize}\itemsep0.3cm
		\item
		$\ia$ is contractible, does not contain any critical point of $\beta$, and, for subintervals $\ia'\subseteq \ia$ and $\ib'\subseteq\ib$ with $\rbf{\ib}{\ib'}\neq0$ whenever $\ib'\neq\ib$, $\ia'\sgpp\ib'$ is either undefined or 
		non--homeomorphic to $S^1$;
		\item $\ia\perp\ib$, \ie $\ia\sgpp\ib$ does not exist and $\ia\cap\ib=\emptyset$.
	\end{itemize}
\end{definition}

\begin{ex}
	One has $\serre{\R}=\intsf(\R)\times \intsf(\R)$ and $\serre{S^1}=\big(\intsf(S^1)\setminus \{S^1\}\big)\times \intsf(S^1)$. \hfill $\triangle$
\end{ex}

Set 
\begin{align}\label{eq:ca-cb}
	\ca{\ia}{\ib}\coloneqq(-1)^{\abf{{\ia} }{{\ib} }}\cdot\rbf{\ia}{\ib}\quad
	\text{and}\quad \cb{\ia}{\ib}\coloneqq\ca{\ia}{,\, \ia\sgpp\ib}	\ .
\end{align}
Note that, if $\ia\sgpm\ib$ or $\ib\sgpm\ia$ are defined, then $\ca{\ia}{\ib}\in\{0,\pm1\}$, 
and, if $\ia\sgpp\ib$ is defined and $(\ia,\ib)\in\serre{X}$, then $\cb{\ia}{\ib}\in\{\pm1\}$.

\begin{theorem}[{cf.\ \cite[Theorem~5.16]{appel-sala-schiffmann-18}}]\label{thm:ass-18}
	The continuum Kac--Moody algebra $\g_X $ is freely generated by the abelian Lie algebra $\fun{X}$ and 
	the elements $\xpm{\ia} $, $\ia\in \intsf(X)$, subject to the following defining relations: 
	\begin{enumerate}\itemsep0.2cm
		\item {\bf Diagonal action:} for $\ia, \ib\in\intsf(X)$, 
		\begin{align*}
			[\xz{\ia} , \xpm{\ib} ] =\pm \rbf{\ia} {\ib} \cdot\, \xpm{\ib} \ ;
		\end{align*}
		\item {\bf Double relations:} for $\ia, \ib\in\intsf(X)$, 
		\begin{align*}
			[\xp{\ia} ,\xm{\ib} ]=\drc{\ia\ib}\, \xz{\ia} +\ca{\ia}{\ib}\cdot\left(\xp{\ia\sgpm\ib}-\xm{\ib\sgpm\ia}\right) \ ;
		\end{align*} 
		\item {\bf Serre relations:} for $(\ia,\ib)\in \serre{X}$,
		\begin{align*}
			[\xpm{\ia} , \xpm{\ib} ]=  \pm\cb{\ia}{\ib}\cdot \xpm{\ia\sgpp\ib} \ .
		\end{align*}
	\end{enumerate}
\end{theorem}

\begin{remark}\label{rem:serre-TQ}
	If $\ib\simeq S^1$ and $\ia\subseteq\ib$, then $(\ia,\ib)\in \serre{X}$. Hence, by (2) above $[\xpm{\ia} , \xpm{\ib} ]=0$.
\end{remark}

\subsection{Colimit realization}\label{ss:bkm-and-sgp}

One fundamental ingredient in the proof of Theorem~\ref{thm:ass-18} is the relation
between $\g_X $ and certain Borcherds--Kac--Moody algebras naturally arising from families
of intervals.
Let $\calJ=\{\ia_k\}_{k}$ be a finite set of intervals $\ia_k\in\intsf(X)$. We say
that $\calJ$ is \emph{irreducible} if the following conditions hold: 
\begin{enumerate}\itemsep0.2cm
	\item every interval is either contractible or homeomorphic to $S^1$; 
	\item given two intervals $\ia,\ib\in\calJ$, $\ia\neq\ib$, one of the following mutually exclusive cases occurs:
	\begin{itemize}\itemsep0.2cm
		\item[(a)] $\ia\sgpp\ib$ exists;
		\item[(b)] $\ia\sgpp\ib$ does not exist and $\ia\cap\ib=\emptyset$;
		\item[(c)] $\ia\simeq S^1$ and $\ib\subset\ia$\ .
	\end{itemize}
\end{enumerate}
Assume henceforth that $\calJ$ is an irreducible set of intervals.
Let $\sfA_{\calJ}$ be the matrix given by the values of $\rbf{\cdot}{\cdot}$ on $\calJ$, 
\ie $\big(\sfA_{\calJ}\big)_{\ia\ib}=\rbf{\ia}{\ib}$ for $\ia,\ib\in \calJ$. Note that the diagonal
entries of $\sfA_{\calJ}$ are either $2$ or $0$, while the only possible
off--diagonal entries are $0,-1,-2$. Let $\calQ_{\calJ}$ be the corresponding quiver with Cartan matrix $\sfA_{\calJ}$. 
Note that a contractible elementary interval in $\calJ$ corresponds to a vertex of $\calQ_{\calJ}$ without 
loops at it. For example, we obtain the following quivers.

\begin{align*}
	\begin{array}{|c|c|c|}
		\hline
		\text{Configuration of intervals} & \text{Borcherds--Cartan diagram}\\
		\hline &\\
		\begin{tikzpicture}[scale=.35]
			\begin{scope}[on background layer]
				\draw [white] (0,0) rectangle (15,10);
			\end{scope}
			\coordinate (C1) at (3, 5);
			\coordinate (C2) at (6, 5);
			\coordinate (C3) at (9, 5);
			\coordinate (C4) at (12, 7.5);
			\coordinate (C5) at (12, 2.5);
			\coordinate (CON1) at (10,5);
			\coordinate (CON2) at (12,5);
			\draw [->,very thick, purple]	(C1) -- (C2);
			\draw [->,very thick, blue]		 (C2) -- (C3);
			\draw [->,very thick, yellow]      (C3) arc (270:360:3);
			\draw [->,very thick, orange]      (C3) arc (90:0:3);
			\node at (4.5, 6) {$\ia_1 $};
			\node at (7.5, 6) {$\ia_2 $};
			\node at (13, 6.5) {$\ia_3$};
			\node at (13, 3.5) {$\ia_4$};
		\end{tikzpicture}
		&
		\begin{tikzpicture}[scale=.35]
			\begin{scope}[on background layer]
				\draw [white] (0,0) rectangle (15,10);
			\end{scope}
			\node (V1) at (3.5,5)      [circle,draw=purple,fill=purple, inner sep=3pt]    {};
			\node (V2) at (7.5,5)      [circle,draw=blue,fill=blue, inner sep=3pt]          {};
			\node (V3) at (10.5,7.5) [circle,draw=yellow,fill=yellow, inner sep=3pt]   {};
			\node (V4) at (10.5,2.5) [circle,draw=orange,fill=orange, inner sep=3pt]  {};
			\draw [->, very thick] (V1) -- (V2);  
			\draw [<-, very thick] (V3) -- (V2);
			\draw [<-, very thick] (V4) -- (V2);
			\node at (3.5, 6)    {$\alpha_1$};
			\node at (7.5, 6)    {$\alpha_2$};
			\node at (11.5, 7.5)  {$\alpha_3$};
			\node at (11.5, 2.5)  {$\alpha_4$};
		\end{tikzpicture}
		\\
		\hline &\\
		\begin{tikzpicture}[scale=.35]
			\begin{scope}[on background layer]
				\draw [white] (0,0) rectangle (15,10);
			\end{scope}
			\draw [->, very thick, purple] (7.5,7.5) arc (90:270:2.5);
			\draw [->, very thick, blue] (7.5,2.5) arc (-90:90:2.5);
			\node at (4, 5)    {$\ia_1 $};
			\node at (11, 5)    {$\ia_2 $};
		\end{tikzpicture}
		&
		\begin{tikzpicture}[scale=.35]
			\begin{scope}[on background layer]
				\draw [white] (0,0) rectangle (15,10);
			\end{scope}
			\node (V1) at (5,5)      [circle,draw=purple,fill=purple, inner sep=3pt]    {};
			\node (V2) at (10,5)      [circle,draw=blue,fill=blue, inner sep=3pt]          {};
			\coordinate (CON1) at (6.5, 7.5);
			\coordinate (CON2) at (8.5, 7.5);
			\coordinate (CON3) at (6.5, 2.5);
			\coordinate (CON4) at (8.5, 2.5);
			\draw [->, very thick] (V1) .. controls (CON1) and (CON2) .. (V2);  
			\draw [<-, very thick] (V1) .. controls (CON3) and (CON4) .. (V2);  
			\node at (5, 6)    {$\alpha_1$};
			\node at (10.2, 6)    {$\alpha_2$};
		\end{tikzpicture}\\
		\hline
		\begin{tikzpicture}[scale=.35]
			\begin{scope}[on background layer]
				\draw [white] (0,0) rectangle (15,11);
			\end{scope}
			\draw [->, very thick, purple] (7.5,7.5) arc (90:180:2.5);
			\draw [->, very thick, yellow] (5,5) arc (180:270:2.5);
			\draw [->, very thick, blue] (7.5,2.5) arc (270:360:2.5);
			\draw [->, very thick, orange] (10,5) arc (0:90:2.5);
			\draw [<-, very thick, green] (2.5,2.5) arc (270:360:2.5);
			\draw [->, very thick, red] (10,5) arc (180:90:2.5);
			\node at (3.5, 3.5)      {$\ia_1 $};
			\node at (5.5, 7.5)    {$\ia_2 $};
			\node at (5.5, 2.5)    {$\ia_3$};
			\node at (9.5, 7.5)    {$\ia_4$};
			\node at (9.5, 2.5)    {$\ia_5$};
			\node at (11.5, 6.5)      {$\ia_6$};
		\end{tikzpicture}
		&
		\begin{tikzpicture}[scale=.35]
			\begin{scope}[on background layer]
				\draw [white] (0,0) rectangle (15,10);
			\end{scope}
			\node (V1) at (2.5,5)        [circle,draw=green,fill=green, inner sep=3pt]    {};
			\node (V2) at (5,5)      [circle,draw=purple,fill=purple, inner sep=3pt]          {};
			\node (V3) at (7.5,2.5)      [circle,draw=yellow,fill=yellow, inner sep=3pt]          {};
			\node (V4) at (7.5,7.5)      [circle,draw=orange,fill=orange, inner sep=3pt]          {};
			\node (V5) at (10,5)      [circle,draw=blue,fill=blue, inner sep=3pt]          {};
			\node (V6) at (12.5,5)      [circle,draw=red,fill=red, inner sep=3pt]          {};
			\node at (2.5, 6)    {$\alpha_1$};
			\node at (4.8, 6)    {$\alpha_2$};
			\node at (7.5, 1.5)    {$\alpha_3$};
			\node at (7.5, 8.5)    {$\alpha_4$};
			\node at (10, 6)    {$\alpha_5$};
			\node at (12.5, 6)    {$\alpha_6$};
			\draw [<-, very thick] (V1) -- (V2);  
			\draw [->, very thick] (V2) -- (V3);
			\draw [->, very thick] (V3) -- (V5);
			\draw [->, very thick] (V5) -- (V4);
			\draw [->, very thick] (V4)-- (V2);  
			\draw [->, very thick] (V5) -- (V6);
		\end{tikzpicture}
		\\
		\hline
	\end{array}
\end{align*}

Instead, an interval of $\calJ$ homeomorphic to $S^1$ corresponds in $\calQ_{\calJ}$ to 
a vertex having exactly one loop at it, as in the following examples.

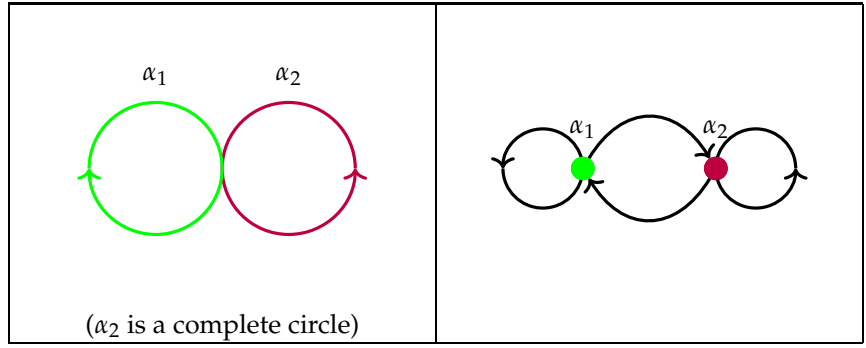
\begin{align*}
	\begin{array}{|c|c|}
		\hline
		\text{Configuration of intervals} & \text{Borcherds--Cartan diagram}\\
		\hline &\\
		\begin{tikzpicture}[scale=.35]
			\begin{scope}[on background layer]
				\draw [white] (0,0) rectangle (15,10);
			\end{scope}
			\draw [->, very thick, purple] (10,5) arc (0:360:2.5);
			\draw [->, very thick, blue] (7.5,7.5) arc (90:180:2.5);
			\draw [<-, very thick, green] (2.5,2.5) arc (270:360:2.5);
			\node at (3.5, 3.5)      {$\ia_1 $};
			\node at (5.5, 7.5)    {$\ia_2 $};
			\node at (11, 5)    {$\ia_3$};
		\end{tikzpicture}
		&
		\begin{tikzpicture}[scale=.35]
			\begin{scope}[on background layer]
				\draw [white] (0,0) rectangle (15,10);
			\end{scope}
			\node (V2) at (7.5, 7.5)  [circle,draw=blue,fill=blue, inner sep=3pt]    {};
			\node (V1) at (5, 5)  [circle,draw=green,fill=green, inner sep=3pt]    {};
			\node (V3) at (10, 5)  [circle,draw=purple,fill=purple, inner sep=3pt]    {};
			\draw [->, very thick] (V2) -- (V1);   
			\draw [->, very thick] (V3) -- (V1);
			\draw [->, very thick] (13,5) arc (0:360:1.5);  
			\node at (10, 5)  [circle,draw=purple,fill=purple, inner sep=3pt]    {};
			\node at (4.8, 6)    {$\alpha_1$};
			\node at (7.5, 8.5)    {$\alpha_2$};
			\node at (9.5, 6)    {$\alpha_3$};
		\end{tikzpicture}\\
		\text{($\ia_3$ is a complete circle)} & \\
		\hline
	\end{array}
\end{align*}
\vspace{-0.5cm}
\begin{align*}
	\begin{array}{|c|c|}
		\hline &\\
		\begin{tikzpicture}[scale=.35]
			\begin{scope}[on background layer]
				\draw [white] (0,0) rectangle (15,10);
			\end{scope}
			\draw [->, very thick, purple] (12.5,5) arc (0:360:2.5);
			\draw [<-, very thick, green] (2.5,5) arc (180:540:2.5);
			\node at (5, 8.5)      {$\ia_1 $};
			\node at (10, 8.5)      {$\ia_2 $};
		\end{tikzpicture}
		&
		\begin{tikzpicture}[scale=.35]
			\begin{scope}[on background layer]
				\draw [white] (0,0) rectangle (15,10);
			\end{scope}
			\node (V1) at (5, 5)  [circle,draw=green,fill=green, inner sep=3pt]    {};
			\node (V3) at (10, 5)  [circle,draw=purple,fill=purple, inner sep=3pt]    {};
			\coordinate (CON1) at (6.5, 7.5);
			\coordinate (CON2) at (8.5, 7.5);
			\coordinate (CON3) at (6.5, 2.5);
			\coordinate (CON4) at (8.5, 2.5);
			\draw [->, very thick] (V1) .. controls (CON1) and (CON2) .. (V3);  
			\draw [<-, very thick] (V1) .. controls (CON3) and (CON4) .. (V3);  
			\draw [->, very thick] (13,5) arc (0:360:1.5);  
			\draw [->, very thick] (2,5) arc (180:540:1.5);  
			\node at (V3) [circle,draw=purple,fill=purple, inner sep=3pt]    {};
			\node at (V1) [circle,draw=green,fill=green, inner sep=3pt]    {};
			\node at (5, 6.5)    {$\alpha_1$};
			\node at (10, 6.5)    {$\alpha_2$};
		\end{tikzpicture}\\
		\text{($\ia_2 $ is a complete circle)} & \\
		\hline
	\end{array}
\end{align*}

To any irreducible set of intervals $\calJ$, we can associate two Lie algebras:
\begin{enumerate}\itemsep0.2cm
	\item the Lie subalgebra $\g_{\calJ}\subset\g_X $ generated by the elements 
	$\{\xpm{\ia},\, \xz{\ia}\;\vert\; \ia\in \calJ\}$;
	\item the derived Borcherds--Kac--Moody algebra $\g_{\calJ}^{\scsop{BKM}}\coloneqq
	\g(\sfA_{\calJ})'$.
\end{enumerate}
We prove in \cite[Section~5.5]{appel-sala-schiffmann-18} that $\g_{\calJ}$ and 
$\g_{\calJ}^{\scsop{BKM}}$ are canonically isomorphic. More precisely, we 
have the following.

\begin{proposition}\label{prop:bkm-sgp}
	The assignment 
	\begin{align*}
		e_{\ia}\mapsto \xp{\ia} \ , \quad f_{\ia}\mapsto \xm{\ia}  \quad \text{and} \quad h_{\ia} \mapsto \xz{\ia}
	\end{align*}
	for any $\ia\in \calJ$, defines an isomorphism of Lie algebras 
	$\Phi_\calJ\colon \g_{\calJ}^{\scsop{BKM}}\to \g_{\calJ}$.
\end{proposition}

The proof relies on the simple observation that, for $\alpha,\beta\in\calJ$,
$\alpha\neq S^1, \beta$,
\begin{align*}
	\sfad(\xpm{\alpha})^{1-\rbf{\alpha}{\beta}}(\xpm{\beta})=0\ .
\end{align*}
It is then clear that $\g_X $ can be constructed exclusively 
from Borcherds--Kac--Moody algebras. That is, we have
the following.

\begin{corollary}\label{cor:bkm-sgp}
	Let $\calJ,\calJ'$ be two irreducible (finite) sets of intervals in $X$.
	\begin{enumerate}\itemsep0.2cm
		\item If $\calJ'\subseteq\calJ$, there is a canonical embedding
		$\phi'_{\calJ,\calJ'}\colon\g_{\calJ'}\to\g_{\calJ}$ sending $\xpm{\ia} \mapsto \xpm{\ia}$, 
		$\xz{\ia}\to\xz{\ia}$ for $\ia\in\calJ'$.
		\item If $\calJ$ is obtained from $\calJ'$ by replacing an element $\ic\in\calJ'$
		with two intervals $\ia ,\ib $ such that $\ic=\ia \sgpp \ib $,
		there is a canonical embedding $\phi''_{\calJ,\calJ'}\colon\g_{\calJ'}\to\g_{\calJ}$, which is the identity on $\g_{\calJ'\setminus\{\ic\}}=\g_{\calJ\setminus\{\ia,\ib\}}$
		and sends
		\begin{align*}
			\xz{\ic} \mapsto \xz{\ia } +\xz{\ib } \ , \quad \xpm{\ic}  \mapsto \pm(-1)^{\abf{\ib } {\ia } }\, [\xp{\ia } ,\xp{\ib } ] \ .
		\end{align*}
		\item The collection of embeddings $\phi'_{\calJ,\calJ'}, \phi''_{\calJ,\calJ'}$, indexed by all possible irreducible sets of intervals
		in $X$, form a direct system. Moreover,
		\begin{align*}
			\g_X\simeq\colim_{\calJ}\, \g_{\calJ}^{\scsop{BKM}}  \ .
		\end{align*}
	\end{enumerate}
\end{corollary}

\subsection{The Lie algebras of the line and of the circle}\label{ss:semigroup-line}

In this section we recall the defining relations of the Lie algebras of the line and the circle, 
$\sl(\K)$ and $\sl(\K/\Z)$ with $\K=\Q,\R$, introduced in \cite{sala-schiffmann-17}, and their 
realizations as continuum Kac--Moody algebras. Indeed, these examples were the stepping 
stones for the definition of continuum Kac--Moody algebras.

First, we need to distinguish all relative positions of two intervals. For any two 
intervals $\ia=(a,b]$ and $\ib=(a',b']$, we write
\begin{itemize}\itemsep0.3cm
	\item $\ia\intnext\ib$ if $b=a'$ (\emph{adjacent}) 
	\item $\ia\perp\ib$ if $b<a'$ or $b'<a$ (\emph{disjoint and non--adjacent})
	\item $\ia\rsub\ib$ if $a=a'$ and $b<b'$ (\emph{closed subinterval})
	\item $\ia\lsub\ib$ if $a'<a$ and $b=b'$ (\emph{open subinterval})
	\footnote{The symbol $\rsub$ (resp. $\lsub$) should be read as \emph{$\ia$ is a proper subinterval in 
			$\ib$ starting from the left (resp. right) endpoint}.}
	\item $\ia<\ib$ if $a'<a<b<b'$ (\emph{strict subinterval})
	\item $\ia\intcap\ib$ if $a<a'<b<b'$ (\emph{overlapping})
\end{itemize}

We shall say that $\ia$ and $\ib$ are \emph{orthogonal} if $\ia\perp\ib$.
We are ready to give the definition of $\sl(\K)$.
\begin{definition}\label{def:sl(K)}
	Let $\sl(\K)$ be the Lie algebra generated by elements $e_{\ia}, f_{\ia}, h_{\ia}$, with ${\ia}\in\intsf(\K)$, 
	modulo the following set of relations:
	\begin{itemize}\itemsep0.3cm
		\item {\bf Kac--Moody type relations:} for any two intervals $\ia, \ib$,
		\begin{align*}
			[h_{\ia} ,h_{\ib} ]=0\ , \qquad [h_{\ia} , e_{\ib} ]=\rbfcf{{\ia}}{{\ib}}\, e_{\ib} \ , \qquad [h_{\ia} , f_{\ib} ] =-\rbfcf{{\ia}}{{\ib}}\, f_{\ib} \ , 
		\end{align*}
		\begin{align*}
			[e_{\ia} ,f_{\ib} ]=
			\begin{cases}
				h_{\ia}  & \text{if }  \ia=\ib\ ,\\[3pt]
				0 & \text{if }  \ia\perp \ib, \ia\intnext \ib, \text{ or } \ib\intnext \ia\ ,
			\end{cases}
		\end{align*}
		
		\item {\bf join relations:} for any two intervals $\ia , \ib $ with $(\ia ,\ib )\in\intsf(\K)^{(2)}_{\sgpp}$,
		\begin{align*}
			h_{\ia \sgpp \ib }=h_{\ia} + h_{\ib} \ ,\\
			e_{\ia \sgpp \ib }=(-1)^{\abfcf{{\ib } }{{\ia } }}[e_{\ia } , e_{\ib } ]\ , \qquad
			f_{\ia \sgpp \ib }&= (-1)^{\abfcf{{\ia } }{{\ib } }} [f_{\ia } , f_{\ib } ]\ ;
		\end{align*}
		
		\item {\bf nest relations:} for any nested $\ia ,\ib \in\intsf(\K)$ (that is, such that	$\ia =\ib $, 
		$\ia \perp \ib $, $\ia <\ib $, $\ib <\ia $, $\ia \rsub \ib $, $\ia \lsub \ib $, $\ib \rsub \ia $, or $\ib \lsub \ia $),
		\begin{align*}
			[e_{\ia } ,e_{\ib } ]=0\quad\text{and}\quad [f_{\ia } , f_{\ib } ]=0\ .
		\end{align*}
		
	\end{itemize}
\end{definition}

\begin{remark}
	It is easy to check that the bracket is anti--symmetric and satisfies the Jacobi identity.
	Note that the joint relations are consistent with anti--symmetry, since, whenever $\ia\sgpp \ib$ is defined, 
	$(-1)^{\abfcf{\ia} {\ib} }=-(-1)^{\abfcf{\ib} {\ia} }$. Moreover, the combination of join and nest relations yields 
	the \emph{(type $\sfA$) Serre relations}  ($\ia\neq\ib$)
	\begin{align}\label{eq:serre-sl(K)}
		\begin{aligned}
			\begin{array}{ll}
				[e_{\ia}, [e_{\ia}, e_{\ib} ]]=0=[f_{\ia}, [f_{\ia}, f_{\ib} ]] & \text{if } \rbfcf{\ia}{\ib} =-1 \ ,  \\[8pt]
				[e_{\ia}, e_{\ib} ]=0=[f_{\ia}, f_{\ib} ] & \text{if } \rbfcf{\ia}{\ib} =0\ .
			\end{array}
		\end{aligned}
	\end{align}
	Let  $\sl(\Z)$ be the subalgebra generated by the elements $e_{\ia}, h_{\ia}, f_{\ia}$ for $\ia$ of the form $(i, i+1]$, $i\in\Z$. Then it is clear that $\sl(\Z)\simeq\sl(\infty)$ and there are 
	canonical strict embeddings $\sl(\Z)\subset\sl(\Q)\subset\sl(\R)$.
\end{remark}

First, note that the Cartan subalgebra of $\sl(\K)$,
$\h\coloneqq\langle h_{\ia}\;|\;\ia\in\intsf(\K)\rangle$, is canonically isomorphic, as a Lie algebra, to the commutative algebra 
$\fun{\K}$ generated by the characteristic functions $\{\xz{\ia} \coloneqq\cf{\ia}\;|\;\ia\in\intsf(\K)\}$. In \cite[Corollary~2.10]{appel-sala-schiffmann-18}, we show that the set of relations satisfied by the generators of $\sl(\K)$ can be simplified, indeed we have:

\begin{proposition}\label{prop:slR-presentation}
	The Lie algebra $\sl(\K)$ is isomorphic to $\g_\K$.
\end{proposition}

Let us now move to the Lie algebra of the circle.

\begin{definition}\label{def:sl(K/Z)}
	Let $\sl(\K/\Z)$ be the Lie algebra generated by elements $e_{\ia}, f_{\ia}, h_{\ib}$, with ${\ia, \ib}\in\intsf(\K/\Z)$ and $\ia\neq S^1$, 
	modulo the following set of relations:
	\begin{itemize}\itemsep0.3cm
		\item {\bf Kac--Moody type relations:} for any two intervals $\ia, \ib$,
		\begin{align*}
			[h_{\ia} ,h_{\ib} ]=0\ , \qquad [h_{\ia} , e_{\ib} ]=\rbfcf{{\ia}}{{\ib}}\, e_{\ib} \ , \qquad [h_{\ia} , f_{\ib} ] =-\rbfcf{{\ia}}{{\ib}}\, f_{\ib} \ , 
		\end{align*}
		\begin{align*}
			[e_{\ia} ,f_{\ib} ]=
			\begin{cases}
				h_{\ia}  & \text{if }  \ia=\ib\ ,\\[3pt]
				0 & \text{if }  \ia\perp \ib, \ia\intnext \ib, \text{ or } \ib\intnext \ia\ ,
			\end{cases}
		\end{align*}
		
		\item {\bf join relations:} 
		\begin{itemize}
			\item for any two intervals $\ia , \ib $ with $(\ia ,\ib )\in\intsf(\K/\Z)^{(2)}_{\sgpp}$, we have $h_{\ia \sgpp \ib }=h_{\ia} + h_{\ib}$;
			\item for any two intervals $\ia , \ib $ with $(\ia ,\ib )\in\intsf(\K/\Z)^{(2)}_{\sgpp}$, such that $\ia,\ib, \ia\sgpp\ib\neq S^1$,
			\begin{align*}
				e_{\ia \sgpp \ib }=(-1)^{\abfcf{{\ib } }{{\ia } }}[e_{\ia } , e_{\ib } ]\ , \qquad
				f_{\ia \sgpp \ib }&= (-1)^{\abfcf{{\ia } }{{\ib } }} [f_{\ia } , f_{\ib } ]\ ;
			\end{align*}
			
		\end{itemize}
		
		\item {\bf nest relations:} for any nested $\ia ,\ib \in\intsf(\K/\Z)$ (that is, such that	$\ia =\ib $, 
		$\ia \perp \ib $, $\ia <\ib $, $\ib <\ia $, $\ia \rsub \ib $, $\ia \lsub \ib $, $\ib \rsub \ia $, or $\ib \lsub \ia $), with $\ia, \ib\neq S^1$,
		\begin{align*}
			[e_{\ia } ,e_{\ib } ]=0\quad\text{and}\quad [f_{\ia } , f_{\ib } ]=0\ .
		\end{align*}
	\end{itemize}
\end{definition}

The continuum Kac--Moody algebra $\g_{S^1}$ strictly contains the Lie algebra 
$\sl(S^1)$ and their difference is reduced to the elements $\xpm{S^1}$ corresponding to the entire space.
More precisely, let $\ol{\g}_{S^1}$ be the subalgebra in $\g_{S^1}$ generated by the elements 
$\xpm{\ia}$, $\xz{\ia}$, $\ia\neq S^1$. Note that the elements $\xpm{S^1}$, $\xz{S^1}$,
generate a Heisenberg Lie algebra of order one in $\g_{S^1}$, which we denote $\heis_{S^1}$. Then, 
$\g_{S^1}=\ol{\g}_{S^1}\oplus\heis_{S^1}$ and there is a canonical embedding $\sl(S^1)\to\g_{S^1}$, whose image is 
$\ol{\g}_{S^1}\oplus\bsfld\cdot \xz{S^1}$.

\bigskip\section{The classical continuum $r$--matrix}\label{s:cont-km-lba}

In this section, we show that continuum Kac--Moody algebras are
naturally endowed with a standard \emph{topological} quasi--triangular 
Lie bialgebra structure. To this end, we provide here an alternative 
construction of continuum Kac--Moody algebras \emph{by duality}
in the spirit of \cite{halbout-99}.

Note that the results of this section rely on the non--degeneracy of 
the Euler form on $\fun{X}$, which is automatic whenever $X\not\simeq S^1$. 
If $X=S^1$, the kernel of the Euler form is one--dimensional, spanned by
the central element $\xz{S^1}$. However, this can be easily corrected by 
extending the vector space $\fun{S^1}$ with a derivation corresponding 
to the Heisenberg Lie algebras $\heis_{S^1}$.
\footnote{In other words, we need to consider the canonical realization of
	the Cartan matrix $[0]$ (cf.\ Section~\ref{ss:km}).}
Henceforth, we will therefore assume that the Euler form is non--degenerate 
on $\fun{X}$ for any vertex space $X$.

\subsection{Continuum free Lie algebras}\label{ss:topological cobracket}

Let $\FL{\pm}$ be the free Lie algebras generated,
respectively, by the sets $V_{\pm}=\{\xpm{\alpha},\xzpm{\alpha}\;|\;\alpha\in\intsf(X)\}$.
Let $\cf{\alpha}$ be the characteristic function corresponding to the interval $\alpha\in\intsf(X)$,
and $\sfF\coloneqq\fun{X}^+=\mathsf{span}_{\Z_{\geqslant 0}}\{\cf{\ia} \;|\;\ia\in\intsf(X)\}$.
We consider on $\FL{\pm}$ the natural grading over $\sfF$ given by
$\deg(\xpm{\alpha})=\cf{\alpha}$ and $\deg(\xzpm{\alpha})=0$, thus
\begin{align*}
	\FL{\pm}=\bigoplus_{\phi\in\sfF}\FL{\pm,\,\phi}\ .
\end{align*}

\begin{example}
	Let $\alpha,\beta,\gamma\in\intsf(X)$ such that $\alpha=\beta\sgpp\gamma$. Then, 
	the elements $\xpm{\alpha}, [\xpm{\beta},\xpm{\gamma}]$, and
	$[[[\xpm{\beta},\xzp{\gamma}],\xpm{\gamma}], \xzpm{\alpha}]]$ have degree $\cf{\alpha}$.
\end{example}

For $N>0$ and $\phi\in\sfF$, a $N$-th partition of $\phi$ is a tuple 
$\ul{\psi}=(\psi_1,\dots,\psi_N)\in\sfF^N$ such that $\psi_1+\cdots+\psi_N=\phi$. 
We denote the set of $N$th partitions of $\phi$  by $\sfF(\phi, N)$.
Then, we set
\begin{align*}
	\FL{\pm,\, N}^{\odot}\coloneqq\bigoplus_{\phi\in\sfF}
	\left(
	\prod_{\ul{\psi}\in\sfF(\phi,\, N)}
	\FL{\pm,\, \psi_1}\odot\cdots \odot\FL{\pm,\, \psi_N}
	\right)\ ,
\end{align*}
where $\odot=\ten, \wedge$.
We regard $\FL{\pm,\, N}^{\ten}$ (resp. $\FL{\pm,\, N}^{\wedge}$) as a completion 
of $\FL{\pm}^{\ten N}$ (resp. $\wedge^N\FL{\pm}$).
The following is a straightforward generalization of \cite[Propositions~2.2, 2.3, and 2.5]{halbout-99}.

\begin{proposition}\label{prop:halb}\hfill
	\begin{enumerate}[leftmargin=0.7cm]\itemsep0.3cm
		\item For any collection of antisymmetric elements 
		$u^{\pm}\colon \intsf(X)\to\FL{\pm}^{\wedge, 2}$,
		there exist unique maps $\delta_{\pm}\colon \FL{\pm}\to\FL{\pm}^{\wedge, 2}$ such that 
		\begin{align*}
			\delta_{\pm}(\xpm{\alpha})=u^{\pm}_{\alpha}\quad\mbox{and}\quad
			\delta_{\pm}(\xzpm{\alpha})=0 
		\end{align*}
		and the cocycle condition \eqref{eq:cocycle} holds. Moreover, if the co--Jacobi identity 
		\eqref{eq:coJacobi} holds for the generators $\xpm{\alpha}$, \ie 
		\begin{align}\label{eq:u-cond-1}
			(\id+(1\,2\,3)+(1\,3\,2))\circ\id\ten\delta_{\pm}(u_{\alpha}^{\pm})=0\ ,
		\end{align}
		then $(\FL{\pm}, [\cdot,\cdot],\delta_{\pm})$ are topological Lie bialgebras.
		\item Fix two matrices $\kappa_{i}\colon \intsf(X)\times\intsf(X)\to\bsfld$, $i=0,1$, and let 
		$\bfV_{\pm}\subset\FL{\pm}$ be the subspace spanned by the set $V_{\pm}$.
		Assume that the elements $u^{\pm}_{\alpha}$ satisfy the condition \eqref{eq:u-cond-1}, so that  $(\FL{\pm}, [\cdot,\cdot],\delta_{\pm})$ are topological Lie bialgebras, and
		belong to $\bfV_{\pm}^{\wedge, 2}$. Then, there exists a 
		unique pairing of Lie bialgebras $\abf{\cdot}{\cdot}\colon\FL{+}\ten\FL{-}\to\bsfld$ 
		such that 
		\begin{align*}
			\begin{array}{|c|c|c|}
				\hline
				\abf{\cdot}{\cdot} & \xzm{\gamma} & \xm{\delta}\\
				\hline
				\xzp{\alpha} & \kappa_{0}(\alpha,\gamma) & 0\\
				\hline
				\xp{\beta} & 0 & \kappa_{1}(\beta,\delta)\\
				\hline
			\end{array}
		\end{align*}
		\item Fix two matrices $\kappa_{i}\colon\intsf(X)\times\intsf(X)\to\bsfld$, $i=0,1$, and 
		a collection of elements $u^{\pm}\colon \intsf(X)\to\bfV_{\pm}^{\wedge, 2}$ satisfying the condition 
		\eqref{eq:u-cond-1}, so that $(\FL{\pm}, [\cdot,\cdot],\delta_{\pm})$ are topological 
		Lie bialgebras with a pairing $\abf{\cdot}{\cdot}\colon \FL{+}\ten\FL{-}\to\bsfld$. Let $\bfJ$
		be a set, and let
		\begin{align*}
			U^{\pm}\colon \bfJ\to\FL{\pm}\quad \text{and} \quad V^{\pm}\colon \bfJ\times\bfJ\to
			\FL{\pm}
		\end{align*}
		be two collections of elements such that 
		$\delta_{\pm}(U^{\pm}_{j})\in\bfV^{\pm}(U_{\cdot},V_{j,\cdot})$, where
		$\bfV^{\pm}(U_{\cdot},V_{j,\cdot})$ denotes the completion in $\FL{\pm}^{\wedge, 2}$
		of the subspace spanned by the elements $U^{\pm}_{\cdot}\wedge V^{\pm}_{j,\cdot}\colon \bfJ\to\FL{\pm}$. Then, if 
		\begin{align}\label{eq:ort-cond}
			\abf{U^{\pm}_{\cdot}}{\xmp{\alpha}}=0=\abf{U^{\pm}_{\cdot}}{\xzmp{\alpha}}\ ,
		\end{align}
		the ideal generated by $U^{\pm}_{\cdot}$ is orthogonal to $\FL{\mp}$ and is a 
		topological coideal in $\FL{\pm}$.
	\end{enumerate}
\end{proposition}

\subsection{Orthogonal coideals}

We shall now fix a Lie bialgebra structure on $\FL{\pm}$ with a pairing,
and show that the defining relations of continuum Kac--Moody algebras
arise from orthogonal coideals. We shall use repeatedly Proposition~\ref{prop:halb}--(3). 

Henceforth, we consider the Lie bialgebra structure on $\FL{\pm}$ given by
\begin{align}\label{eq:cobracket-FL}
	\delta_{\pm}(\xzpm{\ia})\coloneqq 0\quad\mbox{and}\quad
	\delta_{\pm}(\xpm{\ia} )\coloneqq \mp\xzpm{\ia} \wedge\xpm{\ia} \mp\sum_{\ib\sgpp \ic=\ia}
	\cb{\ib}{\ic}\xpm{\ib} \wedge\xpm{\ic} \ ,
\end{align}
where $\cb{\ib}{\ic}=(-1)^{\abf{\ib}{\ib\sgpp\ic}}\rbf{\ib}{\ib\sgpp\ic}
=\ca{\ib}{, \ib\sgpp\ic}$. Then, we define a pairing $\abf{\cdot}{\cdot}\colon \FL{+}\ten\FL{-}
\to\bsfld$ by setting $\abf{\xpm{\ia}}{\xzmp{\ib}}\coloneqq 0$ and
\begin{align*}
	\abf{\xp{\ia}}{\xm{\ib}}\coloneqq \drc{\ia\ib}
	\quad\mbox{and}\quad
	\abf{\xzp{\ia}}{\xzm{\ib}}\coloneqq 2\rbf{\ia}{\ib}\ .
\end{align*}

\begin{proposition}\label{prop:diag-rel}
	Let $\mathfrak{i}_{\pm}$ be the ideal generated in $\FL{\pm}$ by the elements 
	\begin{align*}
		\xzpm{\ia\sgpp\ib}-\drc{\ia\sgpp\ib}\left(\xzpm{\ia}+\xzpm{\ib}\right)\ ,
		\quad
		[\xzpm{\ia},\xzpm{\ib}]\ ,
		\quad
		[\xzpm{\ia},\xpm{\ib}]\mp\rbf{\ia}{\ib}\xpm{\ib}\ .
	\end{align*}
	Then, $\mathfrak{i}_{\pm}$ is a coideal and it is orthogonal to $\FL{\mp}$.
\end{proposition}

\begin{proof}
	We show that the conditions of Proposition~\ref{prop:halb}--(3)
	apply. We first observe that the elements
	$\xzpm{\ia\sgpp\ib}-\drc{\ia\sgpp\ib}\left(\xzpm{\ia}+\xzpm{\ib}\right)$
	are orthogonal to $\FL{\mp}$. 
	Namely, if $\ia\sgpp\ib$ is defined, then $\rbf{\ia\sgpp\ib}{\ic}=\rbf{\ia}{\ic}+\rbf{\ib}{\ic}$,
	and therefore
	\begin{align*}
		\abf{\xzpm{\ia\sgpp\ib}}{\xzmp{\ic}}=\abf{\xzpm{\ia}+\xzpm{\ib}}{\xzmp{\ic}}\ ,
	\end{align*}
	while $\abf{\xzpm{\ia\sgpp\ib}}{\xzmp{\ic}}=0=\abf{\xzpm{\ia}+\xzpm{\ib}}{\xzmp{\ic}}$.
	Moreover, $\delta_{\pm}(\xzpm{\ia})=0$, therefore the condition on the cobracket is trivially 
	satisfied. Similarly, by duality, one has
	\begin{align*}
		\abf{[\xzpm{\ia},\xzpm{\ib}]}{\xzmp{\ic}}=0=\abf{[\xzpm{\ia},\xzpm{\ib}]}{\xmp{\ic}}\ ,
	\end{align*}
	and, by Formula~\eqref{eq:cobracket-FL}, $\delta_{\pm}([\xzpm{\ia},\xzpm{\ib}])=0$.
	Finally, we have
	\begin{align*}
		\abf{[\xzpm{\ia},\xpm{\ib}]}{\xzmp{\ic}}\mp\rbf{\ia}{\ib}\abf{\xpm{\ib}}{\xzmp{\ic}}=0\ ,
	\end{align*}
	and
	\begin{align*}
		\abf{[\xzpm{\ia},\xpm{\ib}]}{\xmp{\ic}}\mp\rbf{\ia}{\ib}\abf{\xpm{\ib}}{\xmp{\ic}}
		&=
		\abf{\xzpm{\ia}\wedge\xpm{\ib}}{\delta_{\mp}(\xmp{\ic})}
		\mp\drc{\ib\ic}\rbf{\ia}{\ib}\\[4pt]
		&=\pm\drc{\ib\ic}\abf{\xzpm{\ia}\wedge\xpm{\ib}}{\xzmp{\ic}\wedge\xmp{\ic}}
		\mp\drc{\ib\ic}\rbf{\ia}{\ib}\\[4pt]
		&=
		\pm\drc{\ib\ic}\rbf{\ia}{\ic}\mp\drc{\ib\ic}\rbf{\ia}{\ib}=0\ .
	\end{align*}
	Moreover, since $\rbf{\ia}{\ib}=\rbf{\ia}{\ic}+\rbf{\ia}{\ic'}$ whenever
	$\ic\sgpp\ic'=\ib$, we get
	\begin{align*}
		\delta_{\pm}&([\xzpm{\ia},\xpm{\ib}]\mp\rbf{\ia}{\ib}\xpm{\ib})\\[4pt]
		&=[\xzpm{\ia}\ten1+1\ten\xzpm{\ia},
		\delta_{\pm}(\xpm{\ib})]\mp\rbf{\ia}{\ib}\delta_{\pm}(\xpm{\ib})\\[4pt]
		&=\mp[\xzpm{\ia},\xzpm{\ib}]\wedge\xpm{\ib}\mp\xzpm{\ib}\wedge[\xzpm{\ia},\xpm{\ib}]\\[3pt]
		&\qquad
		\mp\sum_{\ic\oplus\ic'=\ib}\cb{\ic}{\ic'}
		\left(
		[\xzpm{\ia},\xpm{\ic}]\wedge\xpm{\ic'}
		+\xpm{\ic}\wedge[\xzpm{\ia},\xpm{\ic'}]
		\right)
		\mp\rbf{\ia}{\ib}\delta_{\pm}(\xpm{\ib})\\[4pt]
		&=\mp[\xzpm{\ia},\xzpm{\ib}]\wedge\xpm{\ib}
		\mp\xzpm{\ib}\wedge\left([\xzpm{\ia},\xpm{\ib}]\mp\rbf{\ia}{\ib}\xpm{\ib}\right)\\[3pt]
		&\qquad
		\mp\sum_{\ic\oplus\ic'=\ib}\cb{\ic}{\ic'}
		\left(
		\left([\xzpm{\ia},\xpm{\ic}]\mp\rbf{\ia}{\ic}\xpm{\ic}\right)\wedge\xpm{\ic'}
		+\xpm{\ic}\wedge\left([\xzpm{\ia},\xpm{\ic'}]\mp\rbf{\ia}{\ic'}\xpm{\ic'}\right)
		\right)\ .
	\end{align*}
	The result follows from Proposition~\ref{prop:halb}--(3).
\end{proof}

Thanks to this result, the pairing $\abf{\cdot}{\cdot}\colon \FL{+}\ten\FL{-}\to\bsfld$ 
descends to a pairing between the topological Lie bialgebras
$\wt{\d}_{\pm}\coloneqq\FL{\pm}/\mathfrak{i}_{\pm}$.

\begin{proposition}\label{prop:serre-rel}
	Let $\mathfrak{s}_{\pm}$ be the ideal generated in $\wt{\d}_{\pm}$
	by the elements 
	\begin{align*}
		s_{\ia\ib}^{\pm}\coloneqq [\xpm{\ia},\xpm{\ib}]\mp\cb{\ia}{\ib}\xpm{\ia\sgpp\ib} \ ,
	\end{align*}
	with $(\ia,\ib)\in\serre{X}$. 
	Then, $\mathfrak{s}_{\pm}$ is a coideal and it is orthogonal to $\wt{\d}_{\mp}$.
\end{proposition}

\begin{proof}
	We proceed as before. Clearly, we have $\abf{s_{\ia\ib}^{\pm}}{\xzmp{\ic}}=0$ and
	\begin{align*}
		\abf{s_{\ia\ib}^{\pm}}{\xmp{\ic}}
		&=\abf{[\xpm{\ia},\xpm{\ib}]}{\xmp{\ic}}\mp\cb{\ia}{\ib}\abf{\xpm{\ia\sgpp\ib}}{\xmp{\ic}}\\[4pt]
		&=\abf{\xpm{\ia}\wedge\xpm{\ib}}{\delta_{\mp}(\xmp{\ic})}
		\mp\drc{\ic,\ia\sgpp\ib}\cb{\ia}{\ib}\\[4pt]
		&=\pm\drc{\ic,\ia\sgpp\ib}\cb{\ia}{\ib}\mp\drc{\ic,\ia\sgpp\ib}\cb{\ia}{\ib}=0\ ,
	\end{align*}
	therefore the elements $s_{\ia\ib}^{\pm}$ are orthogonal to $\wt{\d}_{\mp}$.
	Finally, one checks by direct inspection that 
	$\delta_{\pm}(s^{\pm}_{\ia\ib})=\sum_{\ic\ic'} s^{\pm}_{\ic\ic'}\wedge V^{\ia\ib}_{\ic\ic'}$
	for some vectors $V^{\ia\ib}_{\ic\ic'}\in\wt{\d}_{\pm}$.
	The result follows from Proposition~\ref{prop:halb}--(3).
\end{proof}

\subsection{Continuum Kac--Moody algebras by duality}

We now show that the procedure described above realizes the
continuum Kac--Moody algebra $\g_X$ as a topological Lie bialgebras,
endowed with a non--degenerate invariant bilinear form.

Set $\d_{\pm}\coloneqq\wt{\d}_{\pm}/\mathfrak{s}_{\pm}$. Then, $\d_{\pm}$
are topological Lie bialgebras endowed with a Lie bialgebra pairing 
$\abf{\cdot}{\cdot}\colon\d_+\ten\d_-\to\bsfld$. In particular, $(\d_+,\d_-)$ is a matched 
pair of Lie algebras with respect to the coadjoint actions given by 
$\sfad^\ast(d_{\pm})(d'_{\mp})\coloneqq\pm\abf{1\ten d_{\pm}}{\delta_{\mp}(d'_{\mp})}$, $d_{\pm}, 
d'_{\pm}\in\d_{\pm}$.
Let $\dtwo\coloneqq\d_+\dcs\d_-$ be the double cross sum Lie bialgebra. 

\begin{proposition}\label{prop:db-rel}
	The following relations hold in $\dtwo$. For any $\ia,\ib\in\intsf(X)$
	\begin{align}\label{eq:db-rel}
		[\xp{\ia},\xm{\ib}]=\drc{\ia\ib}\frac{\xzp{\ia}+\xzm{\ia}}{2}
		+\ca{\ia}{\ib}\left(\xp{\ia\sgpm\ib}-\xm{\ib\sgpm\ia}\right)\ ,
	\end{align}
	where $\ca{\ia}{\ib}\coloneqq (-1)^{\abf{\ia}{\ib}}\rbf{\ia}{\ib}$\ .
\end{proposition}

\begin{proof}
	It is enough to observe that, by definition, 
	\begin{align*}
		\sfad^\ast(\xp{\ia})(\xm{\ib})&=\drc{\ia\ib}\frac{\xzm{\ia}}{2}
		+\frac{\cb{\ib\sgpm\ia}{,\, \ia}-\cb{\ia}{, \,\ib\sgpm\ia}}{2}\xm{\ib\sgpm\ia}\ ,\\
		\sfad^\ast(\xm{\ib})(\xp{\ia})&=\drc{\ia\ib}\frac{\xzp{\ia}}{2}
		+\frac{\cb{\ia\sgpm\ib}{, \, \ib}-\cb{\ib}{,\, \ia\sgpm\ib}}{2}\xp{\ia\sgpm\ib}\ .
	\end{align*}
	Moreover, since $\cb{a}{b}=\ca{a}{,\, a\sgpp b}$ and $\ca{a}{,\, a\sgpp b}=-\ca{b}{,\, a\sgpp b}$, 
	we get
	\begin{align*}
		\frac{\cb{\ib\sgpm\ia}{,\, \ia}-\cb{\ia}{,\, \ib\sgpm\ia}}{2}=-\ca{\ia}{\ib}\ ,
		\quad
		\frac{\cb{\ia\sgpm\ib}{,\, \ib}-\cb{\ib}{,\, \ia\sgpm\ib}}{2}=\ca{\ia}{\ib}\ .
	\end{align*}
\end{proof}

The combination of Propositions \ref{prop:diag-rel}, \ref{prop:serre-rel}, and \ref{prop:db-rel}
leads to the following.

\begin{theorem}\label{thm:cont-km-lba}
	Let $\calQ_X$ be a continuum quiver and $\g_X$ the corresponding 
	continuum Kac--Moody algebras. 
	\begin{enumerate}\itemsep0.2cm
		\item The Euler form \eqref{eq:Euler-form} on $\fun{X}$ uniquely extends to 
		a non--degenerate invariant symmetric bilinear form $\rbf{\cdot}{\cdot}\colon \g_X \ten\g_X \to\bsfld$
		defined on the generators as follows:
		\begin{align*}
			\rbf{\xz{\ia}}{\xz{\ib}}\coloneqq \rbf{\ia}{\ib}\ ,\quad 
			\rbf{\xpm{\ia}}{\xz{\ib}}\coloneqq 0\ ,\quad 
			\rbf{\xpm{\ia}}{\xpm{\ib}}\coloneqq 0\ ,\quad
			\rbf{\xp{\ia}}{\xm{\ib}}\coloneqq \drc{\ia\ib}\ .
		\end{align*}
		\item There is a unique topological cobracket $\delta\colon \g_X \to\g_X \wh{\ten}\g_X$ 
		defined on the generators by
		\begin{align}\label{eq:cobracket-g(X)}
			\delta(\xz{\ia})\coloneqq 0\quad\mbox{and}\quad
			\delta(\xpm{\ia} )\coloneqq \xzpm{\ia} \wedge\xpm{\ia}+\sum_{\ib\sgpp \ic=\ia}
			\cb{\ib}{\ic}\xpm{\ib} \wedge\xpm{\ic} \ ,
		\end{align}
		and inducing on $\g_X $ a topological Lie bialgebra structure, with respect to which
		the positive and negative Borel subalgebras $\b_X^{\pm}$ are Lie sub-bialgebras.
		\item The Euler form restricts to a non--degenerate pairing of Lie bialgebras 
		$\rbf{\cdot}{\cdot}\colon \b_X^+\ten(\b_X^-)^{\scsop{cop}}\to\bsfld$. Then, the canonical 
		element $r_X\in\b_X^+\wh{\ten}\b_X^-$ corresponding to $\rbf{\cdot}{\cdot}$
		defines a quasi--triangular structure on $\g_X$.  
	\end{enumerate}
\end{theorem}

\begin{proof}
	First, let $\c$ be the ideal generated in $\dtwo$ by the elements 
	$\xzp{\ia}-\xzm{\ia}$, $\ia\in\intsf(X)$. It is clear that $\c$ is central
	in $\dtwo$, is a coideal, and moreover it is contained in the kernel
	of the pairing $\abf{\cdot}{\cdot}$ naturally extended to $\dtwo$. 
	Therefore, $\d\coloneqq\dtwo/\c$ is also Lie bialgebra endowed with
	a pairing, which we denote by $\abf{\cdot}{\cdot}_{\d}$.
	
	Set  $\xz{\ia}\coloneqq\frac{1}{2}(\xzp{\ia}+\xzm{\ia})$. In particular, 
	we have
	\begin{align}\label{eq:euler-d}
		\abf{\xz{\ia}}{\xz{\ib}}_{\d}=\rbf{\ia}{\ib}\ .
	\end{align}
	By Propositions \ref{prop:diag-rel}, \ref{prop:serre-rel}, and \ref{prop:db-rel},
	there is an obvious identification $\g_X=\d$ as Lie algebras (cf.\ Theorem \ref{thm:ass-18}).
	This allows to define a cobracket and possibly degenerate pairing on $\g_X$. 
	However, it follows from \eqref{eq:euler-d} that the kernel of  $\abf{\cdot}{\cdot}_{\d}$
	is a two--sided graded ideal, which trivially intersects $\fun{X}$. Therefore,
	by definition of $\g_X$, it must hold $\ker\abf{\cdot}{\cdot}_{\d}=0$. Therefore, 
	(1), (2), (3) follows directly from the identification $\g_X=\d$.
\end{proof}

From the proof above, we also deduce the following 

\begin{corollary}
	The Euler form \eqref{eq:Euler-form} on $\fun{X}$ uniquely extends to 
	a non--degenerate invariant symmetric bilinear form $\rbf{\cdot}{\cdot}\colon \wt{\g}_X \ten\wt{\g}_X \to\bsfld$
	defined on the generators as follows:
	\begin{align*}
		\rbf{\xz{\ia}}{\xz{\ib}}\coloneqq \rbf{\ia}{\ib}\ ,\quad 
		\rbf{\xpm{\ia}}{\xz{\ib}}\coloneqq 0\ ,\quad 
		\rbf{\xpm{\ia}}{\xpm{\ib}}\coloneqq 0\ ,\quad
		\rbf{\xp{\ia}}{\xm{\ib}}\coloneqq \drc{\ia\ib}\ .
	\end{align*}
	Moreover, $\r_X=\ker\rbf{\cdot}{\cdot}$, \ie $\ker\rbf{\cdot}{\cdot}$ is the maximal two--sided ideal trivially
	intersecting $\fun{X}$ and it is generated by the Serre relations from Theorem \ref{thm:ass-18}. 
\end{corollary}

\bigskip\section{Continuum quantum groups}\label{s:cont-qg}

In this section we shall introduce the \emph{continuum quantum groups}, which provide a quantization of the continuum Kac--Moody algebras. We will see that they can be similarly realized as uncountable colimits of Drinfeld--Jimbo quantum groups. Finally, when the underlying vertex space is the line or the circle, they coincide with the line quantum group and the circle quantum group of \cite{sala-schiffmann-17}.

\subsection{Definition of continuum quantum groups}

Let $\calQ_X\coloneqq(\intsf(X), \sgpp,\sgpm,\abf{\cdot}{\cdot}, \rbf{\cdot}{\cdot})$ 
be a \emph{continuum quiver} with underlying vertex space $X$. In order to define the \emph{continuum quantum group}, we need to introduce some new operations on intervals. 

\begin{definition}
	We define the following partial operations on $\intsf(X)$:
	\begin{enumerate}\itemsep0.2cm
		\item the \emph{strict union of two non--orthogonal intervals $\ia$ and $\ib$}, whenever defined, 
		is the smallest interval $\iM{\ia}{\ib}\in\intsf(X)$ for which $(\iM{\ia}{\ib})\sgpm\ia$ 
		and $(\iM{\ia}{\ib})\sgpm\ib$ are both defined;
		\item the \emph{strict intersection of two non--orthogonal intervals $\ia$ and $\ib$}, whenever defined, 
		is the biggest interval $\im{\ia}{\ib}\in\intsf(X)$ for which $\ia\sgpm(\im{\ia}{\ib})$ and 
		$\ib\sgpm(\im{\ia}{\ib})$ are both defined.
	\end{enumerate}
\end{definition}

\begin{remark}
	Note that $\iM{\ia}{\ib}$ (resp. $\im{\ia}{\ib}$) is defined and coincides with $\ia\cup\ib$ 
	(resp. $\ia\cap\ib$) whenever it contains \emph{strictly} $\ia$ and $\ib$ (resp.\ it is contained 
	\emph{strictly} in $\ia$ and $\ib$). In particular, $\iM{}{}$ and $\im{}{}$ are clearly symmetric.
\end{remark}

\begin{remark}
	Let $X=\R$ and $\ia,\ib\in\intsf(\R)$.
	\begin{itemize}\itemsep0.2cm
		\item If $\ia\intnext\ib$, then $\iM{\ia}{\ib}=\ia\sgpp\ib$ and $\im{\ia}{\ib}$ is not defined.
		\item If $\ia\intcap\ib$, then  $\iM{\ia}{\ib}=\ia\cup\ib$ and $\im{\ia}{\ib}=\ia\cap\ib$. Moreover,
		\begin{align*}
			(\negthinspace(\iM{\ia}{\ib})\sgpm\ib)\sgpp(\im{\ia}{\ib})=\ia=(\iM{\ia}{\ib})\sgpm(\ib\sgpm(\im{\ia}{\ib}))
		\end{align*}
		\item If $\ia$ and $\ib$ are nested, then $\iM{\ia}{\ib}$ and $\im{\ia}{\ib}$ are not defined.
		\footnote{Recall that $\ia$ and $\ib$ are nested if they are orthogonal or one contained in the other.}
	\end{itemize}
\end{remark}

\begin{definition}\label{def:new-coeff}
	We shall use the following functions on $\intsf(\K)\times\intsf(\K)$:
	\begin{itemize}\itemsep0.3cm
		\item $\displaystyle \ca{\ia}{\ib}\coloneqq(-1)^{\abf{\ia}{\ib}}\rbf{\ia}{\ib}$;
		\item $\displaystyle \qcb{\ia}{\ib}{}\coloneqq\ca{\ia}{,\, \iM{\ia}{\ib}}$, which generalizes the function 
		$\qcb{\ia}{\ib}{}$ defined in \eqref{eq:ca-cb};
		\item $\displaystyle \qcc{\ia}{\ib}{+}\coloneqq\half\left(\ca{\ib}{,\,\ia\sgpm\ib}-1\right)$,
		and $\displaystyle \qcc{\ia}{\ib}{-}\coloneqq\half\left(\ca{\ib\sgpm\ia}{,\,\ia}+1\right)$;
		\item $\displaystyle \qcr{\ia}{\ib}{}\coloneqq(1-\drc{\ia\ib})(-1)^{\abf{\ia}{\ib}}\rbf{\ia}{\ib}^2$ ;
		\item $\displaystyle \qcs{\ia}{\ib}{\pm}\coloneqq\half\left(\ca{\ib}{,\, \ia\sgpp\ib}\pm1\right)$.
	\end{itemize}
\end{definition}

\begin{remark}\label{rmk:coefficient-table}
	Let $X=\K$, with $\K=\Q, \R$. We summarize below all possible values of the functions above.
	\begin{align*}
		\begin{array}{|c|c|c|c|c|c|c|c|c|c|c|}
			\hline
			& \ia\star\ib & \abf{\ia}{\ib} & \abf{\ib}{\ia} & \ca{\ia}{\ib} & \qcb{\ia}{\ib}{} & \qcc{\ia}{\ib}{+} & \qcc{\ia}{\ib}{-}  &\qcr{\ia}{\ib}{} & \qcs{\ia}{\ib}{+} & \qcs{\ia}{\ib}{-} \\
			\hline 
			(a)&\ia\intnext\ib &-1&\phantom{+}0&\phantom{+}1&\phantom{+}1&n.d.&n.d.&-1&\phantom{+}0&-1\\
			\hline 
			(b)&\ib\intnext\ia &\phantom{+}0&-1&-1&-1&n.d.&n.d.&\phantom{+}1&\phantom{+}1&\phantom{+}0\\
			\hline 
			(c)&\ia\intcap\ib &-1&\phantom{+}1&\phantom{+}0&\phantom{+}1&n.d.&n.d.&\phantom{+}0&n.d.&n.d.\\
			\hline 
			(d)&\ib\intcap\ia &\phantom{+}1&-1&\phantom{+}0&-1&n.d.&n.d.&\phantom{+}0&n.d.&n.d.\\
			\hline 
			(e)&\ia\perp\ib &\phantom{+}0&\phantom{+}0&\phantom{+}0&n.d.&n.d.&n.d.&\phantom{+}0&n.d.&n.d.\\
			\hline 
			(f)&\ia<\ib &\phantom{+}0&\phantom{+}0&\phantom{+}0&n.d.&n.d.&n.d.&\phantom{+}0&n.d.&n.d.\\
			\hline 
			(g)&\ib<\ia &\phantom{+}0&\phantom{+}0&\phantom{+}0&n.d.&n.d.&n.d.&\phantom{+}0&n.d.&n.d.\\
			\hline 
			(h)&\ia\rsub\ib &\phantom{+}0&\phantom{+}1&\phantom{+}1&n.d.&n.d.&\phantom{+}0&\phantom{+}1&n.d.&n.d.\\
			\hline 
			(i)&\ia\lsub\ib &\phantom{+}1&\phantom{+}0&-1&n.d.&n.d.&\phantom{+}1&-1&n.d.&n.d.\\
			\hline 
			(j)&\ib\rsub\ia &\phantom{+}1&\phantom{+}0&-1&n.d.&\phantom{+}0&n.d.&-1&n.d.&n.d.\\
			\hline 
			(k)&\ib\lsub\ia &\phantom{+}0&\phantom{+}1&\phantom{+}1&n.d.&-1&n.d.&\phantom{+}1&n.d.&n.d.\\
			\hline
		\end{array}
	\end{align*}
\end{remark}

\begin{definition}\label{def:cont-qg}
	Let $\calQ_X$ be a continuum quiver. The \emph{continuum quantum group of $X$} 
	is the associative algebra $\bfU_q\g_X $ generated by $\fun{X}$ 
	and the elements $\qxpm{\ia}$, $\ia\in\intsf(\K)$, satisfying the following defining relations:
	\begin{enumerate}\itemsep0.2cm
		\item {\bf Diagonal action:} for any $\ia,\ib\in\intsf(X)$,
		\begin{align*}
			[\xz{\ia},\xz{\ib}]=0\qquad\mbox{and}\qquad [\xz{\ia},\qxpm{\ib}]=\pm\rbf{\ia}{\ib}\qxpm{\ib}\ .
		\end{align*}
		In particular, for $\qxz{\ia}{}\coloneqq\exp(\hbar/2\cdot\xz{\ia})$,
		it holds $\qxz{\ia}{}\qxpm{\ib}=q^{\pm\rbf{\ia}{\ib}}\cdot\qxpm{\ib}\qxz{\ia}{}$.
		
		\item {\bf Quantum double relations:} for any $\ia,\ib\in\intsf(X)$,
		\begin{align*} 
			[\qxp{\ia} ,\qxm{\ib} ] = &\drc{\ia, \ib}\frac{\qxz{\ia}{}-\qxz{\ia}{-1}}{q-q^{-1}}\\ &\qquad+\ca{\ia}{\ib}\cdot\left(q^{\qcc{\ia}{\ib}{+}}\, \qxp{\ia\sgpm\ib}
			\, \qxz{\ib}{\ca{\ia}{\ib}}-
			q^{\qcc{\ia}{\ib}{-}}\, \qxz{\ia}{\ca{\ia}{\ib}}\, \qxm{\ib\sgpm\ia}\right)\\
			&\qquad\qquad+\qcb{\ib}{\ia}{}\, q^{\qcb{\ib}{\ia}{}}\,(q-q^{-1})\,\qxp{(\iM{\ia}{\ib})\sgpm\ib}\,
			\qxz{\im{\ia\, }{\, \ib}}{\qcb{\ia}{\ib}{}}\,\qxm{(\iM{\ia}{\ib})\sgpm{\ia}}\ .
		\end{align*}
		\item {\bf Quantum Serre relations:} for any $(\ia,\ib)\in\serre{X}$,
		\begin{align*}
			\qxpm{\ia}\qxpm{\ib}-q^{\qcr{\ia}{\ib}{}}\cdot&\qxpm{\ib}\qxpm{\ia}=
			\pm\qcb{\ia}{\ib}{}\cdot q^{\qcs{\ia}{\ib}{\pm}}\cdot\qxpm{\ia\sgpp\ib}
			+\qcb{\ia}{\ib}{}\cdot(q-q^{-1})\cdot\qxpm{\iM{\ia}{\ib}}\qxpm{\im{\ia\, }{\,\ib}}\ .
		\end{align*}
	\end{enumerate}
	We assume that $\qxpm{\ia \odot \ib }=0$ whenever $\ia \odot \ib $ 
	is not defined, for $\odot=\sgpp,\sgpm, \iM{}{}, \im{}{}$, and the functions
	$\ca{}{}, \qcb{}{}{}, \qcc{}{}{}, \qcr{}{}{}, \qcs{}{}{}$ are those introduced of Definition~\ref{def:new-coeff}.
\end{definition}

\subsection{Colimit structure}\label{ss:q-colimit}
In analogy with Section \ref{ss:bkm-and-sgp}, we prove that the continuum quantum group 
$\bfU_q\g_X $ is \emph{covered} by an uncountable family of Drinfeld--Jimbo quantum groups.
Let $\calJ$ be an irreducible family of intervals in $\intsf(X)$ 
(cf.\ Section~\ref{ss:bkm-and-sgp}). We then consider two quantum algebras associated to $\calJ$:
\begin{itemize}\itemsep0.2cm
	\item the Drinfeld--Jimbo quantum group $\bfU_q\g_{\calJ}^{\scsop{BKM}}$ with Cartan matrix 
	$\sfA_{\calJ}=[\rbf{\ia}{\ib}]_{\ia,\ib\in\calJ}$\ ;
	\item the subalgebra $\bfU_q\g_{\calJ}$ generated in $\bfU_q\g_X $ by the elements
	$\{\xz{\ia}, \qxpm{\ia}\;|\; \ia\in\calJ\}$.
\end{itemize}

\begin{proposition}\label{prop:q-bkm-sgp}
	The assignment 
	\begin{align*}
		E_{\ia}\mapsto \qxp{\ia} \ , \quad F_{\ia}\mapsto \qxm{\ia}  \quad \text{and} \quad H_{\ia} \mapsto \xz{\ia}
	\end{align*}
	for any $\ia\in \calJ$, defines a surjective morphism of algebras $\Phi_\calJ\colon \bfU_q\g_{\calJ}^{\scsop{BKM}}\to \bfU_q\g_{\calJ}$.
\end{proposition}

\begin{proof}
	First, note that Proposition~\ref{prop:bkm-sgp} follows from the result above by setting $\hbar=0$.
	It is easy to check that, applying the quantum Serre relations (3) of Definition~\ref{def:cont-qg} corresponding to the elements $\qxpm{\ia}$, with $\ia\in \calJ$, 
	one recovers the \emph{standard} quantum Serre relations of the Drinfeld--Jimbo quantum group $\DJ{\g_{\calJ}^{\scsop{BKM}}}$ (cf.\ Section~\ref{ss:DJ}). 
	Thus, by mimicking the arguments of the proof of Proposition~\ref{prop:bkm-sgp}, the result follows.
\end{proof}

The following is straightforward.

\begin{corollary}\label{cor:q-bkm-sgp}
	Let $\calJ,\calJ'$ be two irreducible (finite) sets of intervals in $X$.
	\begin{enumerate}\itemsep0.2cm
		\item If $\calJ'\subseteq\calJ$, there is a canonical embedding
		$\phi'_{\calJ,\calJ'}\colon \bfU_q\g_{\calJ'}\to\bfU_q\g_{\calJ}$ sending $\qxpm{\ia} \mapsto \qxpm{\ia}$, 
		$\xz{\ia}\mapsto \xz{\ia}$, $\ia\in\calJ'$.
		\item If $\calJ$ is obtained from $\calJ'$ by replacing an element $\ic\in\calJ'$
		with two intervals $\ia ,\ib $ such that $\ic=\ia \sgpp \ib $,
		there is a canonical embedding $\phi''_{\calJ,\calJ'}\colon \bfU_q\g_{\calJ'}\to\bfU_q\g_{\calJ}$, 
		which is the identity on $\bfU_q\g_{\calJ'\setminus\{\ic\}}=\bfU_q\g_{\calJ\setminus\{\ia,\ib\}}$
		and sends
		\begin{align*}
			\xz{\ic} \mapsto \xz{\ia } +\xz{\ib } \ , \quad 
			\qxpm{\ic}  \mapsto 
			\mp\qcb{\ia}{\ib}{-1}\cdot q^{-\qcs{\ia}{\ib}{\pm}}\cdot
			\left(\qxpm{\ia}\qxpm{\ib}-q^{\qcr{\ia}{\ib}{}}\cdot\qxpm{\ib}\qxpm{\ia}\right)\ .
		\end{align*}
		\item The collection of embeddings $\phi'_{\calJ,\calJ'}, \phi''_{\calJ,\calJ'}$, 
		indexed by all possible irreducible sets of intervals in $X$, form a direct system. 
		Moreover, there is a canonical surjective morphism  of algebras
		\begin{align*}
			\colim_{\calJ}\, \bfU_q\g_{\calJ}^{\scsop{BKM}}\to \bfU_q\g(X)\ .
		\end{align*}
	\end{enumerate}
\end{corollary}

\subsection{Comparison with the quantum group of the line}
We will now show that the continuum quantum groups of $\DJ{\g_X}$, $X=\R,S^1$, 
coincide with the quantum groups of the line and the circle introduced in \cite{sala-schiffmann-17}.
Let us first recall the definition of the line quantum group $\bfU_q\sl(\R)$.

\begin{definition}\label{def:q-sl(K)}
	Let $\K=\Q,\R$. The \emph{quantum group of the line} is the associative algebra  $\bfU_q\sl(\K)$
	generated over $\hext{\C}$ by elements $E_{\ia}, F_{\ia}, H_{\ia}$, with ${\ia}\in\intsf(\K)$, 
	with the following defining relations. Set $q\coloneqq\exp(\hbar/2)$ and $K_{\ia}\coloneqq\exp(\hbar/2\cdot H_{\ia})$.
	\begin{itemize}\itemsep0.3cm
		\item {\bf Kac--Moody type relations:} for any two intervals $\ia, \ib$,
		\begin{align}\label{eq:q-kac-rel-1}
			[H_{\ia} ,H_{\ib} ]=0, \quad [H_{\ia} , E_{\ib} ]=\rbf{{\ia}}{{\ib}}\, E_{\ib}, \quad
			[H_{\ia} , F_{\ib} ] &=-\rbf{{\ia}}{{\ib}}\, F_{\ib}, 
		\end{align}
		\begin{align}\label{eq:q-kac-rel-2}
			[E_{\ia} ,F_{\ib} ]=
			\begin{cases}
				\displaystyle \frac{K_{\ia}-K_{\ia}^{-1}}{q-q^{-1}}  & \text{if }  \ia=\ib\ ,\\[5pt]
				\phantom{K_{\ia}}0 & \text{if }  \ia\perp \ib, \ia\intnext \ib, \text{ or } \ib\intnext \ia\ ,
			\end{cases}
		\end{align}
		
		\item {\bf join relations:} for any two intervals $\ia , \ib $ with $\ia\intnext\ib$,
		\begin{align}
			H_{\ia \sgpp \ib }&=H_{\ia} + H_{\ib} \ ,\label{eq:q-Joining-h-finite}\\[3pt]
			E_{\ia \sgpp \ib }&=\phantom{+}q^{1/2}E_{\ia}E_{\ib}-q^{-1/2}E_{\ib}E_{\ia}\ , \label{eq:q-Joining-e-finite}\\[3pt]
			F_{\ia \sgpp \ib }&=-q^{1/2}F_{\ia}F_{\ib}+q^{-1/2}F_{\ib}F_{\ia}\ ; \label{eq:q-Joining-f-finite}
		\end{align}
		
		\item {\bf nest relations:} for any nested $\ia ,\ib \in\intsf(\K)$ (that is, such that	$\ia =\ib $, 
		$\ia \perp \ib $, $\ia <\ib $, $\ib <\ia $, $\ia \rsub \ib $, $\ia \lsub \ib $, $\ib \rsub \ia $, or $\ib \lsub \ia $),
		\begin{align}\label{eq:q-nest-finite}
			q^{\abf{\ia}{\ib}}E_{\ia}E_{\ib}=q^{\abf{\ib}{\ia}}E_{\ib}E_{\ia}\quad\text{and}\quad q^{\abf{\ia}{\ib}}F_{\ia}F_{\ib}=q^{\abf{\ib}{\ia}}F_{\ib}F_{\ia}\ .
		\end{align}
	\end{itemize}
\end{definition}

It follows, in particular, that 
\begin{align*}\label{eq:q-kac-rel-K}
	K_{\ia}K_{\ib}=K_{\ib}K_{\ia}, \quad K_{\ia}E_{\ib}=q^{\rbfcf{{\ia}}{{\ib}}}\, E_{\ib}K_{\ia}, \quad K_{\ia}F_{\ib}=q^{-\rbfcf{{\ia}}{{\ib}}}\, F_{\ib}K_{\ia}\ . 
\end{align*}

As in the case of $\sl(\K)$, the Cartan subalgebra of $\bfU_q\sl(\K)$, namely $\bfU_q\h\coloneqq\langle H_{\ia}\;\vert\;\ia\in\intsf(\K)\rangle$, 
is canonically isomorphic to the symmetric algebra $\hext{S\fun{\K}}$ generated by the characteristic functions 
$\{\xz{\ia} \coloneqq\cf{\ia}\;|\;\ia\in\intsf(\K)\}$.  

We have the following:
\begin{proposition}\label{prop:q-slR-presentation}
	There is an isomorphism of algebras $\bfU_q\g_\K\to\bfU_q\sl(\K)$ given by 
	\begin{align*}
		\qxp{\ia} \mapsto q^{\half}E_{\ia}\ ,\quad
		\qxm{\ia} \mapsto q^{-\half}F_{\ia}\ ,\quad
		\xz{\ia}\mapsto H_{\ia}\ ,
	\end{align*}
	with $\ia\in\intsf(\K)$.	
	%
\end{proposition}

\begin{proof}
	First, we show that the relations (1)--(3) from Definition~\ref{def:cont-qg} imply those from Definition~\ref{def:q-sl(K)}.
	
	\subsubsection{}
	The Kac--Moody relations \eqref{eq:q-kac-rel-1} and \eqref{eq:q-kac-rel-2} follow 
	immediately from (1) and (2), respectively.
	The join relation \eqref{eq:q-Joining-h-finite} is automatic, while \eqref{eq:q-Joining-e-finite} 
	and \eqref{eq:q-Joining-f-finite} follow from (3). Namely, if $\ia\intnext\ib$, then $\iM{\ia}{\ib}=
	\ia\sgpp\ib$, and $\im{\ia}{\ib}$ is not defined (therefore the last summand on the RHS of (3) does 
	not appear) and
	\begin{align*}
		\qcr{\ia}{\ib}{}=- 1\ ,\quad 
		\qcb{\ia}{\ib}{}= 1\ ,\quad 
		\qcs{\ia}{\ib}{+}=0\ , \quad
		\qcs{\ia}{\ib}{-}=-1\ .
	\end{align*}
	So that (3) reads $\qxp{\ia}\qxp{\ib}-q^{-1}\qxp{\ib}\qxp{\ia}= \qxp{\ia\sgpp\ib}$ (resp.
	$\qxm{\ia}\qxm{\ib}-q^{-1}\qxm{\ib}\qxm{\ia}=-q^{-1}\qxm{\ia\sgpp\ib}$).
	Then, since $\qxp{\ia} = q^{\half}E_{\ia}$ and $\qxm{\ia} = q^{-\half}F_{\ia}$, one has
	\begin{align*}
		qE_{\ia}E_{\ib}-E_{\ib}E_{\ia}= q^{\half}E_{\ia\sgpp\ib}  \quad\mbox{and}\quad
		q^{-1}F_{\ia}F_{\ib}-q^{-2}F_{\ib}F_{\ia}=-q^{-\frac{3}{2}}F_{\ia\sgpp\ib}\ ,
	\end{align*}
	which corresponds to \eqref{eq:q-Joining-e-finite} and \eqref{eq:q-Joining-f-finite}, respectively.
	Assume now that $\ia$ and $\ib$ are nested and $\ia\neq\ib$, so that $\ia\sgpp\ib$, $\iM{\ia}{\ib}$ and $\im{\ia}{\ib}$
	are not defined, and (3) reduces to $\qxpm{\ia}\qxpm{\ib}=
	q^{\qcr{\ia}{\ib}{}}\qxpm{\ib}\qxpm{\ia}$. Then, \eqref{eq:q-nest-finite} 
	follows by observing that, in case of nested intervals, $\qcr{\ia}{\ib}{}
	\coloneqq(-1)^{\abf{\ia}{\ib}}\rbf{\ia}{\ib}^2=\abf{\ib}{\ia}-\abf{\ia}{\ib}$, 
	as one checks easily from the last seven rows (e--k) of the table 
	\ref{rmk:coefficient-table} above.
	
	\subsubsection{}
	Conversely, we shall show that the relations (1)--(3) holds in the algebra $\bfU_q\sl(\K)$.
	(1) follows from  \eqref{eq:q-kac-rel-1}. 
	By the previous discussion, (3) holds for the cases (a) and (e--k) listed in the table \ref{rmk:coefficient-table}. 
	It remains to prove it holds in the cases (b--d).
	
	\begin{itemize}[leftmargin=1em]\itemsep0.5cm
		\item {\it Case (b): $\ib\intnext\ia$.} 
		From \eqref{eq:q-Joining-e-finite} and \eqref{eq:q-Joining-f-finite}, we get
		\begin{align*}
			q^{\half}E_{\ib}E_{\ia}-q^{-\half}E_{\ia}E_{\ib}=E_{\ia\sgpp\ib}\ , \qquad
			q^{\half}F_{\ib}F_{\ia}-q^{-\half}F_{\ia}F_{\ib}=-F_{\ia\sgpp\ib}\ .
		\end{align*}
		Then, by $\qxp{\ia} = q^{\half}E_{\ia}$ and $\qxm{\ia} = q^{-\half}F_{\ia}$, we get
		\begin{align*}
			q^{-\frac{3}{2}}\qxp{\ia}\qxp{\ib}-q^{-\half}\qxp{\ib}\qxp{\ia}&=-q^{-\half}\qxp{\ia\sgpp\ib}\ ,\\
			q^{\half}\qxm{\ia}\qxm{\ib}-q^{\frac{3}{2}}\qxm{\ib}\qxm{\ia}&=q^{\half}\qxm{\ia\sgpp\ib}\ ,
		\end{align*}
		Finally, we get
		\begin{align*}
			\qxp{\ia}\qxp{\ib}-q\qxp{\ib}\qxp{\ia}=-q\qxp{\ia\sgpp\ib}\ ,
			\qquad
			\qxm{\ia}\qxm{\ib}-q\qxm{\ib}\qxm{\ia}=\qxm{\ia\sgpp\ib}\ ,
		\end{align*}
		which agrees with (3), since for $\ib\intnext\ia$ we have $\qcr{\ia}{\ib}{}=1$,
		$\qcb{\ia}{\ib}{}=-1$, $\qcs{\ia}{\ib}{+}=1$, and $\qcs{\ia}{\ib}{-}=0$.\\
		
		\item {\it Case (c): $\ia\intcap\ib$.} 
		Note that, in this case, $\iM{\ia}{\ib}$ and $\im{\ia}{\ib}$ are
		both defined, $\qcr{\ia}{\ib}{}=0$, $\qcb{\ia}{\ib}{}=1$, 
		and (3) reads
		\begin{align*}
			\qxpm{\ia}\qxpm{\ib}-\qxpm{\ib}\qxpm{\ia}=(q-q^{-1})\qxpm{\iM{\ia}{\ib}}\qxpm{\im{\ia\,}{\,\ib}}\ .
		\end{align*}
		
		Set $a=\ia\sgpm(\im{\ia}{\ib})$, $b=\ia\sgpm(\im{\ia}{\ib})$,
		and $c=\im{\ia}{\ib}$. Thus, $\ia=a\sgpp c$ with $a\intnext c$, 
		$\ib=c\sgpp b$ with $c\intnext b$, and $\iM{\ia}{\ib}=\ia\sgpp b$ with
		$\ia\intnext b$. Since $c\lsub\ia$ and $\ia\intnext b$, we have
		\begin{align*}
			E_{\ib}=q^{\half}E_{c}E_{b}-q^{-\half}E_{b}E_{c}\ ,
			\qquad
			E_{\ia}E_{c}=qE_{c}E_{\ia}\ ,\\
			E_{\ia}E_{b}=q^{-\half}E_{\iM{\ia}{\ib}}+q^{-1}E_{b}E_{\ia}\ .
		\end{align*}
		Therefore,
		\begin{align*}
			E_{\ia}E_{\ib}&=E_{\ia}\left(q^{\half}E_{c}E_{b}-q^{-\half}E_{b}E_{c}\right)\\
			&=q^{\frac{3}{2}}E_c\left(q^{-\half}E_{\iM{\ia}{\ib}}+q^{-1}E_{b}E_{\ia}\right)-
			q^{-\half}\left(q^{-\half}E_{\iM{\ia}{\ib}}+q^{-1}E_{b}E_{\ia}\right)E_c\\
			&=qE_{c}E_{\iM{\ia}{\ib}}+q^{\half}E_cE_bE_{\ia}
			-q^{-1}E_{\iM{\ia}{\ib}}E_c-q^{-\frac{1}{2}}E_bE_cE_{\ia} \ .
		\end{align*}
		Since $c=\im{\ia}{\ib}<\iM{\ia}{\ib}$, we get
		\begin{align*}
			E_{\ia}E_{\ib}&=E_{\ib}E_{\ia}+(q-q^{-1})E_{\iM{\ia}{\ib}}E_{\im{\ia\,}{\,\ib}}\ ,
		\end{align*}
		which agrees with (3) under the identification $\qxp{\ia}=q^{\half}E_{\ia}$.
		Similarly,
		\begin{align*}
			F_{\ib}=-q^{\half}F_{c}F_{b}+q^{-\half}F_{b}F_{c}\ ,
			\qquad
			F_{\ia}F_{c}=qF_{c}F_{\ia}\ ,\\
			F_{\ia}F_{b}=-q^{-\half}F_{\iM{\ia}{\ib}}+q^{-1}F_{b}F_{\ia}\ .
		\end{align*}
		Therefore,
		\begin{align*}
			F_{\ia}&F_{\ib}=F_{\ia}\left(-q^{\half}F_{c}F_{b}+q^{-\half}F_{b}F_{c}\right)\\
			&=-q^{\frac{3}{2}}F_c\left(-q^{-\half}F_{\iM{\ia}{\ib}}+q^{-1}F_{b}F_{\ia}\right)
			+q^{-\half}\left(-q^{-\half}F_{\iM{\ia}{\ib}}+q^{-1}F_{b}F_{\ia}\right)F_c\\
			&=qF_cF_{\iM{\ia}{\ib}}-q^{\half}F_cF_bF_{\ia}
			-q^{-1}F_{\iM{\ia}{\ib}}F_c+q^{-\frac{1}{2}}F_bF_cF_{\ia}\\
			&=F_{\ib}F_{\ia}+(q-q^{-1})F_{\iM{\ia}{\ib}}F_{\im{\ia\, }{\, \ib}} \ ,
		\end{align*}
		which agrees with (3) under the identification $\qxm{\ia}=q^{-\half}F_{\ia}$.
		
		\item {\it Case (d): $\ib\intcap\ia$.}
		In this case, $\qcr{\ia}{\ib}{}=0$, $\qcb{\ia}{\ib}{}=-1$, 
		and (3) reads
		\begin{align*}
			\qxpm{\ia}\qxpm{\ib}-\qxpm{\ib}\qxpm{\ia}=-(q-q^{-1})\qxpm{\iM{\ia}{\ib}}\qxpm{\im{\ia\,}{\,\ib}}\ .
		\end{align*}
		Thus, $\ia=c\sgpp a$ with $c\intnext a$, 
		$\ib=b\sgpp c$ with $b\intnext c$, and $\iM{\ia}{\ib}=b\sgpp \ia$ with
		$b\intnext \ia$. Since $c\rsub\ia$ and $b\intnext\ia$, we have
		\begin{align*}
			E_{\ib}=q^{\half}E_{b}E_{c}-q^{-\half}E_{c}E_{b}\ ,
			\qquad
			E_{\ia}E_{c}=q^{-1}E_{c}E_{\ia}\ ,\\
			E_{\ia}E_{b}=-q^{\half}E_{\iM{\ia}{\ib}}+qE_{b}E_{\ia}\ .
		\end{align*}
		Therefore,
		\begin{align*}
			E_{\ia}E_{\ib}&=E_{\ia}\left(q^{\half}E_{b}E_{c}-q^{-\half}E_{c}E_{b}\right)\\
			&=q^{\half}\left(-q^{\half}E_{\iM{\ia}{\ib}}+qE_{b}E_{\ia}\right)E_c
			-q^{-\frac{3}{2}}E_c\left(-q^{\half}E_{\iM{\ia}{\ib}}+qE_{b}E_{\ia}\right)\\
			&=-qE_{\iM{\ia}{\ib}}E_c+q^{\half}E_bE_cE_{\ia}+q^{-1}E_{c}E_{\iM{\ia}{\ib}}-q^{-\half}E_cE_bE_{\ia}\\
			&=E_{\ib}E_{\ia}-(q-q^{-1})E_{\iM{\ia}{\ib}}E_{\im{\ia\, }{\, \ib}}\ ,
		\end{align*}
		which agrees with (3). Similarly,
		\begin{align*}
			F_{\ib}=-q^{\half}F_{b}F_{c}+q^{-\half}F_{c}F_{b}\ ,
			\qquad
			F_{\ia}F_{c}=q^{-1}F_{c}F_{\ia}\ ,\\
			F_{\ia}F_{b}=q^{\half}F_{\iM{\ia}{\ib}}+qF_{b}F_{\ia}\ .
		\end{align*}
		Therefore,
		\begin{align*}
			F_{\ia}F_{\ib}&=F_{\ia}\left(-q^{\half}F_{b}F_{c}+q^{-\half}F_{c}F_{b}\right)\\
			&=-q^{\half}\left(q^{\half}F_{\iM{\ia}{\ib}}+qF_{b}F_{\ia}\right)F_c
			+q^{-\frac{3}{2}}F_c\left(q^{\half}F_{\iM{\ia}{\ib}}+qF_{b}F_{\ia}\right)\\
			&=-qF_{\iM{\ia}{\ib}}F_c-q^{\half}F_bF_cF_{\ia}
			+q^{-1}F_cF_{\iM{\ia}{\ib}}+q^{-\half}F_cF_bF_{\ia}\\
			&=F_{\ib}F_{\ia}-(q-q^{-1})F_{\iM{\ia}{\ib}}F_{\im{\ia\,}{\,\ib}}\ ,
		\end{align*}
		which agrees with (3).
	\end{itemize}
	
	\subsubsection{} We now show that relations (2) hold in $\bfU_q\sl(\K)$.
	This is clear for the cases (a), (b), (e) in the table \ref{rmk:coefficient-table}. We should prove 
	it for all remaining cases. We start with the cases of a \emph{boundary}
	subinterval (rows h--k).
	
	\begin{itemize}[leftmargin=1em]\itemsep0.5cm
		
		\item {\it Case (h): $\ia\rsub\ib$.}
		In this case, we have $\ca{\ia}{\ib}=1$ and $\qcc{\ia}{\ib}{-}=0$, so that (2) reads
		\begin{align*}[\qxp{\ia},\qxm{\ib}]=-\qxz{\ia}{}\qxm{\ib\sgpm\ia}\ .\end{align*}
		Set $\ic=\ib\sgpm\ia$. Thus, $\ib=\ia\sgpp\ic$ with $\ia\intnext\ic$.
		We have 
		\begin{align*}
			F_{\ib}=-q^{\half}F_{\ia}F_{\ic}+q^{-\half}F_{\ic}F_{\ia}\ ,
			\quad
			[E_{\ia}, F_{\ic}]=0\ ,\\
			F_{\ic}\qxz{\ia}{}=q^{-1}\qxz{\ia}{}F_{\ic}\ ,
			\quad
			F_{\ic}\qxz{\ia}{-1}=q\qxz{\ia}{-1}F_{\ic}\ .
		\end{align*} 
		Therefore,
		\begin{align*}
			[E_{\ia}, F_{\ib}]&=-q^{\half}[E_{\ia},F_{\ia}]F_{\ic}+q^{-\half}F_{\ic}[E_{\ia},F_{\ia}]\\
			&=-q^{\half}\frac{\qxz{\ia}{}-\qxz{\ia}{-1}}{q-q^{-1}}F_{\ic}
			+q^{-\half}F_{\ic}\frac{\qxz{\ia}{}-\qxz{\ia}{-1}}{q-q^{-1}}\\
			&=\left(-q^{\half}\frac{\qxz{\ia}{}-\qxz{\ia}{-1}}{q-q^{-1}}
			+q^{-\half}\frac{q^{-1}\qxz{\ia}{}-q\qxz{\ia}{-1}}{q-q^{-1}}\right)F_{\ic}\\
			&=\frac{-q^{\half}+q^{-\frac{3}{2}}}{q-q^{-1}}\qxz{\ia}{}F_{\ic}
			=-q^{-\half}\frac{q-q^{-1}}{q-q^{-1}}\qxz{\ia}{}F_{\ic}=-q^{-\half}\qxz{\ia}{}F_{\ib\sgpm\ia}\ ,
		\end{align*}
		and we get $[\qxp{\ia}, \qxm{\ib}]=-\qxz{\ia}{}\qxm{\ib\sgpm\ia}$.
		
		\item {\it Case (i): $\ia\lsub\ib$.}
		
		In this case, we have $\ca{\ia}{\ib}=-1$ and $\qcc{\ia}{\ib}{-}=1$, so that (2) reads
		\begin{align*}
			[\qxp{\ia},\qxm{\ib}]=q\qxz{\ia}{-1}\qxm{\ib\sgpm\ia}\ .
		\end{align*}
		Set $\ic=\ib\sgpm\ia$. Thus, $\ib=\ia\sgpp\ic$ with $\ic\intnext\ia$.
		We have 
		\begin{align*}
			F_{\ib}=-q^{\half}F_{\ic}F_{\ia}+q^{-\half}F_{\ia}F_{\ic}\ ,
			\quad
			[E_{\ia}, F_{\ic}]=0\ ,\\
			F_{\ic}\qxz{\ia}{}=q^{-1}\qxz{\ia}{}F_{\ic}\ ,
			\quad
			F_{\ic}\qxz{\ia}{-1}=q\qxz{\ia}{-1}F_{\ic}\ .
		\end{align*} 
		Therefore,
		\begin{align*}
			[E_{\ia}, F_{\ib}]&=-q^{\half}F_{\ic}[E_{\ia},F_{\ia}]+q^{-\half}[E_{\ia},F_{\ia}]F_{\ic}\\
			&=-q^{\half}F_{\ic}\frac{\qxz{\ia}{}-\qxz{\ia}{-1}}{q-q^{-1}}
			+q^{-\half}\frac{\qxz{\ia}{}-\qxz{\ia}{-1}}{q-q^{-1}}F_{\ic}\\	
			&=\left(-q^{\half}\frac{q^{-1}\qxz{\ia}{}-q\qxz{\ia}{-1}}{q-q^{-1}}
			+q^{-\half}\frac{\qxz{\ia}{}-\qxz{\ia}{-1}}{q-q^{-1}}\right)F_{\ic}\\
			&=\frac{q^{\frac{3}{2}}-q^{-\half}}{q-q^{-1}}\qxz{\ia}{-1}F_{\ic}=
			q^{\half}\qxz{\ia}{-1}F_{\ib\sgpm\ia}\ ,
		\end{align*}
		and we get $[\qxp{\ia}, \qxm{\ib}]=q\qxz{\ia}{-1}\qxm{\ib\sgpm\ia}
		=\qxm{\ib\sgpm\ia}\qxz{\ia}{-1}$.
		
		\item {\it Case (j): $\ib\rsub\ia$.}
		
		In this case, we have $\ca{\ia}{\ib}=-1$ and $\qcc{\ia}{\ib}{+}=0$, so that (2) reads
		\begin{align*}
			[\qxp{\ia},\qxm{\ib}]=-\qxp{\ia\sgpm\ib}\qxz{\ib}{-1}\ .
		\end{align*}
		Set $\ic=\ia\sgpm\ib$. Thus, $\ia=\ib\sgpp\ic$ with $\ib\intnext\ic$.
		We have 
		\begin{align*}
			E_{\ia}=q^{\half}E_{\ib}E_{\ic}-q^{-\half}E_{\ic}E_{\ib}\ ,
			\quad
			[E_{\ic}, F_{\ib}]=0\ ,\\
			\qxz{\ib}{}E_{\ic}=q^{-1}E_{\ic}\qxz{\ib}{}\ ,
			\quad
			\qxz{\ib}{-1}E_{\ic}=qE_{\ic}\qxz{\ib}{-1}\ .
		\end{align*} 
		Therefore,
		\begin{align*}
			[E_{\ia}, F_{\ib}]&=q^{\half}[E_{\ib},F_{\ib}]E_{\ic}-q^{-\half}E_{\ic}[E_{\ib},F_{\ib}]\\
			&=q^{\half}\frac{\qxz{\ib}{}-\qxz{\ib}{-1}}{q-q^{-1}}E_{\ic}
			-q^{-\half}E_{\ic}\frac{\qxz{\ib}{}-\qxz{\ib}{-1}}{q-q^{-1}}\\
			&=E_{\ic}\left(q^{\half}\frac{q^{-1}\qxz{\ib}{}-q\qxz{\ib}{-1}}{q-q^{-1}}
			-q^{-\half}\frac{\qxz{\ib}{}-\qxz{\ib}{-1}}{q-q^{-1}}\right)\\
			&=E_{\ic}\qxz{\ib}{-1}\frac{-q^{\frac{3}{2}}+q^{-\half}}{q-q^{-1}}
			=-q^{\half}E_{\ic}\qxz{\ib}{-1}
		\end{align*}
		and we get $[\qxp{\ia}, \qxm{\ib}]=-\qxp{\ia\sgpm\ib}\qxz{\ib}{-1}
		=-q^{-1}\qxz{\ib}{-1}\qxp{\ia\sgpm\ib}$.
		
		\item {\it Case (k): $\ib\lsub\ia$.}	
		In this case, we have $\ca{\ia}{\ib}=1$ and $\qcc{\ia}{\ib}{+}=-1$, so that (2) reads
		\begin{align*}[\qxp{\ia},\qxm{\ib}]=q^{-1}\qxp{\ia\sgpm\ib}\qxz{\ib}{}\ .\end{align*}
		Set $\ic=\ia\sgpm\ib$. Thus, $\ia=\ic\sgpp\ib$ with $\ic\intnext\ib$.
		We have 
		\begin{align*}
			E_{\ia}=q^{\half}E_{\ic}E_{\ib}-q^{-\half}E_{\ib}E_{\ic}\ ,
			\quad
			[E_{\ic}, F_{\ib}]=0\ ,\\
			\qxz{\ib}{}E_{\ic}=q^{-1}E_{\ic}\qxz{\ib}{}\ ,
			\quad
			\qxz{\ib}{-1}E_{\ic}=qE_{\ic}\qxz{\ib}{-1}\ .
		\end{align*} 
		Therefore,
		\begin{align*}
			[E_{\ia}, F_{\ib}]&=q^{\half}E_{\ic}[E_{\ib},F_{\ib}]-q^{-\half}[E_{\ib},F_{\ib}]E_{\ic}\\
			&=q^{\half}E_{\ic}\frac{\qxz{\ib}{}-\qxz{\ib}{-1}}{q-q^{-1}}
			-q^{-\half}\frac{\qxz{\ib}{}-\qxz{\ib}{-1}}{q-q^{-1}}E_{\ic}\\	
			&=E_{\ic}\left(q^{\half}\frac{\qxz{\ib}{}-\qxz{\ib}{-1}}{q-q^{-1}}
			-q^{-\half}\frac{q^{-1}\qxz{\ib}{}-q\qxz{\ib}{-1}}{q-q^{-1}}\right)\\
			&=E_{\ic}\qxz{\ib}{}\frac{q^{\half}-q^{-\frac{3}{2}}}{q-q^{-1}}=q^{-\half}E_{\ic}\qxz{\ib}{}
		\end{align*}
		and we get $[\qxp{\ia}, \qxm{\ib}]=q^{-1}\qxp{\ia\sgpm\ib}\qxz{\ib}{}=\qxz{\ib}{}\qxp{\ia\sgpm\ib}$.
		
		\item {\it Case (c): $\ia\intcap\ib$.}
		In this case, we have $\qcb{\ia}{\ib}{}=1$ and $\qcb{\ib}{\ia}{}=-1$, so that (2) reads
		\begin{align*}
			[\qxp{\ia},\qxm{\ib}]=-q^{-1}(q-q^{-1})\qxp{(\iM{\ia}{\ib})\sgpm\ib}
			\qxz{\im{\ia\,}{\,\ib}}{}\qxm{(\iM{\ia}{\ib})\sgpm\ia}\ .
		\end{align*}
		Set $c=\im{\ia}{\ib}$, $a=\ia\sgpm c$, $b=\ib\sgpm c$. 
		Thus, $\ia=a\sgpp c$ with $a\intnext c$, $\ib=c\sgpp b$ with $c\intnext b$,
		and $c\rsub \ib$.
		Therefore, 
		\begin{align*}
			E_{\ia}=q^{\half}E_{a}E_{c}-q^{-\half}E_{c}E_{a}\ ,
			\quad
			[E_{a}, F_{\ib}]=0\ ,\\
			[E_{c}, F_{\ib}]=-q^{-\half}\qxz{c}{}F_{b} \ ,
			\quad
			\qxz{c}{}E_a=q^{-1}E_a\qxz{c}{}\ ,
		\end{align*}
		and we have
		\begin{align*}
			[E_{\ia}, F_{\ib}]&=[q^{\half}E_{a}E_{c}-q^{-\half}E_{c}E_{a}, F_{\ib}]\\
			&=q^{\half}E_{a}[E_{c}, F_{\ib}]-q^{-\half}[E_{c}, F_{\ib}]E_{a}\\
			&=q^{\half}E_{a}\left(-q^{-\half}\qxz{c}{}F_{b}\right)
			-q^{-\half}\left(-q^{-\half}\qxz{c}{}F_{b}\right)E_{a}\\
			&=-E_a\qxz{c}{}F_b+q^{-2}E_a\qxz{c}{}F_b=-(1-q^{-2})E_a\qxz{c}{}F_b\ .
		\end{align*}
		Therefore, $[\qxp{\ia},\qxm{\ib}]=
		-q^{-1}(q-q^{-1})\qxp{(\iM{\ia}{\ib})\sgpm\ib}\qxz{\im{\ia\,}{\,\ib}}{}\qxm{(\iM{\ia}{\ib})\sgpm\ia}$.
		
		\item {\it Case (d): $\ib\intcap\ia$.}
		In this case, we have $\qcb{\ia}{\ib}{}=-1$ and $\qcb{\ib}{\ia}{}=1$, so that (2) reads
		\begin{align*}
			[\qxp{\ia},\qxm{\ib}]=q(q-q^{-1})\qxp{(\iM{\ia}{\ib})\sgpm\ib}
			\qxz{\im{\ia\,}{\,\ib}}{-1}\qxm{(\iM{\ia}{\ib})\sgpm\ia}\ .
		\end{align*}
		Set $c=\im{\ia}{\ib}$, $a=\ia\sgpm c$, $b=\ib\sgpm c$. 
		Thus, $\ia=c\sgpp a$ with $c\intnext a$, $\ib=b\sgpp c$ with $b\intnext c$, 
		and $c\lsub\ib$.
		Therefore, 
		\begin{align*}
			E_{\ia}=q^{\half}E_{c}E_{a}-q^{-\half}E_{a}E_{c}\ ,
			\quad
			[E_{a}, F_{\ib}]=0\ ,\\
			[E_{c}, F_{\ib}]=q^{\half}\qxz{c}{-1}F_{b}\ ,
			\quad
			\qxz{c}{-1}E_{a}=qE_{a}\qxz{c}{-1}\ ,
		\end{align*}
		and we have
		\begin{align*}
			[E_{\ia}, F_{\ib}]&=[q^{\half}E_{c}E_{a}-q^{-\half}E_{a}E_{c}, F_{\ib}]\\
			&=q^{\half}[E_{c},F_{\ib}]E_{a}-q^{-\half}E_{a}[E_{c},F_{\ib}]\\
			&=q^{\half}\left(q^{\half}\qxz{c}{-1}F_{b}\right)E_{a}
			-q^{-\half}E_{a}\left(q^{\half}\qxz{c}{-1}F_{b}\right)\\
			&=(q^2-1)E_a\qxz{c}{-1}F_{b}\ .
		\end{align*}
		Therefore, $[\qxp{\ia},\qxm{\ib}]=
		q(q-q^{-1})\qxp{(\iM{\ia}{\ib})\sgpm\ib}\qxz{\im{\ia\,}{\,\ib}}{-1}\qxm{(\iM{\ia}{\ib})\sgpm\ia}$.

		\item {\it Case (f): $\ia<\ib$.}
		Note that, in this case, $\ib\sgpm\ia$, $\ia\sgpm\ib$, 
		$\iM{\ia}{\ib}$, $\im{\ia}{\ib}$ are not defined. Let $b,b''$
		be the two connected components of $\ib\setminus\ia$, so that
		$\ib=b\sgpp b'$ with $b'=\ia\sgpp b''$, $b\intnext b'$, $b\intnext\ia$
		and $\ia\rsub b'$.
		Then, 
		\begin{align*}
			F_{\ib}=-q^{\half}F_{b}F_{b'}+q^{-\half}F_{b'}F_{b}\ ,
			\quad
			[E_{\ia}, F_{b}]=0\ ,\\
			[E_{\ia}, F_{b'}]=-q^{-\half}\qxz{\ia}{}F_{b''}\ ,
			F_b\qxz{\ia}{}=q^{-1}\qxz{\ia}{}F_b\ .
		\end{align*}
		and we get
		\begin{align*}
			[E_{\ia},F_{\ib}]&=[E_{\ia}, -q^{\half}F_{b}F_{b'}+q^{-\half}F_{b'}F_{b}]\\	
			&=-q^{\half}F_{b}[E_{\ia},F_{b'}]+q^{-\half}[E_{\ia},F_{b'}]F_{b}\\
			&=-q^{\half}F_{b}\left(-q^{-\half}\qxz{\ia}{}F_{b''}\right)+
			q^{-\half}\left(-q^{-\half}\qxz{\ia}{}F_{b''}\right)F_{b}\\
			&=q^{-1}\qxz{\ia}{}F_{b}F_{b''}-q^{-1}\qxz{\ia}{}F_{b''}F_{b}=0\ ,
		\end{align*}
		where the last equality follows from \eqref{eq:q-nest-finite}, since $b\perp b''$
		and therefore $F_{b}F_{b''}=F_{b''}F_{b}$. Thus, we get $[\qxp{\ia},\qxm{\ib}]=0$.
		
		\item {\it Case (g): $\ib<\ia$.}
		Note that, in this case, $\ib\sgpm\ia$, $\ia\sgpm\ib$, 
		$\iM{\ia}{\ib}$, $\im{\ia}{\ib}$ are not defined. Let $a,a''$
		be the two connected components of $\ia\setminus\ib$, so that
		$\ia=a\sgpp a'$ with $a'=\ib\sgpp a''$, $a\intnext a'$, $a\intnext\ib$
		and $\ib\rsub a'$.
		Then, 
		\begin{align*}
			E_{\ia}=q^{\half}E_{a}E_{a'}-q^{-\half}E_{a'}E_{a}\ ,
			\quad
			[E_{a}, F_{\ib}]=0\ ,\\
			[E_{a'}, F_{\ib}]=-q^{\half}E_{a''}\qxz{\ib}{-1}\ ,
			\quad
			\qxz{\ib}{-1}E_a=qE_a\qxz{\ib}{-1}\ .
		\end{align*}
		and we get
		\begin{align*}
			[E_{\ia},F_{\ib}]&=[q^{\half}E_{a}E_{a'}-q^{-\half}E_{a'}E_{a}, F_{\ib}]\\
			&=q^{\half}E_{a}[E_{a'},F_{\ib}]-q^{-\half}[E_{a'},F_{\ib}]E_{a}\\
			&=q^{\half}E_{a}\left(-q^{\half}E_{a''}\qxz{\ib}{-1}\right)
			-q^{-\half}\left(-q^{\half}E_{a''}\qxz{\ib}{-1}\right)E_{a}\\
			&=-qE_{a}E_{a''}\qxz{\ib}{-1}+qE_{a''}E_{a}\qxz{\ib}{-1}=0\ ,
		\end{align*}
		where the last equality follows from \eqref{eq:q-nest-finite}, since $a\perp a''$
		and therefore $E_{b}E_{b''}=E_{b''}E_{b}$. Thus, we get $[\qxp{\ia},\qxm{\ib}]=0$.
	\end{itemize}
\end{proof}


\subsection{Quasi--triangular bialgebra structure}\label{s:cont-qg-bia}

We now prove the second main result of the paper. Namely, we show that
the continuum quantum group $\DJ{\g_X}$ is naturally endowed with a topological 
quasi--triangular Hopf algebra structure, quantizing the topological quasi--triangular 
Lie bialgebra $\g_X$. 

More precisely, we prove the following.

\begin{theorem}\label{thm:cont-qg-bia}
	Let $\calQ_X$ be a continuum quiver and $\bfU_q\g_X $ the corresponding 
	continuum quantum group. 
	\begin{enumerate}\itemsep0.2cm
		\item
		The algebra $\DJ{\g_X}$ is a topological Hopf algebra with respect to the maps
		$\Delta\colon\DJ{\g_X}\to\DJ{\g_X}\wh{\ten}\DJ{\g_X}$ and $\varepsilon\colon \DJ{\g_X}\to\hext{\C}$ defined on the generators
		by $\varepsilon(\xz{\ia})\coloneqq 0\coloneqq \varepsilon(\qxpm{\ia})$, $\Delta(\xz{\ia})\coloneqq \xz{\ia}\ten1+1\ten\xz{\ia}$, and
		\begin{align*}
			\Delta(\qxp{\ia}) &\coloneqq \qxp{\ia}\otimes 1+ K_\ia \otimes \qxp{\ia}+\sum_{\ia=\ib\sgpp\ic}\, \ca{\ic}{, \,\ib\sgpp\ic} \, \qcs{\ib}{\ic}{-}\cdot q^{-1}(q-q^{-1})\, \qxp{\ib} K_\ic\otimes \qxp{\ic}\ ,\\[3pt]
			\Delta(\qxm{\ia}) &\coloneqq 1\otimes \qxm{\ia}+ \qxm{\ia}\otimes K_\ia^{-1}-\sum_{\ia=\ib\sgpp\ic}\, \ca{\ic}{, \,\ib\sgpp\ic} \, \qcs{\ib}{\ic}{-}\cdot (q-q^{-1})\, \qxm{\ic} \otimes \qxm{\ib} K_\ic^{-1}\ .
		\end{align*}
		In particular, $\varepsilon(\qxz{\ia}{})=1$ and $\Delta(\qxz{\ia}{})=\qxz{\ia}{}\ten\qxz{\ia}{}$.
		As usual, the antipode is given by the formula
		\begin{align*}
			S\coloneqq \sum_n m^{(n)}\circ(\id-\iota\circ\varepsilon)^{\ten n}\circ\Delta^{(n)}\ ,
		\end{align*}
		where $m^{(n)}$ and $\Delta^{(n)}$ denote the iterated product and coproduct, respectively.
		\item 
		Denote by $\DJ{{\b_X^{\pm}}}$ the Hopf subalgebras generated by 
		$\fun{X}$ and $\qxpm{\ia}$, $\ia\in\intsf(X)$. Then, there exists a unique Hopf pairing 
		\begin{align*}
			\rbf{\cdot}{\cdot}\colon \DJ{{\b}_X^+}\ten(\DJ{{\b}_X^-})^{\scsop{cop}}\to\C(\negthinspace(\hbar)\negthinspace)\ ,
		\end{align*}
		defined on the generators by
		\begin{align*}
			\rbf{1}{1}\coloneqq 1\ , \quad 
			\rbf{\xz{\ia}}{\xz{\ib}} \coloneqq \frac{1}{\hbar}\rbf{\ia}{\ib}\ , \quad
			\rbf{\qxp{\ia}}{\qxm{\ib}}\coloneqq \frac{\delta_{\ia\ib}}{q-q^{-1}}\ ,
		\end{align*}
		and zero otherwise. In particular, $\rbf{\qxz{\ia}{}}{\qxz{\ib}{}}=q^{\rbf{\ia}{\ib}}$.
		\item Through the Hopf pairing $\rbf{\cdot}{\cdot}$, the Hopf algebras 
		$(\DJ{\b_X^{+}}, \DJ{\b_X^{-}})$ form a match pair. Therefore, $\DJ{\g_X}$ 
		can be realized as a quotient of the double cross product Hopf algebra
		$\DJ{{\b}_X^+}\dcs\DJ{\b_X^-}$ obtained by identifying the two copies of $\fun{X}$.
		In particular, $\DJ{\g_X}$ is a topological quasi--triangular Hopf algebra. 
		\item
		The topological quasi--triangular Hopf algebra $\DJ{\g_X}$ is a quantization of the 
		topological quasi--triangular Lie bialgebra $\g_X$. 
	\end{enumerate}
\end{theorem}

The strategy of the proof is essentially identical to that of Theorem~\ref{thm:cont-km-lba}
and consists in showing that the continuum quantum group $\DJ{\g_X}$ can be equivalently
realized by duality. This is obtained by considering the quantum analogue of the techniques 
used earlier, generalizing the construction of Drinfeld--Jimbo quantum groups given by Lusztig
(cf.\ \cite[Chapter~1]{lusztig-book-94}). We will schematically described the proof below,
leaving the details to reader. 

\begin{itemize}[leftmargin=1cm]\itemsep0.5cm
	\item Let $\calH_{\pm}$ be the free associative algebras over $\hext{\C}$ with set of generators $\xzpm{\ia}$ and $\qxpm{\ia}$, $\ia\in\intsf(X)$.
	Then, the assigments $\varepsilon(\xz{\ia}^\pm)\coloneqq 0\coloneqq \varepsilon(\qxpm{\ia})$, $\Delta_\pm(\xz{\ia}^\pm) \coloneqq\xz{\ia}^\pm\ten1+1\ten\xz{\ia}^\pm$, and
	\begin{align*}
		\Delta_+(\qxp{\ia}) &\coloneqq \qxp{\ia}\otimes 1+ K_\ia \otimes \qxp{\ia}+\sum_{\ia=\ib\sgpp\ic}\, \ca{\ic}{, \,\ib\sgpp\ic} \, \qcs{\ib}{\ic}{-}\cdot q^{-1}(q-q^{-1})\, \qxp{\ib} K_\ic\otimes \qxp{\ic}\ ,\\[3pt]
		\Delta_-(\qxm{\ia}) &\coloneqq 1\otimes \qxm{\ia}+ \qxm{\ia}\otimes K_\ia^{-1}-\sum_{\ia=\ib\sgpp\ic}\, \ca{\ic}{, \,\ib\sgpp\ic} \, \qcs{\ib}{\ic}{-}\cdot (q-q^{-1})\, \qxm{\ic} \otimes \qxm{\ib} K_\ic^{-1}\ ,
	\end{align*}
	extend uniquely to two algebra maps $\Delta_{\pm}:\calH_{\pm}\to\calH_{\pm}\wh{\ten}\calH_{\pm}$ and $\varepsilon\colon \calH_{\pm}\to\hext{\C}$, 
	defining on $\calH_{\pm}$ a structure of topological bialgebra. 
	\item There exists a unique pairing of bialgebras $\rbf{\cdot}{\cdot}\colon\calH_+\ten\calH_-\to\C(\negthinspace(\hbar)\negthinspace)$ defined on the generators by 
	\begin{align*}
		\rbf{1}{1}\coloneqq 1\ , \quad 
		\rbf{\xzp{\ia}}{\xzm{\ib}} \coloneqq \frac{1}{\hbar}\rbf{\ia}{\ib}\ , \quad
		\rbf{\qxp{\ia}}{\qxm{\ib}}\coloneqq \frac{\delta_{\ia\ib}}{q-q^{-1}}\ ,
	\end{align*}
	where $q\coloneqq\exp(\hbar/2)$, and zero otherwise. In particular, $\rbf{\qxz{\ia}{}}{\qxz{\ib}{}}=q^{\rbf{\ia}{\ib}}$, where $\qxz{\ia}{}\coloneqq\exp{\hbar/2\cdot\xz{\ia}}$.
	\item Let $\calI_{\pm}$ be the ideal generated in $\calH_{\pm}$ by the elements
	\begin{align*}
		\xzpm{\ia\sgpp\ib}-\drc{\ia\sgpp\ib}\left(\xzpm{\ia}+\xzpm{\ib}\right)\ , \quad [\xzpm{\ia},\xzpm{\ib}]\ , \quad [\xzpm{\ia},\qxpm{\ib}]-\pm\rbf{\ia}{\ib}\qxpm{\ib}
	\end{align*}
	for any $\ia,\ib\in\intsf(X)$, and
	\begin{align*}
		\qxpm{\ia}\qxpm{\ib}-q^{\qcr{\ia}{\ib}{}}\cdot&\qxpm{\ib}\qxpm{\ia} \\&
		\pm\qcb{\ia}{\ib}{}\cdot q^{\qcs{\ia}{\ib}{\pm}}\cdot\qxpm{\ia\sgpp\ib}
		-\qcb{\ia}{\ib}{}\cdot(q-q^{-1})\cdot\qxpm{\iM{\ia}{\ib}}\qxpm{\im{\ia\, }{\,\ib}}
	\end{align*}
	for any $(\ia,\ib)\in\serre{X}$. Then, $\calI_{\pm}$ is a coideal and it is orthogonal to $\calH_{\mp}$.
	\item Set $\calB_{\pm}\coloneqq\calH_{\pm}/\calI_{\pm}$. Then, $(\calB_{+},\calB_-)$ form a matched pair of topological bialgebras. 
	Moreover, the quantum double relation (cf.\ Definition~\ref{def:cont-qg}-(2)) holds in the double cross product bialgebra $\D=\calB_+\dcs\calB_-/\sim$,
	where the quotient is obtained by identifying the two copies of the commutative subalgebra generated by the elements $\xzpm{\ia}$, $\ia\in\intsf(X)$.  
	In particular, there is a canonical algebra isomorphism $\DJ{\g_X}\simeq\D$.
	\item Finally one observes that, for any irreducible set $\calJ$, the map $\DJ{\g_{\calJ}^{\scsop{BKM}}}\to\DJ{\g_X}\simeq\D$ from Section~\ref{ss:q-colimit} 
	preserves the pairing. In particular, this implies that the pairing on $\D$, and therefore on $\DJ{\g_X}$ is non--degenerate. The result follows.
\end{itemize}

Moreover we get the following.

\begin{corollary}
	The morphism $\colim_{\calJ}\, \bfU_q\g_{\calJ}^{\scsop{BKM}}\to \bfU_q\g(X)$ from Corollary~\ref{cor:q-bkm-sgp} is an
	algebra isomorphism.
\end{corollary}


\end{document}